\documentclass[letterpaper,reqno,10pt,twoside]{amsart}
\usepackage{amsmath,amsthm,amsfonts,amssymb,euscript,mathrsfs,graphics,color,latexsym,marginnote}
\usepackage{stmaryrd}
\usepackage[dvips]{graphicx}
\usepackage[left=1 in, right=1 in,top=1 in, bottom=1 in]{geometry}
\usepackage{hyperref}
\usepackage[toc,page]{appendix}
\usepackage{relsize}
\usepackage[shortlabels]{enumitem}
\usepackage[all]{xy}
\usepackage{float}

\usepackage{xargs}
\usepackage[dvipsnames]{xcolor}

\usepackage{hyperref}
\hypersetup{colorlinks=true, pdfstartview=FitV,linkcolor=blue!70!black,citecolor=red!70!black, urlcolor=green!60!black}
\definecolor{labelkey}{rgb}{0.6,0,0}

\usepackage[colorinlistoftodos,prependcaption,textsize=small]{todonotes}
\newcommandx{\change}[2][1=]{\todo[#1]{#2}}
\newcommandx{\unsure}[2][1=]{\todo[linecolor=red,backgroundcolor=red!25,bordercolor=red,#1]{#2}}
\newcommandx{\rmk}[2][1=]{\todo[linecolor=blue,backgroundcolor=blue!25,bordercolor=blue,#1]{#2}}
\newcommandx{\info}[2][1=]{\todo[linecolor=OliveGreen,backgroundcolor=OliveGreen!25,bordercolor=OliveGreen,#1]{#2}}
\newcommandx{\improvement}[2][1=]{\todo[linecolor=Plum,backgroundcolor=Plum!25,bordercolor=Plum,#1]{#2}}
\newcommandx{\thiswillnotshow}[2][1=]{\todo[disable,#1]{#2}}

\makeatletter
\renewcommand \theequation {%
\ifnum \c@section>\z@ \@arabic\c@section.%
\fi\@arabic\c@equation} \@addtoreset{equation}{section}
\@namedef{subjclassname@2020}{2020 Mathematics Subject Classification}
\makeatother

\newtheorem{theorem}{Theorem}[section]

\newtheorem{lemma}[theorem]{Lemma}

\theoremstyle{definition}

\theoremstyle{remark}

\def\XXint#1#2#3{{\setbox0=\hbox{$#1{#2#3}{\int}$ }
\vcenter{\hbox{$#2#3$ }}\kern-.6\wd0}}

\title{The steady state of gravity-capillary problem with inclined walls}
\author{Xiaoding Yang}

\begin{document}
\begin{abstract}
The gravity-capillary problem with inclined
walls is a problem that describes an open fluid flowing over an angled wall. It has broad applications in science and engineering. In this paper, we study the steady states of the two-dimensional inclined-wall problem. The steady-state configurations are characterized as solutions of the Euler–Lagrange equation associated with a prescribed energy functional, subject to a fixed contact-angle boundary condition. By parameterizing the free surface using an appropriately chosen maximal point, we construct solutions to this Euler-Lagrange equation via a shooting method, with the fluid volume serving as the shooting parameter. The construction is valid for arbitrary contact angles and arbitrary inclined angles of the walls. 
\end{abstract}

\maketitle

\pagestyle{myheadings} \thispagestyle{plain} 
\setcounter{tocdepth}{1}
\section{Introduction}

The study on Fluid is a fundamental branch of physics and engineering that deals with the behavior of fluids in motion and at rest. In recent years, an increasing number of math studies have focused on this field. One particular area of interest is fluid that flows over inclined surfaces, which is so-called fluid inclined problems. These problems arise in various natural and industrial applications such as the movement of water on sloped terrains, oil transport through pipelines, and the design of spillways and drainage systems\cite{app}.

When a fluid interacts with an inclined surface, several forces influence its motion including gravity, pressure, and surface tension. It is crucial to understand those effects if we want to predict the behavior of the fluid, optimize engineering designs, and improve efficiency in fluid-based systems. Even in the steady state—where the fluid velocity vanishes everywhere—the balance of forces acting on the fluid remains a rich and important subject of study. Analysis of such equilibria provides valuable insight into the structure and behavior of the corresponding dynamical flows. 

Motivated by these considerations, mathematicians and engineers have devoted substantial effort to the analysis of equilibrium and steady-state configurations in gravity–capillary problems. Despite this progress, many challenges remain, both theoretical and computational, leaving significant room for further investigation.

 From the perspective of practical applications, the study on the steady state of Gravity-Capillary problem typically falls into two distinct settings. The first concerns the motion of a fluid flowing down an inclined surface—for instance, a thin film of fluid sliding down a tilted plane under the influence of gravity. The second setting involves the motion of an open fluid flowing over an angled wall, as illustrated in Figure \ref{figure1}. This scenario can also be interpreted as a fluid mechanics problem within a triangular container, such as the cross-section of water in a channel. Both settings are common in everyday life and are important to discuss. In this paper, we focus exclusively on the second setting. 

 \subsection{Formulation} Based on the discussion in the Background part, the purpose of this paper is to study the static viscous incompressible fluid occupied in an open-top vessel in two dimensions. The vessel is modeled as an open, connected, bounded triangular subset $\Gamma\subset \mathbb{R}^{2}$ that obeys the following assumptions (See Figure \ref{figure1}). First, we posit that two inclined angles $\theta_1$ and $\theta_2$ that are defined as angles between the wall of the vessel and the horizontal plane are both angles between $0$ and $\pi$. Second, we assume that $\theta_1+\theta_2<\pi$. Finally, it is assumed that the fluid occupies a subset $\Omega$ of the vessel $\Gamma$, resulting in a free boundary where the fluid meets the air above the vessel.

In this article, we focus on the steady-state configurations, which are equivalent to determining a static boundary surface. In the scenario depicted in Figure \ref{figure1}, the surface can be represented as a graph of
$x$ in Cartesian coordinates, denoted by $v(x)$. However, such a representation is not always possible. For instance,  the boundary curve shown in Figure \ref{figure3} cannot be represented as a graph of $x$. This observation intrigues us to describe the surface using polar coordinates. Specifically, we choose the intersection point of two vessel walls (denoted by point $O$ in Figure \ref{figure1}) as the original point, and represent the surface by a radial function $\rho(\theta )$. With these geometric settings established, we proceed to derive the governing equation satisfied by the boundary surface.

\begin{figure}
\centering
\label{figure1}
\includegraphics[width=0.9\linewidth]{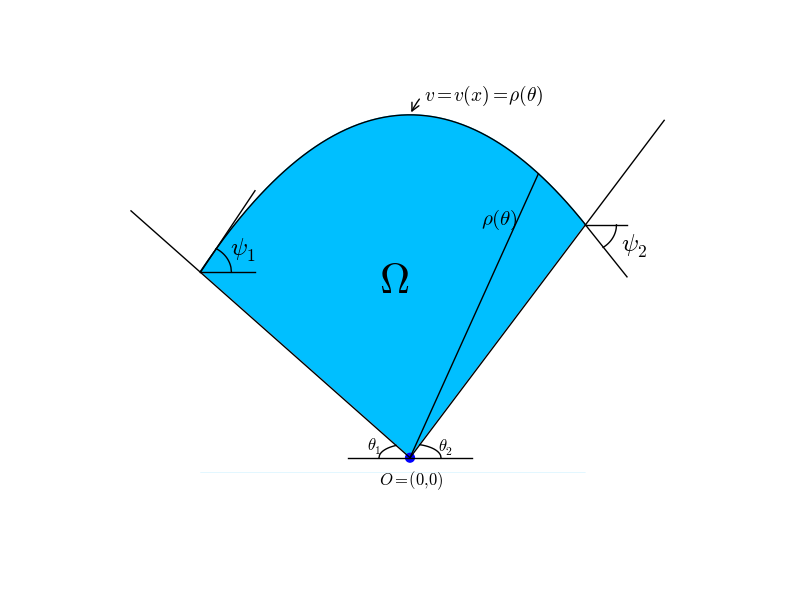}
\caption{}
\end{figure}

\newpage

  Given any nonnegative constant $V$, we formulate the energy functional in polar coordinates subject to the constraint that the total fluid volume is conserved and equals $V$. The functional is given by the following expression.

\begin{equation}{\label{equ:1.2.1}}
        E(\rho):=\frac{1}{3}g\int_{\theta_2}^{\pi-\theta_1}\rho^{3}\sin \theta d\theta+\int_{\theta_2}^{\pi-\theta_1}\sigma \sqrt{\rho^{2}+\rho'^{2}}d\theta-[\![\gamma]\!](\rho(\theta_1)+\rho(\theta_2)),
\end{equation}
\noindent subject to the volume constraint:
\begin{equation}{\label{equ:1.2.2}}
    \mathcal{V}(\rho):=\frac{1}{2}\int_{\theta_2}^{\pi-\theta_1} \rho^{2}(\theta)d\theta=V
\end{equation}

\noindent for any $\rho(\theta)\in C^{1}(\theta_{2},\pi-\theta_{1})$. In equation \eqref{equ:1.2.2}, $g$ is the gravitational constant. $\sigma>0$ is the coefficient of surface tension and $[\![\gamma]\!]:=\gamma_{sv}-\gamma_{sf}$ for $\gamma_{sv},\gamma_{sf}\in\mathbb{R}$, where $\gamma_{sv},\gamma_{sf}$ are measures of the free-energy per unit length with respect to the solid-vapor and solid-fluid intersection. The first term of the functional $E$ defined in \eqref{equ:1.2.1} represents the gravity energy and the second term is the surface tension energy. The third term accounts for the free-boundary energy arising from the interaction between the fluid and the vessel walls. The second equation expresses the conservation of the total volume. To ensure the physical relevance of the model, we assume that
\begin{align}
    \begin{aligned}
        -1\leq\frac{[\![\gamma]\!]}{\sigma}\leq 1.
    \end{aligned}
\end{align}

Applying the variation method to this energy functional $E$, we obtain the associated Euler–Lagrange equations

\begin{equation}{\label{equ:1.2.3}}
\begin{cases}
    g\rho^{2}\sin\theta +\sigma \frac{\rho}{\sqrt{\rho^{2}+\rho'^{2}}}-\sigma\partial_{\theta}(\frac{\rho'}{\sqrt{\rho^{2}+\rho'^{2}}})-P_{0}\rho=0\\
    \frac{\rho'}{\sqrt{\rho^{2}+\rho'^{2}}}(\theta_2)=\frac{[\![\gamma]\!]}{\sigma}\\
     \frac{\rho'}{\sqrt{\rho^{2}+\rho'^{2}}}(\pi-\theta_1)=-\frac{[\![\gamma]\!]}{\sigma}
\end{cases}
\end{equation}

 \noindent where $P_{0}$ is a Lagrange multiplier depending only on the prescribed volume $V$. Then our purpose of finding the steady state is equivalent to constructing a solution to the Euler-Lagrange equation \eqref{equ:1.2.3}. Our approach proceeds as follows. We select a distinguished point on the free surface and use equation \eqref{equ:1.2.3} to construct the corresponding solution curve, parametrized by the coordinate of this point, which we refer to as the special parameter. The total volume $\mathcal{V}(\rho)$ and the pressure $P_{0}$ are then expressed in terms of this parameter. Finally, we determine its value by a shooting method, enforcing the volume constraint $\mathcal{V}(\rho)=V$, which yields the desired equilibrium free-surface profile.

We note that the only unknown parameter in the equation system \eqref{equ:1.2.3} is $P_{0}$, whose determination constitutes the main difficulty of this paper. Even though it is more convenient to  represent the energy $E$ and volume $V$ in polar coordinates, eliminating the unknown parameter $P_{0}$ within this framework proves to be quite challenging. Noting that the surface curve can also be represented as $(x(\theta),y(\theta))$,  we introduce the following shift in Cartesian coordinates to eliminate $P_{0}$

\begin{equation}{\label{equ:1.2.10}}
    (x,y)\rightarrow (x,y-\frac{P_{0}}{g}).
\end{equation}

\noindent We shall see how this transformation simplifies the analysis in the subsequent sections.  

After this shift, the representation of surface curve in polar coordinates is unclear if the pole remains fixed at $(0,0)$. Therefore, to address this difficulty and to better analyze the system \eqref{equ:1.2.3}, we introduce a new intrinsic variable that remains invariant under the above transformation. We introduce the angle $\psi$, defined by
    \begin{equation}{\label{def:slope}}
    \begin{cases}
    \sin\psi(\theta)=\frac{\sin\theta \rho'(\theta)+\rho(\theta)\cos\theta}{\sqrt{\rho^{2}(\theta)+\rho'^{2}(\theta)}}\\
    \cos\psi(\theta)=\frac{\rho(\theta)\sin\theta-\rho'(\theta)\cos\theta}{\sqrt{\rho^{2}(\theta)+\rho'^{2}(\theta)}},
    \end{cases}
    \end{equation}

   \noindent which represents the slope angle of the free surface in Cartesian coordinates. 
   
   Although the expressions in \eqref{def:slope} may be complicated, the geometric meaning of $\psi$ is transparent and is illustrated in Figure~\ref{figure2}. In Cartesian coordinates, $\psi$ is an acute angle and can be expressed more simply as
\begin{equation}\label{equ:psi}
\tan\psi = \partial_x v(x),
\end{equation}
   where the definition of $v(x)$ is given in Figure \ref{figure1}. Given the representation of $\rho(\theta)$, equation \eqref{def:slope} shows how the angle $\psi$ is determined from 
$\theta$. Compared with $\theta$, $\psi$ provides a more natural parametrization of the surface curve $\rho$, considering the structure of equation \eqref{equ:1.2.3}. We will show how we reformulate the equation \eqref{equ:1.2.3} in terms of this new variable $\psi$ in next subsection.
   
     \subsection{The structure of this paper} We introduce the following definitions for contact angles $\psi_{1}$ and $\psi_{2}$ 
    
\begin{figure}
\centering
\includegraphics[width=0.9\linewidth]{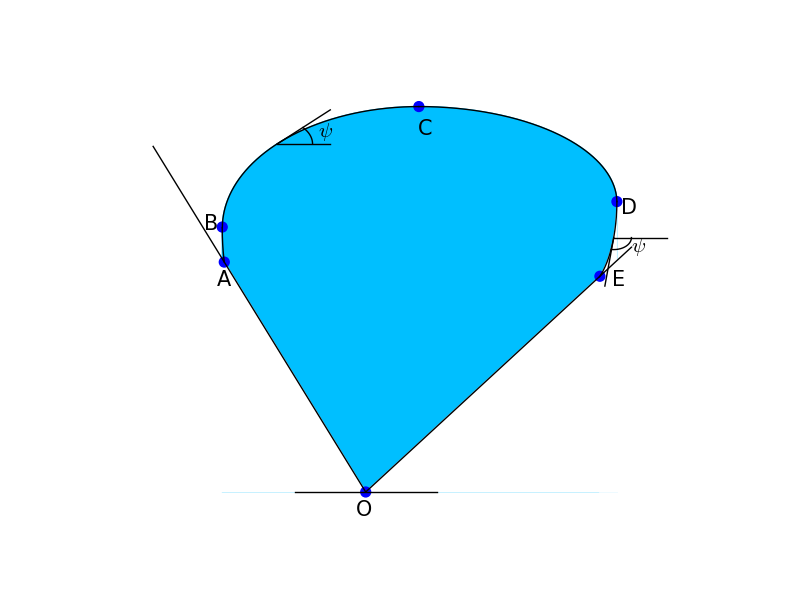}
\caption{\label{figure2}}
\end{figure}

\begin{align}
     \sin\psi_1&:=\sin\psi(\pi-\theta_1)=\sin(\gamma+\frac{\pi}{2}-\theta_1)=(-\frac{[\![\gamma]\!]}{\sigma})\sin\theta_1+\sqrt{1-\frac{[\![\gamma]\!]^2}{\sigma^{2}}}\cos\theta_1 \label{equ:1.2.8}\\
    \sin\psi_2&:=\sin\psi(\theta_2)=\sin(-\gamma-\frac{\pi}{2}+\theta_2)=(\frac{[\![\gamma]\!]}{\sigma})\sin\theta_2+\sqrt{1-\frac{[\![\gamma]\!]^{2}}{\sigma^{2}}}\cos\theta_2 \label{equ:1.2.9}
\end{align}

\noindent where $\psi$ is defined by \eqref{def:slope}. The constant $\gamma\in [-\frac{\pi}{2},\frac{\pi}{2}]$ is determined by

\begin{equation}{\label{equ:1.2.110}}
        \sin\gamma=-\frac{[\![\gamma]\!]}{\sigma}.
\end{equation}

\noindent Then we divide our problem into two cases according to the sign of $\psi_1$ and $\psi_2$. In analyzing each case, we begin by assuming the existence of the solution and deriving some geometric properties that any solution function to \eqref{equ:1.2.3} must satisfy. These properties motivate the introduction of a set of auxiliary parameters. Using these parameters, we combine the first equation of \eqref{equ:1.2.3} with these geometric relations to construct the free-boundary curve. Finally, the remaining boundary conditions—the second and third equations in \eqref{equ:1.2.3}—together with the volume constraint \eqref{equ:1.2.2}, are employed to determine the auxiliary parameters and thereby obtain the desired solution.

\subsection{ \textbf{Case One: $\psi_1$ and $\psi_2$ have opposite signs}}

In Section~2, we consider the case in which $\psi_{1}>0$ and $\psi_{2}<0$, and establish the following result for the angle $\psi$ defined in \eqref{def:slope}.

\begin{theorem}
    Suppose that the system \eqref{equ:1.2.3} admits a solution. Then the angle $\psi$ provides a global parametrization of the free boundary curve. Equivalently, the relation between $\psi$ and $\theta$ given by \eqref{def:slope} is one-to-one.

Let $(x(\theta),y(\theta))=(\rho(\theta)\cos\theta,\rho(\theta)\sin\theta)$ denote the corresponding solution of \eqref{equ:1.2.3} (see Figure~\ref{figure3}). Then, using $\psi$ as the parametrization variable for the boundary curve, we derive the system of equations \eqref{equ:1.2.600}, which is equivalent to the original system \eqref{equ:1.2.3}. This equation can be written as follows

     \begin{equation}{\label{equ:1.2.600}}
    \begin{cases}
        \frac{dx}{d\psi}=\sigma\frac{\cos\psi}{gy-P_{0}}\\
    \frac{dy}{d\psi}=\frac{dy}{dx}\frac{dx}{d\psi}=\sigma\tan\psi\frac{\cos\psi}{gy-P_{0}}=\sigma\frac{\sin\psi}{gy-P_{0}}\\
    \sin\psi(\pi-\theta_1)=\sin\psi_1=(-\frac{[\![\gamma]\!]}{\sigma})\sin\theta_1+\sqrt{1-\frac{[\![\gamma]\!]^{2}}{\sigma^{2}}}\cos\theta_1\\
    \sin\psi(\theta_2)=\sin\psi_{2}=(\frac{[\![\gamma]\!]}{\sigma})\sin\theta_2+\sqrt{1-\frac{[\![\gamma]\!]^{2}}{\sigma^{2}}}\cos\theta_2
    \end{cases}
    \end{equation}
 \end{theorem}

 With this reformulation, our problem reduces to finding the solution function to \eqref{equ:1.2.600} such that the area of the region $\Omega$, enclosed by the free surface and the two inclined walls, equals the prescribed volume $V$. To derive explicit expressions for both the solution curve and the total volume of $\Omega$, we first consider the system \eqref{equ:1.2.600} without the boundary conditions, namely,

    \begin{equation}{\label{equ:1.2.6}}
    \begin{cases}
        \frac{dx}{d\psi}=\sigma\frac{\cos\psi}{gy-P_{0}}\\
    \frac{dy}{d\psi}=\sigma\frac{\sin\psi}{gy-P_{0}}
    \end{cases}
    \end{equation}
    We then apply the shift introduced in \eqref{equ:1.2.10} to eliminate $P_{0}$. This transformation reduces \eqref{equ:1.2.6} to the following system

    \begin{equation}{\label{equ:1.2.7}}
    \begin{cases}
        \frac{dx}{d\psi}=\sigma\frac{\cos\psi}{gy}\\
    \frac{dy}{d\psi}=\sigma\frac{\sin\psi}{gy}.
    \end{cases}
    \end{equation}

    \begin{figure}
    \centering
    \includegraphics[width=0.9\linewidth=2]{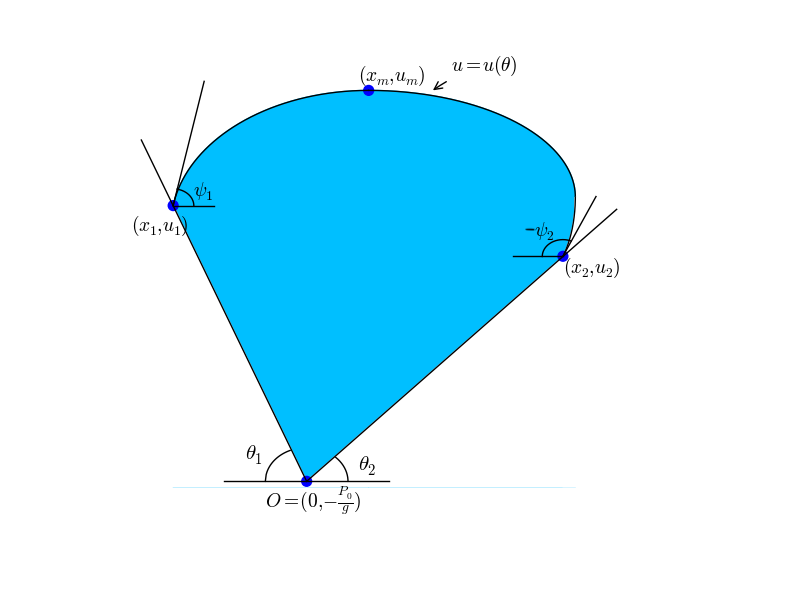}
    \caption{\label{figure3}}
    \end{figure}
    
    \noindent  To construct a solution function of \eqref{equ:1.2.7}, appropriate initial data is required. Observe that for any solution function of \eqref{equ:1.2.600}( assuming existence without imposing the volume constraint), the assumption that $\psi_{1}$ and $\psi_{2}$ have opposite signs implies that $\psi$ ranges from $\psi_{2}<0$ to $\psi_{1}>0$. Therefore, $\psi$ attains value zero at a unique point on the curve, which we denote by $(x_{m},u_{m})$. This point corresponds to the maximum height of the free surface and is characterized by

    \begin{align}
        u_{m}&=\sup_{\psi_2\leq \psi\leq \psi_{1}} y(\psi)=y(0) \label{equ:1.3.1}\\
         x_{m}&=x(0) \label{equ:1.3.2}.
    \end{align}
    
     This geometric observation suggests a natural choice of initial data for constructing solutions of \eqref{equ:1.2.7}. By the existence and uniqueness theorem for ordinary differential equations, once the values of $x_{m}$ and $u_{m}$ are specified, there exists a unique solution curve $(x(\psi),y(\psi))$ of \eqref{equ:1.2.7} satisfying the initial conditions \eqref{equ:1.3.1}–\eqref{equ:1.3.2}. This curve intersects the two inclined walls at points $(x_{1},y_{1})$ and $(x_{2},y_{2})$, respectively. 
     
     With this construction, the remaining task is to determine suitable  values of $u_{m},x_{m}$ such that the resulting curve satisfies the boundary conditions in \eqref{equ:1.2.600}, namely that it meets the two walls with prescribed slope angles $\psi_{1}$ and $\psi_{2}$, and that the area $\Omega$ enclosed by the curve and the walls has the prescribed volume $V$. We denote the volume of this enclosed area by $\mathscr{V}$.
     
     Finally, note that after applying the shift \eqref{equ:1.2.10}, the intersection point $O$ of the two walls is located at $(0,-P_{0}/g)$ (see Figure~\ref{figure3}). We now state the following theorem expressing both $\mathscr{V}$ and $P_{0}$ as functions of the parameter $u_{m}$.

    \begin{figure}
    \centering
    \includegraphics[width=0.9\linewidth=2]{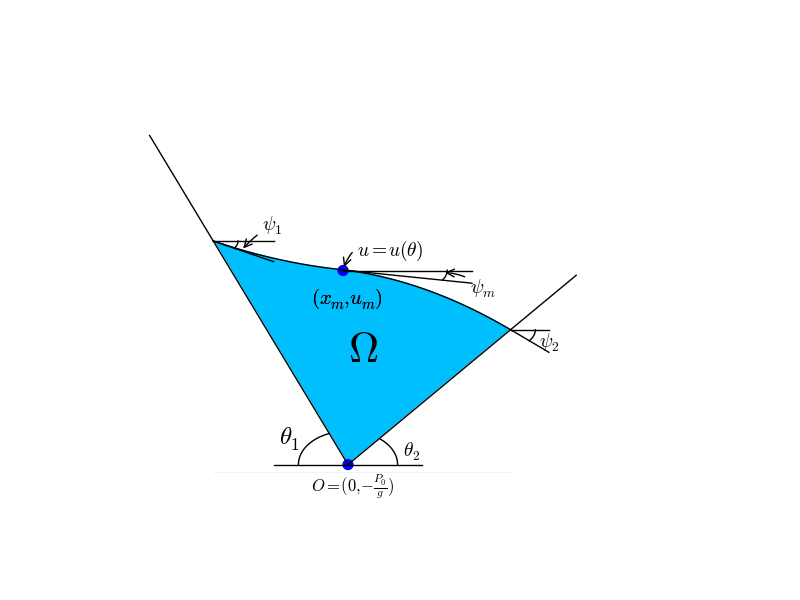}
    \caption{\label{figure4}}
    \end{figure}
    
    \begin{theorem}{\label{thm:V}}
    For any given $x_{m}\in \mathbb{R}$ and $u_{m}\in \mathbb{R}$, let $(x(\psi),y(\psi))$ denote the solution of \eqref{equ:1.2.7} satisfying relations \eqref{equ:1.3.1} and \eqref{equ:1.3.2}. If this curve intersects with two boundary walls with slope angles $\psi_{1}$ and $\psi_{2}$, corresponding to the boundary condition of \eqref{equ:1.2.600}, then the volume functional $\mathscr{V}$, as well as the $y$-coordinate of the intersection point $O$, given by $-\frac{P_{0}}{g}$, can be represented as functions of $u_m$. Moreover, the following two asymptotic properties of $\mathscr{V}$ hold

    \begin{align}
        \lim_{u_{m}\rightarrow 0}\mathscr{V}(u_m)&=+\infty\\
        \lim_{u_m\rightarrow -\infty} \mathscr{V}(u_m)&=0.
    \end{align}

    \noindent Consequently, for any prescribed $V>0$, there exists at least one $u_m$ and a corresponding solution to equation system \eqref{equ:1.2.7} such that $\mathscr{V}(u_m)=V$.  Furthermore, $\mathscr{V}(u_{m})$ is a strictly increasing function of $u_{m}$ when $[\![\gamma]\!]<0$. As a result, for any given $V$, the choice of $u_{m}$ and solution to \eqref{equ:1.2.7} satisfying the volume constraint $\mathscr{V}(u_{m})=V$ are unique when $[\![\gamma]\!]<0$.
    \end{theorem}

    \textbf{Note}: By Theorem \ref{thm:V}, there exists a solution $(x(\psi),y(\psi))$ of \eqref{equ:1.2.7} and $P_{0}$ for any prescribed constant 
    $V$. Therefore, the vertically shifted curve $(x(\psi),y(\psi)+\frac{P_{0}}{g})$ provides a solution of the original system \eqref{equ:1.2.600}. Moreover, this curve can be represented in polar coordinates as a graph $\rho(\theta)$, for which the volume constraint is satisfied, namely $\mathcal{V}(\rho)=\mathscr{V}(u_{m})=V$. The detailed construction and verification of these claims are presented in Section~2.
    
   \subsection{Case two: $\psi_1$ and $\psi_2$ have the same sign} 

    In this case we show in Section 3 that the surface curve can be represented as a graph of $x$. See Figure \ref{figure4}. Accordingly, we reformulate the problem in Cartesian coordinates and denote the surface function by $v(x)$. The resulting system of equations, which is equivalent to \eqref{equ:1.2.3}, is given below:

    \begin{equation}{\label{equ:1.3.300}}
        \begin{cases}
        \sigma\partial_{x}(\frac{\partial_{x}v}{\sqrt{1+\vert \partial_{x}u\vert^{2}}})=gv-P_{0}\\
        \frac{\partial_{x}v}{\sqrt{1+\vert \partial_{x}v\vert^{2}}}(x_1)=\sin\psi_1\\
         \frac{\partial_{x}v}{\sqrt{1+\vert \partial_{x}v\vert^{2}}}(x_2)=\sin\psi_2.
        \end{cases} 
    \end{equation}

     \noindent Here, $\psi_1$ and $\psi_2$ are given by equations \eqref{equ:1.2.8} and \eqref{equ:1.2.9}. Then we apply the shift \eqref{def:slope} to eliminate the constant $P_{0}$, which transforms the system \eqref{equ:1.3.300} into

\begin{equation}{\label{equ:1.3.3}}
        \begin{cases}
        \sigma\partial_{x}(\frac{\partial_{x}u}{\sqrt{1+\vert \partial_{x}u\vert^{2}}})=gu\\
        \frac{\partial_{x}u}{\sqrt{1+\vert \partial_{x}u\vert^{2}}}(x_1)=\sin\psi_1\\
         \frac{\partial_{x}u}{\sqrt{1+\vert \partial_{x}u\vert^{2}}}(x_2)=\sin\psi_2,
        \end{cases} 
    \end{equation}

    \noindent where we use $u(x)$ to denote the surface in shifted coordinates. We focus on analyzing the system \eqref{equ:1.3.3} in Chapter 3.
    
      As in the previous case, we begin by examining the system \eqref{equ:1.3.3} without imposing the boundary conditions, and then introduce the new variable $\psi$, defined by \eqref{equ:psi}. This yields

      \begin{equation}{\label{equ:1.3.4}}
          \sigma\partial_{x}(\sin\psi)=\sigma\partial_{x}(\frac{\partial_{x}u}{\sqrt{1+\vert \partial_{x}u\vert^{2}}})=gu.
      \end{equation}

      \noindent However, unlike Case 1, the assumption that $\psi_{1}$ and $\psi_{2}$ have the same sign implies that $\psi$ may not attain value zero along the boundary curve corresponding to any solution of \eqref{equ:1.3.300}. Therefore, we need to choose a different reference point to serve as the initial condition in constructing the solution to \eqref{equ:1.3.3}. 
      
      For any solution function of \eqref{equ:1.3.300}, we show in Section 3 that the slope angle $\psi$ attains its unique maximum value at a point $(x_{m},u_{m})$ in the interior of the interval. Suppose that the boundary curve meets two inclined walls at $x=x_{1}$ and $x=x_{2}$. We characterize this distinguished point by

    \begin{equation}{\label{equ:1.3.5}}
    \begin{cases}
        \psi_{m}=\sup_{x_1\leq x\leq x_1} \psi(x)\\
         x_m=\{x\in(x_{1},x_{2})|\psi(x)=\psi_{m}\}\\
        u_m=u(x_m).
    \end{cases}
    \end{equation}

    \noindent This geometric property provides a natural choice of initial data for constructing solutions of \eqref{equ:1.3.4}. For any prescribed values of $x_{m}$, $u_{m}$, and $\psi_{m}\in[\max(\psi_{1},\psi_{2}),0)$, we take the point and angle defined by \eqref{equ:1.3.5} as initial conditions. Standard ODE theory then guarantees the existence of a unique solution $(x,u(x))$ of \eqref{equ:1.3.4} satisfying these initial conditions.

As in Case~1, we denote by $\mathscr{V}$ the volume of the region enclosed by the free surface and the two inclined walls. The $y$-coordinate of the intersection point $O$ is given by $-P_{0}/g$. We then establish the following theorem concerning the dependence of $\mathscr{V}$ and $P_{0}$, whose proof is deferred to Section~3.

    \begin{theorem}{\label{thm:main_2}}

        For any $\psi_{m}\in[\max(\psi_{1},\psi_{2}),0)$, we have two distinct cases. 

        \begin{itemize}
            \item Case (i): When $\psi_m=\psi_{1}$ or $\psi_{m}=\psi_2$ (without loss of generality, assume $\psi_m=\psi_{2}$).
            
            In this case, we choose $u_{m},x_{m}$ to be arbitrary nonpositive real numbers. Then there exits a solution $u(x)$ of \eqref{equ:1.3.4} satisfying \eqref{equ:1.3.5}. If the corresponding curve $(x,u(x))$ intersects two inclined walls with slope angles $\psi_{1}$ and $\psi_{2}$, then $x_{m}$ can be expressed as a function of $u_{m}$. Consequently, the quantities $\mathscr{V}$, $P_{0}$, $x_1$ and $x_2$ can all be represented as functions of $u_m$. Moreover,
            \begin{equation*}
                \lim_{u_{m}\rightarrow -\infty} \mathscr{V}(u_m)=0,
            \end{equation*}

            \noindent and $\mathscr{V}(u_{m})$ is bounded from above. Define, 

            \begin{equation*}
                V_1:=\max_{u_m\in(-\infty,0)} \mathscr{V}(u_m).
            \end{equation*}
            
            \noindent Therefore, for any $V$ satisfying $0\le V\le V_{1}$, there exists a value $u_{m}$ such that $\mathscr{V}(u_{m})=V$, together with a corresponding solution of the system \eqref{equ:1.3.300}.
            
            \item Case (ii): $\psi_{m}\in (\max(\psi_{1},\psi_{2}),0)$
            
            In this situation, we have $u_m=0$. For any $\psi_{m}\in(\max(\psi_{1},\psi_{2}),0)$ and any $x_{m}\in\mathbb{R}$, there exists a solution $u(x)$ of \eqref{equ:1.3.4} satisfying \eqref{equ:1.3.5}. If the curve $(x,u(x))$ intersects the two inclined walls with slope angles $\psi_{1}$ and $\psi_{2}$, then $x_{m}$ can be expressed as a function of $\psi_{m}$. In this case, the quantities $\mathscr{V}$, $P_{0}$, $x_{1}$, and $x_{2}$ can all be expressed as functions of $\psi_{m}$. Furthermore,
            \begin{equation*}
            \lim_{\psi_m\rightarrow 0} \mathscr{V}(\psi_{m})=+\infty,
            \end{equation*}

            \noindent and $\mathscr{V}(\psi_{m})$ is bounded from below. Define

            \begin{equation*}
                V_{m}:=\min_{\psi_m \in(\max(\psi_1,\psi_2),0)}\mathscr{V}(\psi_m).
            \end{equation*}

            \noindent Consequently, for any $V\ge V_{m}$, there exist a $\psi_m$ such that $\mathscr{V}(\psi_m)=V$, and a corresponding solution function of system \eqref{equ:1.3.300}.
        \end{itemize}
    \end{theorem}
     
     Theorem \ref{thm:main_2} leads to the following results.
     
     \begin{theorem}
         $V_{1}$ and $V_{m}$ are defined as in Theorem 1.3. For any prescribed $\theta_{1}$, $\theta_{2}$, interfacial parameter $[\![\gamma]\!]$, and surface tension coefficient $\sigma$, the following inequality holds
         \begin{align}
         V_{1}\geq V_{m}
         \end{align}
         Consequently, for any prescribed volume $V \ge 0$, there exists at least one solution of the system \eqref{equ:1.3.300}. Moreover, when $V_{1} > V_{m}$, the solution is not unique.
         
         Furthermore, the free surface determined by such a solution can be represented in polar coordinates as a graph $\rho(\theta)$, and it satisfies the volume constraint
\begin{equation}
\mathcal{V}(\rho)=\mathscr{V}(u_{m})=V.
\end{equation}
     \end{theorem}

    \noindent \textbf{1.4 previous work}:
    The calculus of variations is one of the oldest branches of mathematical analysis, concerned with finding functions that minimize or maximize certain functionals. The subject originated with the brachistochrone problem, posed by Johann Bernoulli in 1696. He wanted to find the curve along which a particle descends under gravity in the shortest possible time. A systematic foundation for the subject was established by Leonhard Euler \cite{Euler} and Joseph-Louis Lagrange \cite{Lagrange} in the mid-18th century. They introduced what is now known as the Euler–Lagrange equation, providing the necessary condition for an extremals of variation problems. 
    
   When it comes to the study of the functional defined by \eqref{equ:1.2.1} and \eqref{equ:1.2.2}, it begins with analyzing the Euler-Lagrange equation, which was derived by Young\cite{Young}, Laplace\cite{Laplace}, and Gauss\cite{Gauss}, along with the boundary conditions encoding prescribed contact angles. The problem fell out of favor during the first half of the 19th century. However, driven on one hand by new mathematical advances in the theory of minimal surfaces and, on the other hand, by practical demands arising from space-age technology and medical applications, the subject has driven more attention in recent decades.

  To study the minimizer of energy functional similar to  \eqref{equ:1.2.1} and the corresponding solution function to Euler-Lagrange equation induced by this energy functional, the BV theory was established by Caccioppoli and De Giorgi, and later deveploped by Miranda, Giaquinta, Anzellotti, Massari, Tamanini, and others. Emmer \cite{Emmer} used this framework to derive the first general existence theorem for capillary surfaces. Independently, the 
ideas of geometric measure theory were introduced and developed by 
Federer, Fleming, Almgren, Allard, and others. Taylor \cite{Taylor} effectively employed this theory to establish boundary regularity results. There are many other related researches on this topic such as the research on the minimal surfaces. For a comprehensive overview of these results, we refer the reader to the monograph by Giusti \cite{Giusti}.

    In the axisymmetric setting in $\mathbb{R}^3$, Robert Finn investigated the existence of equilibrium configurations in several important cases \cite{Finn}. In particular, he proved the existence of a sessile drop for any prescribed volume $V$, as well as the nonexistence of an equilibrium liquid drop on an inclined wall. To the best of our knowledge, however, general contact line problems involving two boundary walls with distinct inclination angles have not yet been systematically studied. In this paper, we focus on this nonsymmetric setting, and our approach is inspired in part by Finn’s work.
    
    More recently, Guo and Tice \cite{Guo,Guo2} studied the dynamic stability of steady states in the case $\theta_{1}=\theta_{2}=\frac{\pi}{2}$. The well-posedness theory for this problem was developed in \cite{Guo1,Guo3,Zheng}, while decay properties of solutions in periodic domains and in the whole space were established in \cite{Guo4,Guo5}, indirectly implying the existence of steady states. Beyond the contact line problem, Tice has analyzed the dynamics and stability of water waves in a variety of related settings \cite{Kim,Tice1,Tice2,Tice3,Tice4,Tice5}. In particular, Tice and Wu \cite{Tice} examined the dynamic stability of sessile drops, which may be viewed as a flat version of the contact line problem corresponding to $\theta_{1}=\theta_{2}=0$, and established both existence and stability of steady states.

All of these problems involve dynamic contact points and contact angles, making them mathematically challenging. Their results have strongly motivated the present work and provided valuable insight. In this paper, our objective is to establish the existence of steady-state solutions for any prescribed volume $V$ and for arbitrary inclination angles $\theta_{1}$ and $\theta_{2}$.
    
    \begin{section}{The first case: $\psi_1\cdot \psi_2<0$ }
    
    In this chapter, we consider the case in which the contact angles $\psi_{1}$ and $\psi_{2}$ have opposite signs. We begin by analyzing the possible shapes of the free surface in Subsection~2.1, and then compute the total volume using the parameter $u_{m}$. Without loss of generality, we assume that $\psi_{1}>0$ and $\psi_{2}<0$, as illustrated in Figure~\ref{figure3}, and that $\theta_{1}>\theta_{2}$.
    
    \begin{subsection}{The shape of the curve}
    
    In this subsection, we assume that there exists a $C^{2}$ curve satisfying the following four properties:

    \begin{itemize}
    
    \item The curve intersects each of the inclined walls at exactly one point, with slope angles $\psi_{1}$ and $\psi_{2}$.
    
     \item It satisfies the equation \eqref{equ:1.2.7} locally. 
     
    \item It can be parametrized piecewise in polar coordinates by the angular variable $\theta$.

    \item The curve does not self-intersect.
    
    \end{itemize}
    
    Under these assumptions, we reparametrize the first equation using the variable $\psi$ defined in \eqref{def:slope}. This yields
    \begin{equation}{\label{equ:bdd_n}}
        \begin{cases}
        \frac{dx}{d\psi}=\sigma\frac{\cos\psi}{gy}\\
    \frac{dy}{d\psi}=\frac{dy}{dx}\frac{dx}{d\psi}=\sigma\tan\psi\frac{\cos\psi}{gy}=\sigma\frac{\sin\psi}{gy}\\
    \end{cases}
    \end{equation}
    \noindent We retain the assumptions stated above in this subsection and use them to study the geometry of the boundary curve. The following key theorem characterizes the geometry of the surface curve.

    \textbf{Remark} At this stage, we do not impose the total volume constraint corresponding to the prescribed constant $V$. Moreover, we only assume that the curve admits a piecewise parametrization in polar coordinates by the angle $\theta$, which allows for the possibility of folds in its polar representation.
    
    \begin{theorem}{\label{Thm:g_1}}({The shape of the surface curve}) Let $\psi$ be defined by \eqref{def:slope}, and suppose there exists a $C^{2}$ curve satisfying the four assumptions stated above. Then the curve have the following geometric properties:
    \begin{itemize}
    \item The curve attains a unique maximum of $y$-coordinate at a point $(x_{m},u_{m})$ and $u_{m}<0$

    \item This curve can be globally parameterized by the slope angle $\psi$. Equivalently, there exists a one-to-one correspondence between $\psi$ and points on the curve.
    \end{itemize}
    \end{theorem}

    \begin{proof}

      We begin by analyzing the geometry of the curve, focusing first on its behavior near the boundary point $(x_{1},u_{1})$. By assumption, the slope angle at this point satisfies $\psi=\psi_{1}>0$, where $\psi_{1}$ is defined by \eqref{equ:1.2.8}.

        \textbf{Case 1: $0< \psi_{1}< \frac{\pi}{2}$}:
        
          In this case, both $\sin\psi_{1}$ and $\cos\psi_{1}$ are strictly positive. Using equation \eqref{equ:bdd_n}, the curve locally satisfies the system

        \begin{equation}{\label{equ:2.1.4}}
        \begin{cases}
            \frac{dy}{d\psi}=\sigma\frac{\sin\psi}{gy}\\
            \frac{dx}{d\psi}=\sigma\frac{\cos\psi}{gy}.
        \end{cases}
        \end{equation}

        \noindent Since $\sin\psi_{1}>0$ and $\cos\psi_{1}>0$, equation \eqref{equ:2.1.4} implies that the curve can be locally parameterized by $\psi$ provided that $u_{1}\neq 0$. We proceed to show  that $u_1<0$. This will be proved through a contradiction argument. Suppose instead that $u_1\geq 0$. The proof proceeds in five steps.

        \textbf{Step 1} We first show that the curve can still be locally parametrized by $\psi$ in a neighborhood of $(x_{1},u_{1})$.

       If $u_{1}>0$, this follows immediately from \eqref{equ:2.1.4} and the positivity of $\sin\psi_{1}$ and $\cos\psi_{1}$. If $u_{1}=0$, then the first equation in \eqref{equ:2.1.4} can be rewritten as
        \begin{equation}{\label{equ:step1}}
            \frac{d\psi}{dy}=\sigma \frac{gy}{\sin\psi}.
        \end{equation}
        \noindent Integrating \eqref{equ:step1} from $u_{1}$ to $y$ yields
        \begin{align}{\label{equ:step1_1}}
            \frac{1}{2}gy^{2}=\sigma\cos\psi_{1}-\sigma \cos\psi .     
        \end{align}
        \noindent Since $\psi_{1}>0$ and $y\ge 0$ in a neighborhood of $(x_{1},u_{1})$, equation \eqref{equ:step1_1} implies that for $\psi$ sufficiently close to $\psi_{1}$, there exists a unique corresponding value of $y$. Substituting \eqref{equ:step1_1} into the second equation of \eqref{equ:2.1.4}, we obtain
        \begin{align}{\label{equ:step1_2}}
            \frac{dx}{d\psi}=\frac{\sigma \cos\psi}{\sqrt{\frac{2\sigma}{g}(\cos\psi_{1}- \cos\psi)}}.
        \end{align}
       Integrating \eqref{equ:step1_2} from $\psi_{1}$ to $\psi$ gives
        \begin{align}
            x=x_{1}+\int_{\psi_{1}}^{\psi}\frac{\sigma \cos\psi}{\sqrt{\frac{2\sigma}{g}(\cos\psi_{1}- \cos\psi)}}d\psi
        \end{align}
        Using the Taylor expansion of $\cos\psi-\cos\psi_{1}$ near $\psi_{1}$, we see that the integrand is locally integrable. Consequently, for each $\psi$ sufficiently close to $\psi_{1}$, there exists a unique point $(x,y)$ on the curve with slope angle $\psi$.
        
        \textbf{Step 2} In this step, we prove that the local parametrization by $\psi$ obtained in Step~1 can be extended to a global parametrization for $\psi \in (\psi_{1}, 2\pi-\psi_{1})$.
        
        Under the assumption that $u_{1}\geq 0$, \eqref{equ:2.1.4} implies that $\frac{du}{d\psi}$ and $\frac{dx}{d\psi}$ are nonnegative in a neighborhood of $\psi_1$. Hence, $u(\psi)$ and $x(\psi)$ are both increasing functions of $\psi$ near $\psi_1$. By a bootstrap argument, we further obtain that $u(\psi)\geq 0$ and $\frac{du}{d\psi}(\psi)\geq 0$ for all $\psi\in(\psi_1,\frac{\pi}{2})$. We then proceed to extend the curve beyond the vertical point $\psi=\frac{\pi}{2}$. 

        Proceeding as in Step~1, we integrate the first equation in the system \eqref{equ:2.1.4} from $\psi_{1}$ to an arbitrary angle $\psi \in (\psi_{1}, \frac{\pi}{2})$, which yields the following relation:
        
        \begin{equation}{\label{equ:2.1.5}}
            \frac{1}{2}gy^{2}(\psi)-\frac{1}{2}gu_1^{2}=\sigma\cos\psi_{1}-\sigma\cos\psi. 
        \end{equation}

        \noindent From equation \eqref{equ:2.1.5}, we notice that $y(\frac{\pi}{2})$ remains finite. Consequently, evaluating the first equation in \eqref{equ:2.1.4} at $\psi=\tfrac{\pi}{2}$ yields

        \begin{equation*}
            \frac{dy}{d\psi}(\frac{\pi}{2})=\frac{\sigma}{gy(\frac{\pi}{2})}\neq 0.
        \end{equation*}
        
        \noindent Therefore, the solution function $(x(\psi),u(\psi))$ of the equation system \eqref{equ:2.1.4} can be extended across the point $\psi=\frac{\pi}{2}$. 

        For $\tfrac{\pi}{2}<\psi\le \pi$, we have $\cos\psi<0$ and $\sin\psi>0$. In this regime, it follows from \eqref{equ:2.1.4} that $x(\psi)$ is strictly decreasing, while $y(\psi)$ is strictly increasing.
        
        When $\pi<\psi<\tfrac{3\pi}{2}$, we again have $\cos\psi<0$ and $\sin\psi<0$. In this interval, $x(\psi)$ remains strictly decreasing, while $y(\psi)$ is decreasing provided that $y(\psi)>0$. In Step~3, we will show that  $y(\psi)=y(2\pi-\psi)$, which means that the curve is symmetric. This symmetry ensures that $y(\psi)>0$ for all $\pi<\psi<\tfrac{3\pi}{2}$, and hence both $x(\psi)$ and $y(\psi)$ are strictly decreasing in this interval; see Figure~\ref{figure5}. As in the case $\psi=\tfrac{\pi}{2}$, the solution can be extended smoothly across $\psi=\tfrac{3\pi}{2}$. By continuing this argument, the curve can be extended up to $\psi=2\pi-\psi_{1}$.

        \begin{figure}
        \centering
        \includegraphics[width=0.9\linewidth=2]{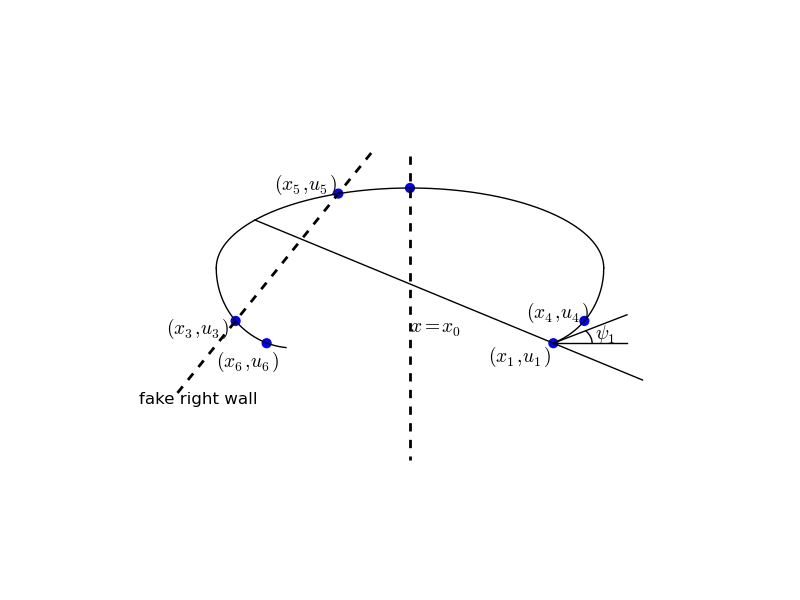}
        \caption{\label{figure5}}
        \end{figure}

        \textbf{Step 3}: In this step, we show that the curve constructed in step 2 exhibits the property $y(\psi)=y(2\pi-\psi)$, which implies the symmetry of the curve. 
        
        Applying the equation \eqref{equ:2.1.5}, we obtain

        \begin{equation}{\label{equ:2.1.51}}
            \frac{1}{2}gy^{2}(\varphi)-\frac{1}{2}gu_1^{2}=\sigma\cos\psi_1-\sigma\cos\varphi
        \end{equation}

        \noindent for any $\varphi\in (\psi_1,2\pi-\psi_1)$. Replacing $\varphi$ by $2\pi-\varphi$ in \eqref{equ:2.1.5} yields

        \begin{equation}{\label{equ:2.1.52}}
            \frac{1}{2}gy^{2}(2\pi-\varphi)-\frac{1}{2}gu_1^{2}=\sigma\cos\psi_1-\sigma\cos(2\pi-\varphi)=\sigma\cos\psi_1-\sigma\cos\varphi
        \end{equation}

         Comparing \eqref{equ:2.1.51} and \eqref{equ:2.1.52}, we conclude that

        \begin{equation}{\label{equ:2.1.53}}
            y(\psi)^{2}=y^{2}(2\pi-\psi),
        \end{equation}

        \noindent for all $\psi\in(\psi_{1},,2\pi-\psi_{1})$. Hence, for each such $\psi$, either $y(\psi)=y(2\pi-\psi)$ or $y(\psi)=-y(2\pi-\psi)$.
        
        Suppose, for the sake of contradiction, that there exists a $\psi_0\in (\pi,2\pi-\psi_1)$ such that $y(\psi_0)=-y(2\pi-\psi_{0})$.  Since $\psi_0\in (\pi,2\pi-\psi_1)$, it follows that $2\pi-\psi_0\in(\psi_1,\pi)$. From the discussion in step 1,  we know that $y(2\pi-\psi_0)>0$ , which implies that $y(\psi_0)<0$. By continuity of $y$, there exists a $\psi_7\in(2\pi-\psi_0,\psi_0)$ such that $y(\psi_7)=y(2\pi-\psi_7)=0$. However, since either $\psi_7\in(\psi_1,\pi]$ or $2\pi-\psi_7\in(\psi_1,\pi]$, the analysis in step 1 yields that either $y(2\pi-\psi_7)> 0$ or $y(\psi_7)>0$, which contradicts the definition of $\psi_7$. Therefore, no such $\psi_{0}$ exists.
        
        We conclude that

        \begin{equation}
            y(\psi)=y(2\pi-\psi),
        \end{equation}

        \noindent which establishes the desired symmetry of the curve.

         \textbf{Step 4} In this step, we prove that the curve we constructed in Step 1 can not satisfy the assumption stated in this theorem.
         
         Using equation \eqref{equ:1.2.600}, it holds that the curve meets the right wall with contact angle $\psi_2$. Accordingly, we consider the point corresponding to the angle $\psi=\psi_2+2\pi$. Suppose, for the sake of contradiction, that the curve intersects the right wall at $\psi=2\pi+\psi_2$. We will show that this assumption leads to a contradiction.
         
         Based on the preceding discussion, we give the definition of $\psi_{3}$ as follows:
         \begin{align}
             \psi_3:=2\pi+\psi_{2}=2\pi-\gamma-\frac{\pi}{2}+\theta_{2}
         \end{align}
         From the definition of $\psi_{2}$ given by \eqref{equ:1.2.9}, it follows that $\psi_{3}$ is an angle lying between $\pi$ and $2\pi$. More precisely, we have:

          \begin{equation*}
              \psi_3=2\pi-\gamma-\frac{\pi}{2}+\theta_{2}<2\pi-(\gamma+\frac{\pi}{2}-\theta_1)=2\pi-\psi_1,
          \end{equation*}

          \noindent where we used the assumption that $\theta_1>\theta_2$ together with equations \eqref{equ:1.2.8} and \eqref{equ:1.2.9}. Hence, $\psi_3\in (\pi,2\pi-\psi_1)$.
        Let $(x_{3},u_{3})=(x(\psi_{3}),y(\psi_{3}))$ denote the point on the curve corresponding to $\psi=\psi_{3}$; see Figure~\ref{figure5}. Under the assumptions of the theorem, the curve cannot intersect the right wall for any $\psi\in(\psi_{1},\psi_{3})$.

          By definition, we have

          \begin{equation*}
          \psi_{3}-\pi=2\pi+\psi(x_2)-\pi=\frac{\pi}{2}-\gamma+\theta_{2}>\theta_2
          \end{equation*}
          
          \noindent By a straightforward geometric argument, it follows that for any $\psi\in (\pi+\theta_2,\psi_3)$, the point $(x(\psi),u(\psi))$ is on the left hand side of the right wall; see Figure \ref{figure5}. 
           Using the property that $\theta_1>\theta_2$, we then define the angle $\psi_{4}$ as follows

          \begin{equation*}
          \psi_{4}:=2\pi-\psi_3=-\psi_{2}=\gamma+\frac{\pi}{2}-\theta_2>\gamma+\frac{\pi}{2}-\theta_1=\psi_1, 
          \end{equation*}
          
          \noindent which implies that $\psi_4\in(\psi_1,\pi)$. Moreover,  by the symmetry established in Step~3, we have $u(\psi_{4})=u(\psi_{3})$; see Figure \ref{figure5}. To avoid self-intersection, it must hold that $x(\psi_{4})>x(\psi_{3})$, meaning that the point $(x(\psi_{4}),u(\psi_{4}))$ lies on the right hand side of the right wall. Therefore, the curve must intersect the wall at a point $(x(\psi_{5}),u(\psi_{5}))$ with $\psi_{4}<\psi_{5}<\psi_{3}$, which contradicts to the assumption made in this step.

        Therefore, the curve can not meet the right wall when $\psi_1<\psi\leq 2\pi-\psi_3$. However, when $\psi_{6}=2\pi-\psi_1$,  Step 3 may yields $y(\psi_{6})=y(\psi_1)$. To avoid self-intersection, we must have $x(\psi_{6})<x(\psi_{1})$, which implies that the point $(x(\psi_{6}),u(\psi_{6}))$ lies to the left of the left wall. Hence, the curve must intersect—and in fact cross—the left wall at some angle between $\psi_{1}$ and $2\pi-\psi_{1}$, contradicting the assumptions of the theorem. We therefore conclude that $u_1<0$.

       With $u_{1}<0$ established, equation \eqref{equ:1.3.4} implies that $\psi(x)$ decreases whenever $y(x)<0$. Suppose that there exists an angle $\psi_{8}$ with $0<\psi_{8}<\psi_{1}$ such that $y(\psi_{8})=0$. Repeating the same argument used in the case $u_{1}\ge 0$, we again arrive at a contradiction. Hence, at the point where $\sin\psi=0$ and $u$ attains its maximum, we must have $u_{m}=y(0)<0$. Consequently, for every $\psi\in(\psi_{2},\psi_{1})$, we have $y(\psi)<0$, ensuring that the system \eqref{equ:2.1.4} is well defined throughout this interval. If the curve meets the right wall at $\psi=\psi_{2}$, then $\psi$ provides a global parametrization of the curve, and the $y$-component attains its unique maximum at $(x_{m},u_{m})$ with $u_{m}<0$. The remainder of the proof is devoted to showing that the curve indeed intersects the right wall at $\psi=\psi_{2}$.

       \textbf{Step 5} In this step, we show that the curve constructed in the last step intersects the right wall precisely at $\psi=\psi_{2}$. As a consequence, the point at which the $y$–component attains its maximum is unique.

       We argue by contradiction. Suppose that the curve does not meet the right wall at $\psi=\psi_{2}$. Since, by assumption, the curve can intersect the right wall only with contact angle $\psi_{2}$, it follows that the curve can be parametrized by $\psi$ for $\psi\in(-2\pi+\psi_{2},\psi_{1})$. At $\psi=-\pi$, we define $(x(-\pi),y(-\pi))=(x_{min},y_{min})$. Using the same discussion as in the Step 2 and Step 3 above, we deduce that $y_{min}<u_{m}$ and $y(-\pi+\psi)=y(-\pi-\psi)$ for all $\psi\in [0,\pi]$. Therefore, we have $y(-2\pi)=y(0)=u_{m}$. 
       
       Let $\mathcal{M}$ denote the region enclosed by the curve corresponding to $\psi\in (-\pi,\psi_{1})$, the line $y=y_{min}$, and the left wall. Using the assumption that the curve does not self-intersect, $(x(-2\pi),y(-2\pi))$ lies outside of region $\mathcal{M}$. Using the Jordan curve theorem, the curve must intersect the boundary of $\mathcal{M}$ at some value of $\psi\in (-\pi,-2\pi)$. However, by the same discussion used in Step 2 above, we have $y(\psi)>y_{min}$ for all $\psi\in (-\pi,-2\pi)$, which rules out intersection with the line $y=y_{\min}$ on this interval. Consequently, the curve must intersect either the left wall or itself for some $\psi\in(-2\pi,-\pi)$, contradicting the assumptions of the theorem.

        \textbf{Case 2} If $\frac{\pi}{2}\le \psi(x_{1})<\pi$, the argument proceeds in the same manner as in Case~1. In fact, the analysis is even simpler in this situation, since there is no need to extend the surface curve across the vertical angle $\psi=\tfrac{\pi}{2}$. This completes the proof of the theorem.
    \end{proof}

    \end{subsection}

    \begin{subsection}{The existence of the solution function}

   In the previous subsection, we derived an a priori geometric description of the free boundary, which provides insight into the shape of the steady-state configuration. In this subsection, for a prescribed volume $V$, we address the existence of a steady state.
     
    By Theorem~\ref{Thm:g_1}, the free boundary can be parametrized by the slope angle $\psi$, defined by \eqref{def:slope}. Accordingly, we construct the curve by solving the system \eqref{equ:1.2.6} subject to the initial conditions \eqref{equ:1.3.1}–\eqref{equ:1.3.2}. After applying the vertical shift $(x,y)\mapsto (x,y-\tfrac{P_{0}}{g})$, the original Euler–Lagrange system \eqref{equ:1.2.3} becomes equivalent to the system \eqref{equ:1.2.7}. Therefore, in this subsection we focus on the analysis of \eqref{equ:1.2.7}.

By the existence and uniqueness theorem for ordinary differential equations, the solution of \eqref{equ:1.2.7} is uniquely determined once the initial data is specified. We begin by the following lemma.

    \begin{lemma}{\label{lem:curve}}
        Suppose that the curve $(x(\psi),y(\psi))$ is a solution of \eqref{equ:bdd_n} with initial condition

        \begin{align}
            \begin{aligned}
                (x(0),y(0))=(x_{m},u_{m}),~~u_{m}<0.
            \end{aligned}
        \end{align}
          Then, for any $\psi\in(\psi_{2},\psi_{1})$, the functions $x(\psi)$ and $y(\psi)$ admit the following expressions
        \begin{align}
            y(\psi)=&-\sqrt{u_{m}^{2}+\frac{2\sigma}{g}(1-\cos\psi)}\label{equ:y_0}\\
            x(\psi)=&x_{m}-\int_{0}^{\psi}\frac{\sigma\cos\varphi}{g\sqrt{u_{m}^{2}+\frac{2\sigma}{g}(1-\cos\varphi)}}d\varphi\label{equ:x_0}\\
            r(\psi):=&|x(\psi)-x_{m}|=|\int_{0}^{\psi}\frac{\sigma\cos\varphi}{g\sqrt{u_{m}^{2}+\frac{2\sigma}{g}(1-\cos\varphi)}}d\varphi| \label{equ:2.2.100}
        \end{align}
    \end{lemma}
    \begin{proof}
        \textbf{Step 1 Proof of \eqref{equ:y_0}}: 
        
        Integrating both sides of the equation \eqref{equ:bdd_n} from $0$ to any $\psi\in (\psi_{2},\psi_{1})$, we obtain
    
        \begin{equation}{\label{equ:2.2.7}}
           \frac{1}{2}gy^{2}(\psi)-\frac{1}{2}gu_m^{2}=\sigma(1-\cos\psi),
        \end{equation}

        \noindent which is equivalent to

        \begin{equation}{\label{equ:2.2.8}}
             \frac{1}{2}gy(\psi)^{2}=\frac{1}{2}gu_{m}^{2}+\sigma(1-\cos\psi)
        \end{equation}

         \noindent By Theorem \ref{Thm:g_1}, it holds that $u_{m}<0$. Since $(x_m,u_m)$ is the point where the y-component of the curve attains its maximum. Consequently, $y(\psi)<0$ for all $\psi\in(\psi_{2},\psi_{1})$. Taking the negative square root of \eqref{equ:2.2.8}, we conclude that
        
        \begin{equation*}
            y(\psi)=-\sqrt{u_m^{2}+\frac{2\sigma}{g}(1-\cos\psi)}.
        \end{equation*} 

\textbf{ Step 2 Proof of \eqref{equ:x_0}}:

 Using the second equation in the system \eqref{equ:bdd_n}, we obtain

        \begin{equation}{\label{equ:2.2.9}}
            \frac{dx}{d\psi}=\sigma\frac{\cos\psi}{gy(\psi)}.
        \end{equation}

        \noindent Substituting the expression for $y(\psi)$ derived in Step~1,

        \begin{equation}{\label{equ:2.2.10}}
            y(\psi)=-\sqrt{u_{m}^{2}+\frac{2\sigma}{g}(1-\cos\psi)},
        \end{equation}

        \noindent into \eqref{equ:2.2.9} yields

        \begin{equation}{\label{equ:2.2.101}}
            \frac{dx}{d\psi}=\frac{\sigma\cos\psi}{-g\sqrt{u_{m}^{2}+\frac{2\sigma}{g}(1-\cos\psi)}}.
        \end{equation}

        Integrating \eqref{equ:2.2.101} from $0$ to $\psi$ gives

        \begin{equation*}
            x(\psi)-x_m=-\int_{0}^{\psi_1}\frac{\sigma\cos\gamma}{g\sqrt{u_{m}^{2}+\frac{2\sigma}{g}(1-\cos\gamma)}}d\gamma,
        \end{equation*}

    \noindent which is precisely the equation \eqref{equ:x_0}.
    
    Finally, the proof of \eqref{equ:2.2.100} is just another straightforward application of the equation \eqref{equ:2.2.101}. 
    \end{proof}
    
     A straightforward observation shows that the curve constructed in Lemma~\ref{lem:curve} is uniquely determined once the parameters $x_{m}$ and $u_{m}$ are specified. We therefore treat $x_{m}$ and $u_{m}$ as new parameters and derive a formula for the total enclosed volume in terms of them, rather than in terms of the unknown constant $P_{0}$. Since the system \eqref{equ:1.2.600} imposes two boundary conditions, the curve constructed in Lemma~\ref{lem:curve} must satisfy both, which effectively reduces the number of independent parameters among $x_{m}$ and $u_{m}$ to one. Consequently, both the total volume and the constant $P_{0}$ can be expressed as functions of a single parameter. By enforcing the volume constraint to equal the prescribed constant $V$, this parameter is uniquely determined. This strategy underlies the remainder of this section.

    A key difficulty in implementing the boundary conditions in \eqref{equ:1.2.7} lies in determining the location of the original point $O$—the intersection of the two walls—after the coordinate shift. In particular, the $y$-coordinate of $O$ is unknown; see Figure~\ref{figure3}. To be consistent with the transformation $(x,y)\mapsto (x,y-\tfrac{P_{0}}{g})$, the $y$-coordinate of $O$ must equal $-P_{0}/g$. However, since $P_{0}$ has not yet been determined, we cannot directly use this geometric relations to compute the coordinates of the two contact points.

Nevertheless, the curve is known to meet the left and right walls at the prescribed contact angles $\psi_{1}$ and $\psi_{2}$, respectively. Hence, once the curve is constructed by solving \eqref{equ:1.2.7} with the chosen initial conditions, the coordinates of the two contact points can be obtained simply by evaluating the solution at $\psi=\psi_{1}$ and $\psi=\psi_{2}$. This observation leads to the following theorem.
    
    \begin{theorem}( Boundary Points){\label{thm:bdd_p}} Suppose that the curve constructed in Lemma \ref{lem:curve} intersects two walls at two contact points $(x_1,u_1)$ and $(x_2,u_2)$, corresponding to the slope angles $\psi_{1}$ and $\psi_{2}$, respectively. Then the coordinates of these contact points are given by

        \begin{align}
            &u_1=-\sqrt{u_m^{2}+\frac{2\sigma}{g}(1-\cos\psi_1)} \label{equ:2.2.3}\\
             &u_2=-\sqrt{u_{m}^{2}+\frac{2\sigma}{g}(1-\cos\psi_2)} \label{equ:2.2.4}\\
              &x_1=x_m-\int_{0}^{\psi_1}\frac{\sigma\cos\gamma}{g\sqrt{u_{m}^{2}+\frac{2\sigma}{g}(1-\cos\gamma)}}d\gamma \label{equ:2.2.5}\\
              &x_2=x_m-\int_{0}^{\psi_2}\frac{\sigma\cos\gamma}{g\sqrt{u_{m}^{2}+\frac{2\sigma}{g}(1-\cos\gamma)}}d\gamma \label{equ:2.2.6}
        \end{align}
        
    \end{theorem}

    \begin{proof}

   The results directly follows from the lemma \ref{lem:curve}. Equations \eqref{equ:2.2.3} and \eqref{equ:2.2.4} are proved by setting $\psi=\psi_{1}$ and $\psi=\psi_{2}$, respectively, in equation \eqref{equ:y_0}. Similarly, equations \eqref{equ:2.2.5} and \eqref{equ:2.2.6} are proved by setting $\psi=\psi_{1}$ and $\psi=\psi_{2}$, respectively, in equation \eqref{equ:x_0}.

    \end{proof}

     Suppose that the curve we constructed in Lemma \ref{lem:curve} satisfies the assumptions of Theorem \ref{thm:bdd_p}. We impose these conditions on the curve throughout the remainder of this subsection.
     
     Having derived the coordinates of boundary contact points, we proceed to express the total volume in terms of $u_{m}$ and $x_{m}$. Before deriving this formula, we first establish a relationship between $x_{m}$ and $u_{m}$ using the geometry relation, thereby reducing the problem to a single independent parameter.

    \begin{theorem}{\label{thm:relation}}

        Suppose that $r(\psi)$ is defined as in Lemma \ref{lem:curve}, and that $x_{1},u_{1},x_{2},u_{2}$ are defined as in Theorem \ref{thm:bdd_p}, then the following relations hold:

        \begin{align}
             &x_{m}=\frac{r_1\tan\theta_{1}-r_{2}\tan\theta_{2}-(u_1-u_{2})}{\tan\theta_{1}+\tan\theta_{2}}~\operatorname{when}~\theta_1\neq \frac{\pi}{2} \label{equ:2.2.11}\\
              &x_m=r_1~\operatorname{when}~\theta_1=\frac{\pi}{2} \label{equ:2.2.12}\\
               &P_{0}=g(x_{2}\tan\theta_2-u_{2}) \label{equ:2.2.13},
        \end{align}
        
        \noindent where

        \begin{align}
            r_1&:=r(\psi_1)=x_m-x_1 \label{equ:2.2.102}\\
             r_2&:=r(\psi_2)=x_2-x_m \label{equ:2.2.103}.
        \end{align}
        
    \end{theorem}

    \textbf{Remark}: Equations \eqref{equ:2.2.11}, \eqref{equ:2.2.12}, and \eqref{equ:2.2.13} imply that $u_1$, $u_2$, $r_1$, and $r_2$ are all functions of $u_{m}$. Consequently, $x_{m}$ can also be expressed as a function of $u_{m}$.
    
    \begin{proof}

        We first note that when $\theta_{1}=\frac{\pi}{2}$, the identity $x_{m}=r_{1}$ follows directly from the definition \eqref{equ:2.2.100}. Therefore, it suffices to prove \eqref{equ:2.2.11} and \eqref{equ:2.2.13}.
        
         Let the two walls intersect at the point $O$. By our choice of coordinates, the $x$–coordinate of $O$ is zero, while its $y$–coordinate, denoted by $y_{0}$, is to be determined. Since both $(0,y_{0})$ and $(x_{1},u_{1})$ lie on the left wall inclined at angle $\theta_{1}$, we obtain

        \begin{equation}{\label{equ:2.2.14}}
        y_0=u_1-(-x_1)\tan\theta_1.
        \end{equation}

        \noindent Similarly, since $(x_2,u_2)$ and $(0,y_0)$ lie on the right wall inclined at angle $\theta_2$, we have:
        
        \begin{equation}{\label{equ:2.2.15}}
          y_0=u_2-x_2\tan\theta_2.
        \end{equation}

        Combining \eqref{equ:2.2.14} and \eqref{equ:2.2.15} together, we deduce

        \begin{equation}{\label{equ:2.2.16}}
            u_1-(-x_1)\tan\theta_1=u_2-x_{2}\tan\theta_2.
        \end{equation}

        \noindent  Substituting the definitions of $r_{1}$ and $r_{2}$ given by \eqref{equ:2.2.102} and \eqref{equ:2.2.103} into \eqref{equ:2.2.16}, we obtain

        \begin{equation*}
            u_1-(r_1-x_m)\tan\theta_1=u_2-(r_2+x_m)\tan\theta_2.
        \end{equation*}

        \noindent Rearranging terms yields
        
        \begin{equation*}
            x_{m}(\tan\theta_1+\tan\theta_2)=u_2-u_1+r_1\tan\theta_1-r_{2}\tan\theta_2,
        \end{equation*}

        \noindent and hence

        \begin{equation*}
            x_{m}=\frac{r_1\tan\theta_1-r_{2}\tan\theta_2-(u_1-u_2)}{(\tan\theta_1+\tan\theta_2)},
        \end{equation*}

        \noindent which is precisely the equation \eqref{equ:2.2.11}.

        Finally, using \eqref{equ:2.2.15}, the coordinates of the intersection point $O$ are given by $(0,y_0)=(0,u_2-x_{2}\tan\theta_2)$. After applying the shift $(x,y)\rightarrow (x,y+\frac{P_{0}}{g})$, the point $O$ is mapped to the origin $(0,0)$. Therefore, 

        \begin{equation*}
            0=\frac{P_{0}}{g}+(u_2-x_2\tan\theta_2),
        \end{equation*}

        \noindent which implies

        \begin{equation*}
            P_{0}=g(x_2\tan\theta_2-u_2).
        \end{equation*}

        \noindent This finishes the proof.
    \end{proof}
    
    Now we have the expressions for $P_{0}$ and $x_m$ in terms of $u_{m}$. In order to proceed with the computation of the enclosed area, it is necessary to determine the sign of $x_{m}$. This motivates the following lemma.

    \begin{lemma}{\label{lem:pos}}

    Suppose that $x_{m}$ is expressed as a function of $u_{m}$ as in Theorem~\ref{thm:relation}. Then the following properties hold:

    \begin{itemize}
    \item If $\theta_1=\theta_2$, $x_{m}=0$.
    \item   If $\theta_{2}<\theta_{1}<\frac{\pi}{2}$ and $[\![\gamma]\!]\le 0$, then $x_{m}$ is an increasing function of $\theta_{1}$, provided that $\theta_{2}$ and $u_{m}$ are fixed.
    \end{itemize}
    \end{lemma}

    \begin{proof}

     Using Theorem \ref{thm:relation}, we have the formula for $x_{m}$:

    \begin{equation}{\label{equ:2.2.17}}
            x_{m}=\frac{r_1\tan\theta_{1}-r_{2}\tan\theta_{2}-(u_1-u_2)}{\tan\theta_{1}+\tan\theta_{2}}
    \end{equation}

    Let $f(\theta_{1})$ denote the numerator on the right-hand side of \eqref{equ:2.2.17}. Then $f(\theta_{1})$ can be written as
    \begin{equation}{\label{equ:2.2.110}}
        f(\theta_1)=r_1\tan\theta_{1}-r_{2}\tan\theta_{2}-(u_1-u_2)
    \end{equation}
    
    \textbf{Case one $\theta_1=\theta_2$}: 

     When $\theta_1=\theta_2$, equations \eqref{equ:1.2.8} and \eqref{equ:1.2.9} imply that $\psi_1=-\psi_2$, and consequently $\cos\psi_{1}=\cos\psi_{2}$. Substituting this relation into equations \eqref{equ:2.2.3}–\eqref{equ:2.2.6} in Theorem~\ref{thm:bdd_p}, we obtain:

    \begin{equation*}
        u_1=u_2~~and~~r_1=r_2
    \end{equation*}

    \noindent Hence, we conclude that $x_m=0$ from equation \eqref{equ:2.2.11}. 

    \textbf{Case two $\theta_1\neq \theta_2$}: 
    
    We now consider the general case where $\theta_1>\theta_2$. To establish the desired result, we compute $\frac{df}{d\theta_{1}}$, where $f$ is defined as in equation \eqref{equ:2.2.110}.

    From equations \eqref{equ:2.2.3}–\eqref{equ:2.2.6} in Theorem~\ref{thm:bdd_p}, we obtain

        \begin{equation}{\label{equ:2.2.18}}
            r_{1}=x_m-x_1=\int_{0}^{\psi_1}\frac{\sigma\cos\psi d\psi}{g\sqrt{u_{m}^{2}+\frac{2\sigma}{g}(1-\cos\psi)}},
        \end{equation}

        \begin{equation}{\label{equ:2.2.19}}
            r_{2}=x_2-x_m=\int_{\psi_2}^{0}\frac{\sigma\cos\psi d\psi}{g\sqrt{u_{m}^{2}+\frac{2\sigma }{g}(1-\cos\psi)}},
        \end{equation}

        \noindent and

        \begin{equation}{\label{equ:2.2.20}}
            u_1-u_2=-\sqrt{u_{m}^{2}+\frac{2\sigma}{g}(1-\cos\psi_{1})}+\sqrt{u_{m}^{2}+\frac{2\sigma}{g}(1-\cos\psi_{2})}.
        \end{equation}

        \noindent  Substituting \eqref{equ:2.2.18}–\eqref{equ:2.2.20} into \eqref{equ:2.2.110}, we differentiate $f(\theta_{1})$ with respect to $\theta_{1}$ to show

        \begin{align}
            \frac{df(\theta_1)}{d\theta_1}=&\frac{1}{\cos^{2}\theta_{1}}r_1+\frac{dr_1}{d\theta_1}\tan\theta_{1}-\frac{du_1}{d\theta_1} \notag\\
            =&\frac{\sigma}{\cos^{2}\theta_1}\int_{0}^{\psi_1}\frac{\cos\psi d\psi}{g\sqrt{u_{m}^{2}+\frac{2\sigma}{g}(1-\cos\psi)}}+\tan\theta_{1}\frac{\sigma\cos\psi_1}{g\sqrt{u_m^{2}+\frac{2\sigma}{g}(1-\cos\psi_1)}}\frac{d\psi_1}{d\theta_{1}} \notag\\
            &+\frac{\sigma\sin\psi_1}{g\sqrt{u_{m}^{2}+\frac{2\sigma}{g}(1-\cos\psi_1)}}\frac{d\psi_1}{d\theta_{1}} \label{equ:2.2.21},
        \end{align}
        where we have used the chain rule and the dependence of $\psi_{1}$ on $\theta_{1}$ given by \eqref{equ:1.2.8}.
        
        To move on to the next step, we first introduce an inequality below :

        \begin{equation}{\label{equ:2.2.22}}
        \int_{0}^{\psi_1}\frac{\cos\psi d\psi}{g\sqrt{u_{m}^{2}+\frac{2\sigma}{g}(1-\cos\psi)}}\geq \int_{0}^{\psi_1} \frac{\cos\psi d\psi}{g\sqrt{u_{m}^{2}+\frac{2\sigma}{g}(1-\cos\psi_1)}}
        \end{equation}

        \noindent This inequality will be established in Lemma~2.6, whose proof is postponed until after the main argument.
        
          Applying inequality \eqref{equ:2.2.22} to the equation $\eqref{equ:2.2.21}$, we obtain:

        \begin{align*}
            \frac{df}{d\theta_1} \geq &\frac{\sigma}{g\cos^{2}\theta_1\sqrt{u_m^{2}+\frac{2\sigma}{g}(1-\cos\psi_1)}}\int_{0}^{\psi_1} \cos\psi d\psi \\
            &+\tan\theta_1\frac{\sigma\cos\psi_1}{g\sqrt{u_{m}^{2}+\frac{2\sigma}{g}(1-\cos\psi_1)}}\frac{d\psi_1}{d\theta_{1}}+\frac{\sigma\sin\psi_1}{g\sqrt{u_{m}^{2}+\frac{2\sigma}{g}(1-\cos\psi_1)}}\frac{d\psi_1}{d\theta_1} 
        \end{align*}

          From \eqref{equ:1.2.8}, we have $\psi_1=\gamma+\frac{\pi}{2}-\theta_1$, which implies $\frac{d\psi(x_1)}{d\theta_{1}}=-1$. Substituting this identity into the inequality above yields

        \begin{align}
            \frac{df(\theta_1)}{d\theta_1}\geq &\frac{\sigma\sin\psi(x_1)}{g\cos^{2}\theta_1\sqrt{u_m^{2}+\frac{2\sigma}{g}(1-\cos\psi(x_1))}}-\tan\theta_1\frac{\sigma \cos\psi(x_1)}{g\sqrt{u_m^{2}+\frac{2\sigma}{g}(1-\cos\psi(x_1))}}\notag\\
             &-\frac{\sigma\sin\psi(x_1)}{g\sqrt{u_m^{2}+\frac{2\sigma}{g}(1-\cos\psi(x_1))}}\notag\\
             =&\frac{\sigma}{g\sqrt{u_{m}^{2}+\frac{2\sigma}{g}(1-\cos\psi(x_1))}}(\frac{\sin\psi(x_1)}{\cos^{2}\theta_1}-\tan\theta_1 \cos\psi(x_1)-\sin\psi(x_1))\notag\\
             =&\frac{\sigma }{g\sqrt{u_{m}^{2}+\frac{2\sigma}{g}(1-\cos\psi(x_1))}}(\sin\psi(x_1)\tan^{2}\theta_1-\tan\theta_1\cos\psi(x_1))\notag\\
             =&\frac{\sigma }{g\sqrt{u_{m}^{2}+\frac{2\sigma}{g}(1-\cos\psi(x_1))}}\tan\theta_1(\sin\psi(x_1)\tan\theta_1-\cos\psi(x_1)) \notag\\
             =&\frac{\sigma }{g\sqrt{u_{m}^{2}+\frac{2\sigma}{g}(1-\cos\psi(x_1))}}\tan\theta_1(\sin(\frac{\pi}{2}-\theta_1+\gamma)\tan\theta_1-\cos(\frac{\pi}{2}-\theta_1+\gamma)) \notag\\
              =&\frac{\sigma}{g\sqrt{u_{m}^{2}+\frac{2\sigma}{g}(1-\cos\psi(x_1))}}\tan\theta_1(\cos(\theta_1-\gamma)\tan\theta_1-\sin(\theta_1-\gamma)) \notag \\
            =&\frac{\sigma}{g\sqrt{u_{m}^{2}+\frac{2\sigma}{g}(1-\cos\psi(x_1))}}\tan\theta_1\cos(\theta_1-\gamma)(\tan\theta_1-\tan(\theta_1-\gamma)) \label{equ:2.2.23}
        \end{align}

        Since $[\![\gamma]\!]\leq 0$, it follows from the definition of $\gamma$ in \eqref{equ:1.2.110} that:

        \begin{equation*}
            \sin\gamma=-\frac{[\![\gamma]\!]}{\sigma} \geq 0,  
        \end{equation*}

        \noindent hence $\gamma\in [0,\frac{\pi}{2})$. Moreover, under the assumption $0<\theta_1<\frac{\pi}{2}$, we have $\cos(\theta_1-\gamma)>0$ and $\tan\theta_1-\tan(\theta_1-\gamma)>0$. Substituting these inequalities into \eqref{equ:2.2.23}, we conclude that

        \begin{equation*}
            \frac{df(\theta_1)}{d\theta_1}\geq \frac{1}{g\sqrt{u_{m}^{2}+\frac{2}{g}(1-\cos\psi(x_1))}}\tan\theta_1\cos(\theta_1-\gamma)(\tan\theta_1-\tan(\theta_1-\gamma))\geq 0
        \end{equation*}

        \noindent This completes the proof.
    \end{proof}

     It remains to prove the inequality \eqref{equ:2.2.22}. To this end, we establish the following lemma.

    \begin{lemma}{\label{lem:1}}
    For any $\psi_{1}\in(0,\pi)$, the following inequality holds:
        \begin{equation}{\label{equ:key_ineq}}
        \int_{0}^{\psi_1}\frac{\cos\psi d\psi}{g\sqrt{u_{m}^{2}+\frac{2\sigma}{g}(1-\cos\psi)}}\geq \int_{0}^{\psi_1} \frac{\cos\psi d\psi}{g\sqrt{u_{m}^{2}+\frac{2\sigma}{g}(1-\cos\psi_{1})}}
        \end{equation}
    \end{lemma}

    \begin{proof}

    We first consider the case $\psi_{1}\in\bigl(0,\tfrac{\pi}{2}\bigr)$. In this range, the result is straightforward. Indeed, $\cos\psi>0$ for all $\psi\in(0,\psi_{1})$, and moreover $\cos\psi\ge\cos\psi_{1}$ for every $\psi\in(0,\psi_{1})$. Consequently, for all such $\psi$ we have

    \begin{equation}{\label{equ:2.2.24}}
        \frac{\cos\psi}{g\sqrt{u_{m}^{2}+\frac{2\sigma}{g}(1-\cos\psi)}}\geq \frac{\cos\psi}{g\sqrt{u_{m}^{2}+\frac{2\sigma}{g}(1-\cos\psi_1)}}
    \end{equation}

    \noindent Multiplying both sides by the positive factor $\cos\psi/g$ and integrating over $(0,\psi_{1})$ yields \eqref{equ:key_ineq}.
    
    If $\psi(x_1)\in (\frac{\pi}{2},\pi)$, we split the integral on the left hand side of \eqref{equ:key_ineq} into three parts as follows

    \begin{align}{\label{eq:integral_0}}
         \int_{0}^{\psi_1}\frac{\cos\psi d\psi}{g\sqrt{u_{m}^{2}+\frac{2\sigma}{g}(1-\cos\psi)}}=&\int_{0}^{\pi-\psi_1} \frac{\cos\psi d\psi}{g\sqrt{u_{m}^{2}+\frac{2\sigma}{g}(1-\cos\psi)}}\\
          &+\int_{\pi-\psi_1}^{\frac{\pi}{2}} \frac{\cos\psi d\psi}{g\sqrt{u_{m}^{2}+\frac{2\sigma}{g}(1-\cos\psi)}}+\int_{\frac{\pi}{2}}^{\psi_1} \frac{\cos\psi d\psi}{g\sqrt{u_{m}^{2}+\frac{2\sigma}{g}(1-\cos\psi)}}.
    \end{align}

       \noindent For $\psi\in(0,\pi-\psi_{1})$, we have $\cos\psi\ge0\ge\cos\psi_{1}$, which implies 

        \begin{equation}{\label{equ:2.2.25}}
            \int_{0}^{\pi-\psi_1} \frac{\cos\psi d\psi}{g\sqrt{u_{m}^{2}+\frac{2\sigma}{g}(1-\cos\psi)}}\geq \int_{0}^{\pi-\psi_1} \frac{\cos\psi d\psi}{g\sqrt{u_{m}^{2}+\frac{2\sigma}{g}(1-\cos\psi_1)}}.
        \end{equation}

        \noindent This finishes the estimate for the first term on the right hand side of equation \eqref{eq:integral_0}. For the second term and the third term, we combine these remaining two integrals and use the identity $\cos(\pi-\psi)=-\cos\psi$. Writing the difference explicitly, we obtain

        \begin{align}
             &\int_{\pi-\psi_1}^{\frac{\pi}{2}} \frac{\cos\psi d\psi}{g\sqrt{u_{m}^{2}+\frac{2\sigma}{g}(1-\cos\psi)}}+\int_{\frac{\pi}{2}}^{\psi_1} \frac{\cos\psi d\psi}{g\sqrt{u_{m}^{2}+\frac{2\sigma}{g}(1-\cos\psi)}}\notag\\
              =&\int_{\pi-\psi_1}^{\frac{\pi}{2}} \frac{\cos\psi d\psi}{g\sqrt{u_{m}^{2}+\frac{2\sigma}{g}(1-\cos\psi_1)}}+\int_{\frac{\pi}{2}}^{\psi_1} \frac{\cos\psi d\psi}{g\sqrt{u_{m}^{2}+\frac{2\sigma}{g}(1-\cos\psi_1)}}\notag\\
               &+\int_{\pi-\psi_1}^{\frac{\pi}{2}} \frac{\cos\psi d\psi}{g\sqrt{u_{m}^{2}+\frac{2\sigma}{g}(1-\cos\psi)}}-\frac{\cos\psi d\psi}{g\sqrt{u_{m}^{2}+\frac{2\sigma}{g}(1-\cos\psi_1)}}\notag\\
               &+\int_{\pi-\psi_1}^{\frac{\pi}{2}} \frac{\cos\psi d\psi}{g\sqrt{u_{m}^{2}+\frac{2\sigma}{g}(1-\cos\psi_1)}}-\frac{\cos\psi d\psi}{g\sqrt{u_{m}^{2}+\frac{2\sigma}{g}(1+\cos\psi)}}\notag\\
               =&\int_{\pi-\psi_1}^{\frac{\pi}{2}} \frac{\cos\psi d\psi}{g\sqrt{u_{m}^{2}+\frac{2\sigma}{g}(1-\cos\psi_1)}}+\int_{\frac{\pi}{2}}^{\psi_1} \frac{\cos\psi d\psi}{g\sqrt{u_{m}^{2}+\frac{2\sigma}{g}(1-\cos\psi_1)}}\label{equ:2.2.26}\\
                &+\int_{\pi-\psi_1}^{\frac{\pi}{2}} \frac{\cos\psi d\psi}{g\sqrt{u_{m}^{2}+\frac{2\sigma}{g}(1-\cos\psi)}}-\frac{\cos\psi d\psi}{g\sqrt{u_{m}^{2}+\frac{2\sigma}{g}(1+\cos\psi)}} \label{equ:2.2.27}
        \end{align}

        \noindent  We retain the terms in the line \eqref{equ:2.2.26} and proceed to show that the expression in the line \eqref{equ:2.2.27} is positive. 
        
        Since

        \begin{equation*}
            \cos\psi>0~\operatorname{when}~\psi\in(\pi-\psi_1,\frac{\pi}{2}),
        \end{equation*}

       \noindent it follows that

        \begin{equation*}
            g\sqrt{u_{m}^{2}+\frac{2\sigma}{g}(1-\cos\psi)}<g\sqrt{u_{m}^{2}+\frac{2\sigma}{g}(1+\cos\psi)}.
        \end{equation*}

        \noindent Consequently,
        
        \begin{equation}{\label{equ:2.1.1000}}
            \int_{\pi-\psi_1}^{\frac{\pi}{2}} \frac{\cos\psi d\psi}{g\sqrt{u_{m}^{2}+\frac{2\sigma}{g}(1-\cos\psi)}}-\frac{\cos\psi d\psi}{g\sqrt{u_{m}^{2}+\frac{2\sigma}{g}(1+\cos\psi)}}>0
        \end{equation}

        \noindent Substituting \eqref{equ:2.1.1000} into \eqref{equ:2.2.27} yields

        \begin{align}{\label{equ:2.2.28}}
        \begin{aligned}
             &\int_{\pi-\psi_1}^{\frac{\pi}{2}} \frac{\cos\psi d\psi}{g\sqrt{u_{m}^{2}+\frac{2\sigma}{g}(1-\cos\psi)}}+\int_{\frac{\pi}{2}}^{\psi_1} \frac{\cos\psi d\psi}{g\sqrt{u_{m}^{2}+\frac{2\sigma}{g}(1-\cos\psi)}}\\
              \geq& \int_{\pi-\psi_1}^{\frac{\pi}{2}} \frac{\cos\psi d\psi}{g\sqrt{u_{m}^{2}+\frac{2\sigma}{g}(1-\cos\psi_1)}}+\int_{\frac{\pi}{2}}^{\psi_1} \frac{\cos\psi d\psi}{g\sqrt{u_{m}^{2}+\frac{2\sigma}{g}(1-\cos\psi_1)}}.
        \end{aligned}
        \end{align}

       Finally, combining inequality \eqref{equ:2.2.28} with \eqref{equ:2.2.25} yields the desired result.
    \end{proof}
    
    \textbf{Remark}: From Lemma~\ref{lem:pos}, we conclude that $x_{m}>0$ when $\theta_{2}<\theta_{1}<\frac{\pi}{2}$ and $[\![\gamma]\!]<0$. Moreover, when $\theta_{1}>\frac{\pi}{2}$, the positivity of $x_{m}$ follows directly from the coordinate construction. Hence, in all cases with $\theta_{1}>\theta_{2}$ and $[\![\gamma]\!]<0$, we have
    \begin{align}
        x_{m}>0
    \end{align}
    
    Theorem~\ref{thm:relation} together with Lemma~\ref{lem:pos} ensures that $P_{0}$ can be expressed as a function of $u_{m}$. It therefore remains to represent the total enclosed volume in terms of $u_{m}$. To this end, we decompose the total volume into five parts.
    
    The first part is the region enclosed by $y=u_{1}$, $x=x_{m}$, and the surface curve $(x(\psi),y(\psi))$ for $\psi\in(0,\psi_{1})$.
The second part is the region enclosed by $y=u_{2}$, $x=x_{m}$, and the surface curve $(x(\psi),y(\psi))$ for $\psi\in(\psi_{2},0)$.
The third part is the triangular region enclosed by $y=u_{1}$, $x=x_{m}$, and the left wall.
The fourth part is the triangular region enclosed by $y=u_{2}$, $x=0$, and the right wall.
The fifth part is the rectangular region enclosed by $y=u_{1}$, $y=u_{2}$, $x=0$, and $x=x_{m}$.

We denote these regions by $\Omega_{1}$ through $\Omega_{5}$, and their corresponding areas by $\mathscr{V}_{1}$ through $\mathscr{V}_{5}$, respectively; see Figure~\ref{figure6}.

    \begin{figure}
        \centering
        \includegraphics[width=0.9\linewidth=2]{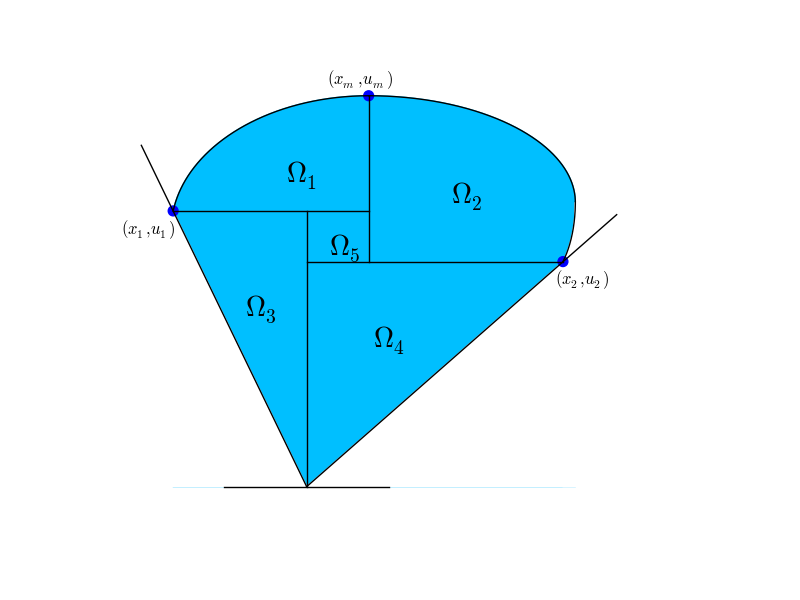}
        \caption{\label{figure6}}
     \end{figure}

   To proceed further, we note that Theorem~\ref{thm:relation} distinguishes between the cases $\theta_{1}=\frac{\pi}{2}$ and $\theta_{1}\neq\frac{\pi}{2}$. Accordingly, we divide the subsequent volume computation into two separate cases.

   \subsubsection{\textbf{The case $\boldsymbol{\theta_{1}\neq \frac{\pi}{2}}$}}

Throughout this subsection, we assume that $\theta_{1}\neq \frac{\pi}{2}$. Under this assumption, the volume $\mathscr{V}_{1}$ defined above can be expressed as

    \begin{align}
         \mathscr{V}_{1}&=  \int_{u_1}^{u_m} r(\psi)dy=\int_{0}^{\psi_1}r(\psi)\frac{dy}{d\psi}d\psi=\int_{0}^{\psi_1}r(\psi)\frac{\sin\psi}{gy}d\psi \notag \\
         &=\int_{0}^{\psi_1}  \int_{0}^{\psi}\frac{\sigma \cos\gamma d\gamma}{g\sqrt{u_m^{2}+\frac{2\sigma}{g}(1-\cos\gamma)}} \frac{\sigma\sin\psi}{g\sqrt{u_{m}^{2}+\frac{2\sigma}{g}(1-\cos\psi)}}d\psi.\label{equ:2.2.29}
    \end{align}

    \noindent where $r(\psi)$ is defined by \eqref{equ:2.2.100}.
    
    By the same reasoning, the volume $\mathscr{V}_{2}$ is given by
    
    \begin{equation}{\label{equ:2.2.30}}
        \mathscr{V}_{2}=\int_{\psi_2}^{0}  \int_{0}^{\psi}\frac{\sigma\cos\gamma d\gamma}{g\sqrt{u_m^{2}+\frac{2\sigma}{g}(1-\cos\gamma)}} \frac{\sigma\sin\psi}{g\sqrt{u_{m}^{2}+\frac{2\sigma}{g}(1-\cos\psi)}}d\psi
    \end{equation}

     Moreover, using the definitions of $r_{1}$, $r_{2}$, $x_{m}$, $u_{1}$, and $u_{2}$, we can directly compute $\mathscr{V}_{3}$–$\mathscr{V}_{5}$ by applying the standard area formulas for triangles and rectangles:

    \begin{align}
        \mathscr{V}_{3}&=\frac{1}{2}(r_1-x_m)^{2}\tan\theta_1 \label{equ:2.2.31}\\
        \mathscr{V}_{4}&=\frac{1}{2}(r_2+x_{m})^{2}\tan\theta_2 \label{equ:2.2.32}\\
        \mathscr{V}_{5}&=x_{m}(u_1-u_2)\label{equ:2.2.33}.
    \end{align}

    \noindent Then it remains to establish the dependence of the total volume $\mathscr{V}=\mathscr{V}_1+\mathscr{V}_2+\mathscr{V}_3+\mathscr{V}_4+\mathscr{V}_5$ on the parameter $u_{m}$.

    \textbf{Remark}: Although Figure~\ref{figure6} illustrates only the case $\theta_{1}\in\bigl(0,\tfrac{\pi}{2}\bigr)$, the representation of the total volume $\mathscr{V}$ remains valid and takes the same form when $\theta_{1}\in\bigl(\tfrac{\pi}{2},\pi\bigr)$.
    
    \begin{theorem}{\label{thm:v_1}}

    The two volume functions $\mathscr{V}_1$ and $\mathscr{V}_2$ are both increasing functions of $u_{m}$. Moreover, as $u_{m}\rightarrow 0$, we have $\mathscr{V}_1\to +\infty$ and $\mathscr{V}_2\rightarrow +\infty$, while as $u_m\rightarrow -\infty$, both $\mathscr{V}_1$ and $\mathscr{V}_{2}$ converge to $0$.
    
    \end{theorem}

    \begin{proof}

        Since $\mathscr{V}_1$ and $\mathscr{V}_2$ have analogous integral representations, it suffices to establish the stated properties for $\mathscr{V}_{1}$.

        From the explicit formula \eqref{equ:2.2.29}, we immediately see that $\lim_{u_{m\rightarrow -\infty}}\mathscr{V}_1=0$, since the denominator of the integrant diverges as $u_{m}\to -\infty$. 

        We then show the monotonicity of the volume $\mathscr{V}_{1}$. Differentiating both sides of the equation \eqref{equ:2.2.29} with respect to $u_{m}$, we obtain the following expression

        \begin{align}
             \frac{d\mathscr{V}_{1}}{du_{m}}=&\sigma^{2} u_{m}\int_{0}^{\psi_1} \int_{0}^{\psi}\frac{\cos\tau d\tau}{g(u_m^{2}+\frac{2\sigma}{g}(1-\cos\tau))^{\frac{3}{2}}}\frac{\sin\psi}{-g\sqrt{u_{m}^{2}+\frac{2\sigma}{g}(1-\cos\psi)}}d\psi\notag\\
             &+\sigma^{2}u_{m}\int_{0}^{\psi_1} \int_{0}^{\psi}\frac{\cos\tau d\tau}{g(u_m^{2}+\frac{2\sigma}{g}(1-\cos\tau))^{\frac{1}{2}}}\frac{\sin\psi}{-g({u_{m}^{2}+\frac{2\sigma}{g}(1-\cos\psi)})^{\frac{3}{2}}}d\psi \label{equ:2.2.34}.
        \end{align}

        To analyze the sign of the first term on the right-hand side of \eqref{equ:2.2.34}, we introduce the auxiliary function

        \begin{equation*}
            a(\psi)=\int_{0}^{\psi} \frac{\cos\tau d\tau}{g(u_{m}^{2}+\frac{2\sigma}{g}(1-\cos\tau))^{\frac{3}{2}}}
        \end{equation*}

       \noindent We will show in Lemma~2.9 (proved after the main argument) that $a(\psi)>0$ for all $\psi\in(0,\psi_{1})$. Since $u_{m}<0$, it follows that $u_{m}a(\psi)<0$. 
       
       Similarly, Lemma~2.9 also implies that the function
        
        \begin{equation*}
            b(\psi):=\int_{0}^{\psi} \frac{\cos\tau d\tau}{g(u_{m}^{2}+\frac{2\sigma}{g}(1-\cos\tau))^{\frac{1}{2}}}
        \end{equation*}

        \noindent is positive for all $\psi\in(0,\psi_{1})$. Hence, $u_{m}b(\psi)<0$ as well.

        Using the definition of functions $a(\psi)$ and $b(\psi)$, we rewrite the equation \eqref{equ:2.2.34} as follows:

        \begin{equation*}
            \frac{d\mathscr{V}_{1}}{du_{m}}=\sigma^{2}\int_{0}^{\psi_1}u_{m}a(\psi)\frac{\sin\psi}{-g\sqrt{u_{m}^{2}+\frac{2\sigma}{g}(1-\cos\psi)}}d\psi+\sigma^{2}\int_{0}^{\psi_1} u_{m}b(\psi) \frac{\sin\psi}{-g(u_{m}^{2}+\frac{2\sigma}{g}(1-\cos\psi))^{\frac{3}{2}}}d\psi
        \end{equation*}

         \noindent Since $\sin\psi>0$ for $\psi\in(0,\pi)$ and both $u_{m}a(\psi)$ and $u_{m}b(\psi)$ are negative, each integrand is non-negative. Consequently, $\frac{d\mathscr{V}_{1}}{du_{m}}$ is non-negative, which shows that $\mathscr{V}_1$ is a monotone increasing function of $u_m$.

         It remains to show the blow-up behavior of $\mathscr{V}_{1}$. Using the expression for $\mathscr{V}_1$, we have: 

        \begin{equation*}
            \mathscr{V}_{1}=\sigma^{2}\int_{0}^{\psi_1}\int_{0}^{\psi} \frac{\cos\tau d\tau}{-g\sqrt{u_{m}^{2}+\frac{2\sigma}{g}(1-\cos\tau)}} \frac{\sin\psi}{-g\sqrt{u_{m}^{2}+\frac{2\sigma}{g}(1-\cos\psi)}} d\psi
        \end{equation*}

        \noindent When $u_{m}= 0$, this expression reduces to

        \begin{equation*}
            \mathscr{V}_{1}(0)=\sigma^{2}\int_{0}^{\psi_1}\int_{0}^{\psi} \frac{\cos\tau d\tau}{-g\sqrt{\frac{2\sigma}{g}(1-\cos\tau)}} \frac{\sin\psi}{-g\sqrt{\frac{2\sigma}{g}(1-\cos\psi)}} d\psi
        \end{equation*}

        \noindent To analyze the behavior of the integrand near $(\tau,\psi)=(0,0)$, we use the Taylor expansions

        \begin{equation*}
            \frac{\cos\tau}{-g\sqrt{\frac{2\sigma}{g}(1-\cos\tau)}} \frac{\sin\psi}{-g\sqrt{\frac{2\sigma}{g}(1-\cos\psi)}}\sim \frac{1}{4g}\frac{1}{\sin\gamma}\frac{\sin\psi}{\sin\psi}\sim\frac{1}{4g}\frac{1}{\gamma}
        \end{equation*}

        \noindent Consequently,

        \begin{equation*}
            \int_{0}^{\psi}\frac{1}{4g} \frac{1}{\gamma}d\gamma=+\infty,
        \end{equation*}

        \noindent which implies that $\mathscr{V}_1(0)= +\infty$.
        
        Since $\mathscr{V}_{1}$ is monotone increasing in $u_{m}$, Fatou’s lemma implies 
        \begin{align}
            \lim_{u_{m}\rightarrow 0^{-}}\mathscr{V}_{1}=\infty
        \end{align}

       The same argument applies to $\mathscr{V}_{2}$, which therefore satisfies identical monotonicity and blow-up properties.
    \end{proof}

    \textbf{Remark}: In fact, the quantity $-u_{m}a(\psi_1)$ appearing in the proof of the theorem above represents the derivative of $r_1$ with respect to $u_{m}$, implying that $r_1$ is a monotone increasing function of $u_m$. Similarly, for the same reason, $r_2$ is a monotone increasing function of $u_m$. Moreover, the proof of Theorem 2.7 further shows that

    \begin{equation}{\label{equ:2.2.111}}
    \lim_{u_{m}\rightarrow 0} r_i=+\infty,
    \end{equation}

    \noindent and

    \begin{equation}{\label{equ:2.2.112}}
        \lim_{u_{m}\rightarrow -\infty} r_i=0
    \end{equation}

    \noindent for $i=1,2$.
    
    We now establish analogous asymptotic properties for $\mathscr{V}_{3}$, $\mathscr{V}_{4}$, and $\mathscr{V}_{5}$. This leads to the following result.

    \begin{theorem}{\label{thm:v_3}}

        $\mathscr{V}_{3}$, $\mathscr{V}_{4}$ and $\mathscr{V}_{5}$ are defined in \eqref{equ:2.2.31}, \eqref{equ:2.2.32} and \eqref{equ:2.2.33}, respectively. As $u_{m}\rightarrow -\infty$ , we have $\mathscr{V}_{i} \rightarrow 0$ for $i=3,4,5$. 
    \end{theorem}

    \begin{proof}

          Substituting \eqref{equ:2.2.11} into the equation \eqref{equ:2.2.31}, we obtain the following expression for $\mathscr{V}_{3}$

        \begin{equation}\label{equ:2.2.35}
            \mathscr{V}_{3}=\frac{1}{2}(\frac{r_1\tan\theta_2+r_2\tan\theta_2+(u_1-u_2)}{\tan\theta_1+\tan\theta_{2}})^{2}\tan\theta_1
        \end{equation}

         From equation \eqref{equ:2.2.111} and equation \eqref{equ:2.2.112}, we know that $r_1\rightarrow 0$ and $r_{2}\rightarrow 0$ when $u_m\rightarrow -\infty$. Moreover, note that

        \begin{align*}
             u_1-u_2&=-\sqrt{u_{m}^{2}+\frac{2\sigma}{g}(1-\cos\psi(x_1))}+\sqrt{u_m^{2}+\frac{2\sigma}{g}(1-\cos\psi(x_2))}\\
              &=\frac{2\sigma}{g}\frac{\cos\psi(x_1)-\cos\psi(x_2)}{\sqrt{u_{m}^{2}+\frac{2\sigma}{g}(1-\cos\psi(x_1))}+\sqrt{u_m^{2}+\frac{2\sigma}{g}(1-\cos\psi(x_2))}}.
        \end{align*}
    
        \noindent It follows immediately that $u_1-u_2 \to 0$ as $u_m \to -\infty$. Returning to equation \eqref{equ:2.2.35}, we conclude the following asymptotic behavior for $\mathscr{V}_{3}$ as $u_{m}\rightarrow -\infty$

        \begin{equation*}
           \lim_{u_m\rightarrow \infty} \mathscr{V}_{3}=\lim_{u_{m}\rightarrow -\infty}\frac{1}{2}(\frac{r_1\tan\theta_2+r_2\tan\theta_2+(u_1-u_2)}{\tan\theta_1+\tan\theta_{2}})^{2}\tan\theta_1\rightarrow 0.
        \end{equation*}

         Similarly, substituting \eqref{equ:2.2.11} into \eqref{equ:2.2.32}, we obtain 

        \begin{equation*}
            \mathscr{V}_{4}=\frac{1}{2}(\frac{r_1\tan\theta_1+r_2\tan\theta_{2}-(u_1-u_2)}{\tan\theta_1+\tan\theta_2})^{2}\tan\theta_2.
       \end{equation*} 
     By the same reasoning used for $\mathscr{V}_{3}$, it follows that $\mathscr{V}_{4}\rightarrow 0$ as $u_{m}\rightarrow -\infty$.

        From the preceding computations, we have already shown that $u_1-u_2\rightarrow 0$ when $u_{m}\rightarrow -\infty$.  Combining this property with the relation \eqref{equ:2.2.112}, we deduce that

        \begin{equation*}
            x_{m}=\frac{r_1\tan\theta_1-r_2\tan\theta_2-(u_1-u_2)}{\tan\theta_1+\tan\theta_2}\rightarrow 0~\operatorname{when}~u_{m}\rightarrow -\infty
        \end{equation*}

        \noindent Consequently, we further obtain

        \begin{equation*}
            \lim_{u_{m}\rightarrow -\infty}\mathscr{V}_{5}=\lim_{u_{m}\rightarrow -\infty}(u_1-u_2)x_{m}=0
        \end{equation*}

        \noindent This completes the proof of the theorem.
    \end{proof}

   We now prove the following lemma, which is used in the proofs of Theorems \ref{thm:v_1} and \ref{thm:v_3}.
   
    \begin{lemma}

     For any $\psi\in(0,\psi_1)$, the following two inequalities hold:

     \begin{equation}{\label{equ:2.2.36}}
         a(\psi)=\int_{0}^{\psi}\frac{\cos\varphi d\varphi }{(u_m^{2}+\frac{2\sigma}{g}(1-\cos\varphi))^{\frac{3}{2}}}\geq \int_{0}^{\psi}\frac{\cos\varphi d\varphi}{(u_m^{2}+\frac{2\sigma}{g}(1-\cos\psi))^{\frac{3}{2}}}\geq 0,
     \end{equation}

     \noindent and

     \begin{equation*}
         b(\psi)=\int_{0}^{\psi}\frac{\cos\varphi d\varphi }{(u_m^{2}+\frac{2\sigma}{g}(1-\cos\varphi))^{\frac{1}{2}}}\geq \int_{0}^{\psi}\frac{\cos\varphi d\varphi}{(u_m^{2}+\frac{2\sigma}{g}(1-\cos\psi))^{\frac{1}{2}}}\geq 0.
     \end{equation*}
     \end{lemma}

    \begin{proof}

    To prove that the integral is non-negative, it suffices to show that

    \begin{equation*}
        \int_{0}^{\psi}\frac{\cos\varphi d\varphi}{(u_m^{2}+\frac{2\sigma}{g}(1-\cos\varphi))^{\frac{3}{2}}}=\frac{\sin\psi}{(u_m^{2}+\frac{2\sigma}{g}(1-\cos\psi))^{\frac{3}{2}}}\geq 0,
    \end{equation*}

    \noindent where the last inequality follows from $\psi\in(0,\pi)$. Therefore, it remains to prove that

    \begin{equation}{\label{equ:2.2.37}}
         \int_{0}^{\psi}\frac{\cos\varphi d\varphi }{(u_m^{2}+\frac{2\sigma}{g}(1-\cos\varphi))^{\frac{3}{2}}}\geq \int_{0}^{\psi}\frac{\cos\varphi d\varphi}{(u_m^{2}+\frac{2\sigma}{g}(1-\cos\psi))^{\frac{3}{2}}}.
    \end{equation}

    We note that the inequality \eqref{equ:2.2.37} has the same structure as the inequality \eqref{equ:2.2.22} which was established in Lemma \ref{lem:1}. Indeed, by applying the same idea of splitting the integral into three parts and performing analogous computations as in Lemma \ref{lem:1}, we obtain the desired result. The proof for the function $b(\psi)$ follows by an entirely analogous argument.
    
    \end{proof}
    
    Finally, we establish the following existence result for the volume function $\mathscr{V}(u_{m})$.
    
    \begin{theorem}{\label{thm:exist_1}}

    Suppose that $\theta_1\neq \frac{\pi}{2}$.
        For any given total volume $V\geq 0$, there exists at least one value $u_m$ such that $\mathscr{V}(u_m)=\mathscr{V}_{1}(u_m)+\mathscr{V}_{2}(u_m)+\mathscr{V}_{3}(u_m)+\mathscr{V}_{4}(u_m)+\mathscr{V}_{5}(u_m)= V$. Consequently, there exists at least one solution function of the equation system \eqref{equ:1.2.600} subject to the conserved volume $V$.
    \end{theorem}

    \begin{proof}

        When $u_{m}\rightarrow 0$, Theorem \ref{thm:v_1} implies that $\mathscr{V}(u_{m})\geq \mathscr{V}_1+\mathscr{V}_2\rightarrow +\infty$.

        On the other hand, when $u_m\rightarrow -\infty$, using Theorem \ref{thm:v_1} and Theorem \ref{thm:v_3}, we have $\mathscr{V}_{i}(u_{m})\rightarrow 0,1\leq i\leq 5$. Hence, we conclude that

        \begin{equation*}
            \lim_{u_{m}\rightarrow +\infty}\mathscr{V}(u_{m})=\lim_{u_{m}\rightarrow +\infty}(\mathscr{V}_{1}+\mathscr{V}_{2}+\mathscr{V}_{3}+\mathscr{V}_{4}+\mathscr{V}_{5})=0.
        \end{equation*}

        Since $\mathscr{V}$ is a continuous function of $u_m$ with respect to $u_{m}$, there is at least one $u_m<0$ such that $\mathscr{V}(u_m)=V$
    \end{proof}

      Using Theorem \ref{thm:exist_1}, we have the existence of the solution function to \eqref{equ:1.2.600}. However,  it is not clear whether this solution is unique. Although $\mathscr{V}_{1}$ and $\mathscr{V}_{2}$ are monotone increasing functions of $u_{m}$, and although $u_1-u_2$, $r_1$, and $r_{2}$ are all monotone increasing functions of $u_{m}$ by the arguments in the proof of Theorem \ref{thm:v_1} and \ref{thm:v_3}, the monotonicity of $x_{m}$, $r_1-x_m$, and $r_{2}+x_{m}$ are still unclear. As a result, the monotonicity of the total volume function $\mathscr{V}(u_{m})$ cannot be concluded directly from the previous estimates.
      
     To address this difficulty, we next analyze the dependence of these geometric quantities on $u_{m}$ and establish the monotonicity of the full volume functional. This leads to the following theorem.

     \begin{theorem}
        Suppose that $[\![\gamma]\!]<0$, and that $\theta_1\geq \theta_2$. $ \mathscr{V}_{c}=\mathscr{V}_{3}+\mathscr{V}_{4}+\mathscr{V}_{5}$ is a monotone increasing function with respect to $u_{m}$.
    \end{theorem}

    \begin{proof}

    Applying \eqref{equ:2.2.31},\eqref{equ:2.2.32} and \eqref{equ:2.2.33} to the definition of $\mathscr{V}_{c}$, we obtain

    \begin{equation*}
        \mathscr{V}_{c}=\mathscr{V}_{3}+\mathscr{V}_{4}+\mathscr{V}_{5}=\frac{1}{2}(r_1-x_{m})^{2}\tan\theta_{1}+\frac{1}{2}(r_{2}+x_{m})^{2}\tan\theta_{2}+x_{m}(u_1-u_2).
    \end{equation*}

    Substituting equation \eqref{equ:2.2.11} into the equation above, we have

    \begin{align}
        \mathscr{V}_{c}=&\frac{1}{2}(r_{1}-\frac{r_1\tan\theta_1-r_2\tan\theta_2-(u_1-u_2)}{\tan\theta_1+\tan\theta_{2}})^{2}\tan\theta_1 \notag\\
        &+\frac{1}{2}(r_{2}+\frac{r_1\tan\theta_1-r_2\tan\theta_2-(u_1-u_2)}{\tan\theta_1+\tan\theta_{2}})^{2}\tan\theta_2 \notag \\
         &+(\frac{r_1\tan\theta_1-r_2\tan\theta_2-(u_1-u_2)}{\tan\theta_1+\tan\theta_{2}})(u_1-u_2)\notag\\
         =&\frac{1}{2}\frac{(r_1\tan\theta_2+r_2\tan\theta_2+(u_1-u_2))^{2}}{(\tan\theta_1+\tan\theta_2)^{2}}\tan\theta_1+\frac{1}{2}(\frac{(r_{2}\tan\theta_1+r_{1}\tan\theta_1-(u_1-u_2))^{2}}{(\tan\theta_1+\tan\theta_2)^{2}})\tan\theta_2 \notag\\
        & +(\frac{r_1\tan\theta_1-r_2\tan\theta_2-(u_1-u_2)}{\tan\theta_1+\tan\theta_{2}})(u_1-u_2) \label{equ:2.2.113}
    \end{align}

    \noindent Rearranging terms on the right hand side of \eqref{equ:2.2.113}, we obtain

    \begin{align}
         \eqref{equ:2.2.113}=&\frac{1}{2}\frac{(r_1\tan\theta_2+r_2\tan\theta_2)^{2}}{(\tan\theta_1+\tan\theta_2)^{2}}\tan\theta_1+\frac{1}{2}\frac{(r_1\tan\theta_1+r_2\tan\theta_1)^{2}}{(\tan\theta_1+\tan\theta_2)^{2}}\tan\theta_2\notag\\
        &+(u_1-u_2)(r_1+r_2)\frac{\tan\theta_1\tan\theta_2}{(\tan\theta_1+\tan\theta_2)^{2}}-(u_1-u_2)(r_1+r_2)\frac{\tan\theta_1\tan\theta_2}{(\tan\theta_1+\tan\theta_2)^{2}}\notag\\
       & +\frac{1}{2}\frac{(u_1-u_2)^{2}\tan\theta_1}{(\tan\theta_1+\tan\theta_2)^{2}}+\frac{1}{2}\frac{(u_1-u_2)^{2}\tan\theta_2}{(\tan\theta_1+\tan\theta_2)^{2}}\notag\\
       & +(\frac{r_1\tan\theta_1-r_2\tan\theta_2-(u_1-u_2)}{\tan\theta_1+\tan\theta_{2}})(u_1-u_2) \notag \\
       =&\frac{1}{2}\frac{(r_1+r_2)^{2}}{\tan\theta_1+\tan\theta_2}\tan\theta_1\tan\theta_2+\frac{1}{2}\frac{(u_1-u_2)^{2}}{(\tan\theta_1+\tan\theta_2)}\notag\\
       & +(u_1-u_2)(\frac{r_1\tan\theta_1-r_2\tan\theta_2}{\tan\theta_1+\tan\theta_2})-\frac{(u_1-u_2)^{2}}{\tan\theta_1+\tan\theta_2}\notag\\
        =&\frac{1}{2}\frac{(r_1+r_2)^{2}}{\tan\theta_1+\tan\theta_2}\tan\theta_1\tan\theta_2+(u_1-u_2)(\frac{r_1\tan\theta_1-r_2\tan\theta_2}{\tan\theta_1+\tan\theta_2})-\frac{1}{2}\frac{(u_1-u_2)^{2}}{\tan\theta_1+\tan\theta_2} \label{equ:2.2.38}
    \end{align}

   To proceed to the next step, we introduce the following result, whose proof will be given after the completion of the main proof of this theorem:
    
    \begin{equation}{\label{equ:2.2.39}}
      \frac{1}{\tan\theta_1+\tan\theta_2} (r_1\tan\theta_1-(u_1-u_2)) ~~~\operatorname{is~ a~ positive~ monotone ~increasing ~function ~with~respect~to~ u_{m}}.
    \end{equation}

    \noindent Using this result, we can rewrite equation \eqref{equ:2.2.38} as:

    \begin{align}
         \eqref{equ:2.2.38}=&\frac{1}{2}\frac{r_1^{2}+r_2^{2}}{\tan\theta_1+\tan\theta_2}\tan\theta_1\tan\theta_2+\frac{r_1r_2}{\tan\theta_1+\tan\theta_2}\tan\theta_1\tan\theta_2-(u_1-u_2)\frac{r_2\tan\theta_2}{\tan\theta_1+\tan\theta_2} \notag\\
         &+(u_1-u_2)\frac{r_1\tan\theta_1}{\tan\theta_1+\tan\theta_2}-\frac{1}{2}\frac{(u_1-u_2)^{2}}{\tan\theta_1+\tan\theta_2} \notag\\
         =&\frac{1}{2}\frac{r_1^{2}+r_2^{2}}{\tan\theta_1+\tan\theta_2}\tan\theta_1\tan\theta_2+\frac{r_2\tan\theta_2}{\tan\theta_1+\tan\theta_2}(r_1\tan\theta_1-(u_1-u_2))\notag \\
         &+\frac{u_1-u_2}{\tan\theta_1+\tan\theta_2}(r_1\tan\theta_1-\frac{1}{2}(u_1-u_2)) \label{equ:2.2.40}
    \end{align}

   \noindent For the first term in \eqref{equ:2.2.40}, note that both $r_1^{2}$ and $r_2^{2}$ are monotone increasing functions of $u_m$ by the Remark following Theorem \ref{thm:v_1}, \eqref{equ:2.2.111} and \eqref{equ:2.2.112}. Moreover, we will show in Lemma 2.12 that

   \begin{equation*}
       \frac{\tan\theta_{1}\tan\theta_{2}}{\tan\theta_{1}+\tan\theta_{2}}>0,
   \end{equation*}
   \noindent which implies that
   \begin{equation*}
       \frac{1}{2}(r_{1}^{2}+r_{2}^{2})\frac{\tan\theta_{1}\tan\theta_{2}}{\tan\theta_{1}+\tan\theta_{2}}
   \end{equation*}
   \noindent is a monotone increasing function of $u_{m}$.
   
   For the second term in \eqref{equ:2.2.40}, we already know that $r_{2}$ is increasing monotonically from the remark following Theorem \ref{thm:v_1}. Moreover, function 
   \begin{align}
   \frac{1}{\tan\theta_1+\tan\theta_1}(r_1\tan\theta_1-(u_1-u_2))
   \end{align}
   is a positive, monotone increasing function of $u_{m}$ by Lemma 2.12. Therefore,

   \begin{equation*}
       \frac{r_2\tan\theta_2}{\tan\theta_1+\tan\theta_2}(r_1\tan\theta_1-(u_1-u_2))
   \end{equation*}

   \noindent is a monotone increasing function of $u_{m}$, since the product of two positive, increasing function remains monotonically increasing.
   
  Finally, for the third term in \eqref{equ:2.2.40}, note that $u_1-u_2$ is positive and monotone increasing, and

   \begin{equation*}
       \frac{1}{\tan\theta_1+\tan\theta_2}(r_1\tan\theta_1-\frac{1}{2}(u_1-u_2))
   \end{equation*}
   
   \noindent is also positive and monotone increasing by Lemma 2.12. Hence, the third term is monotone increasing as well.

    Combining the above discussion for all three terms, we conclude that \eqref{equ:2.2.40} is a monotone increasing function of $u_m$. Consequently, $\mathscr{V}_{c}$ is a monotone increasing function of $u_m$.
    \end{proof}

    To complete the proof of the preceding theorem, it remains to establish the following lemma.

   \begin{lemma}

    Assume that $\theta_1\geq\theta_2$ and $[\![\gamma]\!]<0$. Then the functions $g_1$ and $g_2$ defined by
    
   \begin{equation}{\label{equ:2.2.41}}
        g_1(u_m)=\frac{1}{\tan\theta_1+\tan\theta_2}(r_1\tan\theta_1-(u_1-u_2)),
    \end{equation}

    \begin{equation}{\label{equ:2.2.42}}
        g_2(u_m)=\frac{1}{\tan\theta_1+\tan\theta_2}(r_1\tan\theta_1-\frac{1}{2}(u_1-u_2)),
    \end{equation}

    \noindent and

    \begin{equation}{\label{equ:2.2.115}}
        u_1(u_{m})-u_2(u_{m})
    \end{equation}
    
    \noindent are all positive and monotone increasing functions with respect to $u_{m}$.

   \end{lemma}

  \begin{proof}

  We divide the analysis into two cases: The first correspond to $\theta_1\in(0,\frac{\pi}{2})$, and the second to $\theta_1\in(\frac{\pi}{2},\pi)$.
  
  \textbf{Case 1 $\theta_1\in (0,\frac{\pi}{2})$}: 
  
  Under the assumption $\theta_1\in (\theta_2,\frac{\pi}{2})\subseteq(0,\frac{\pi}{2})$, we have $\tan\theta_1>0$ and $\tan\theta_2>0$. Therefore, $\frac{\tan\theta_1}{\tan\theta_1+\tan\theta_2}>0$. We divide the proof into three steps.

  \textbf{Step 1}: The function \eqref{equ:2.2.115} is positive and monotonically increasing. 
  
 Using the definition of $u_{1}$ and $u_{2}$ given by \eqref{equ:2.2.3} and \eqref{equ:2.2.4}, respectively, we compute

  \begin{align}
      \frac{d(u_1-u_2)}{du_{m}}=&\frac{d(-\sqrt{u_{m}^{2}+\frac{2\sigma}{g}(1-\cos\psi_1)}+\sqrt{u_{m}^{2}+\frac{2\sigma}{g}(1-\cos\psi_2)})}{du_{m}} \notag\\
      =&\frac{u_{m}}{\sqrt{u_{m}^{2}+\frac{2\sigma}{g}(1-\cos\psi_{2})}}-\frac{u_{m}}{\sqrt{u_{m}^{2}+\frac{2\sigma}{g}(1-\cos\psi_1)}}\notag \\
       =&u_{m}\frac{2\sigma(\cos\psi_2-\cos\psi_1)}{(\sqrt{u_{m}^{2}+\frac{2\sigma}{g}(1-\cos\psi_{2})}+\sqrt{u_{m}^{2}+\frac{2\sigma}{g}(1-\cos\psi_{1})})}\times \notag\\
       &\frac{1}{g\sqrt{u_{m}^{2}+\frac{2\sigma}{g}(1-\cos\psi_{2})}\sqrt{u_{m}^{2}+\frac{2\sigma}{g}(1-\cos\psi_1)}} \label{equ:2.2.46}.
  \end{align}
  
   Since $\theta_1\in(\theta_2,\frac{\pi}{2})$ and $\gamma\in (0,\frac{\pi}{2})$, we have

  \begin{equation}{\label{equ:2.2.47}}
  \cos\psi_2=\cos(\gamma+\frac{\pi}{2}-\theta_2)=\sin (\theta_2-\gamma)<\sin(\theta_1-\gamma)=\cos(\gamma+\frac{\pi}{2}-\theta_1)=\cos\psi_1,
  \end{equation}
  
  \noindent where we use the definition of $\psi_1$ and $\psi_2$ given by \eqref{equ:1.2.8} and \eqref{equ:1.2.9}. Therefore, we have:

  \begin{equation}{\label{equ:2.2.116}}
      u_m(\cos\psi_2-\cos\psi_1)\geq 0
  \end{equation}

  \noindent  Substituting \eqref{equ:2.2.116} into the equation \eqref{equ:2.2.46}, we obtain

  \begin{equation*}
      \eqref{equ:2.2.46}\geq 0,
  \end{equation*}

  \noindent which completes the proof of the first step. 
  
  In fact, the same argument works as well for Case 2. The only modification required is to use the condition $\theta_1<\pi-\theta_2$ to establish the inequality \eqref{equ:2.2.47}.

  \textbf{Step 2}: In this step, we compute the first derivative of $g_1$ and $g_2$ with respect to $u_m$.
  
  Taking the derivative of function $g_1$ with respect to $u_{m}$, we obtain

  \begin{align}
      \frac{dg_1(u_{m})}{du_{m}}&=\frac{1}{\tan\theta_1+\tan\theta_2}\frac{d(r_1\tan\theta_1-(u_1-u_2))}{du_{m}} \notag\\
      &=\frac{\tan\theta_1}{\tan\theta_1+\tan\theta_2}\frac{d}{du_m}r_1-\frac{1}{\tan\theta_1+\tan\theta_2}\frac{d(u_1-u_2)}{du_m} \label{equ:2.2.43}
  \end{align}

   \noindent  Using the equation \eqref{equ:2.2.100}, we compute the first term in \eqref{equ:2.2.43} as

  \begin{align}
      \frac{dr_1\tan\theta_1}{du_{m}}&=\frac{d}{du_{m}}\tan\theta_1\int_{0}^{\psi_1}\frac{\sigma\cos\psi d\psi}{g\sqrt{u_{m}^{2}+\frac{2\sigma}{g}(1-\cos\psi)}}\notag \\
      &=-u_{m}\tan\theta_1\int_{0}^{\psi_1} \frac{\sigma\cos\psi d\psi}{{g(u_{m}^{2}+\frac{2\sigma}{g}(1-\cos\psi))^{\frac{3}{2}}}} \label{equ:2.2.44}.
  \end{align}

  \noindent Since $u_{m}<0$ and the integrand in equation \eqref{equ:2.2.44} is positive by Lemma 2.9, this derivative is positive. Applying Lemma 2.9 again yields the estimate
  
  \begin{equation}{\label{equ:2.2.45}}
      \frac{dr_1\tan\theta_1}{du_{m}}\geq -\sigma u_{m}\tan\theta_1\int_{0}^{\psi_1}\frac{\cos\psi d\psi}{{g(u_{m}^{2}+\frac{2\sigma}{g}(1-\cos\psi_1))^{\frac{3}{2}}}}
  \end{equation}

   For the second term in \eqref{equ:2.2.43}, we use the relation \eqref{equ:2.2.47} again in the equation \eqref{equ:2.2.46} to obtain

  \begin{equation}{\label{equ:2.2.48}}
       \frac{d(u_1-u_2)}{du_{m}}\leq u_m\frac{2\sigma}{g}\frac{\cos\psi_2-\cos\psi_1}{2(u_{m}^{2}+\frac{2\sigma}{g}(1-\cos\psi_1))^{\frac{3}{2}}}=\frac{u_{m}}{g}\frac{\cos\psi_2-\cos\psi_1}{(u_m^{2}+\frac{2\sigma}{g}(1-\cos\psi_1))^{\frac{3}{2}}}
  \end{equation}

    Combining equation \eqref{equ:2.2.48} and $\eqref{equ:2.2.44}$ together, we obtain

   \begin{align}
       \frac{dg_{1}(u_{m})}{du_{m}}&=\frac{1}{\tan\theta_1+\tan\theta_2}\frac{d(r_1\tan\theta_1-(u_1-u_2))}{du_{m}}\notag\\
       &\geq \frac{\sigma}{\tan\theta_1+\tan\theta_2}(-u_{m}\tan\theta_1\int_{0}^{\psi_1}\frac{\cos\psi d\psi}{g(u_m^{2}+\frac{2\sigma}{g}(1-\cos\psi_1))^{\frac{3}{2}}}-\frac{u_{m}}{g}\frac{\cos\psi_2-\cos\psi_1}{(u_m^{2}+\frac{2\sigma}{g}(1-\cos\psi_1))^{\frac{3}{2}}}) \notag\\
       &=\frac{\sigma}{\tan\theta_1+\tan\theta_2}(-u_{m}\tan\theta_1\sin\psi_1\frac{1}{g}\frac{1}{(u_m^{2}+\frac{2\sigma}{g}(1-\cos\psi_1))^{\frac{3}{2}}}-\frac{u_{m}}{g}\frac{\cos\psi_2-\cos\psi_1}{(u_m^{2}+\frac{2\sigma}{g}(1-\cos\psi_1))^{\frac{3}{2}}}) \notag\\
        &=\frac{\sigma}{\tan\theta_1+\tan\theta_2}(-\frac{u_m}{g}\frac{1}{(u_m^{2}+\frac{2\sigma}{g}(1-\cos\psi_1))^{\frac{3}{2}}}(\tan\theta_1\sin\psi_1-(\cos\psi_1-\cos\psi_2))) \label{equ:2.2.49}
   \end{align}
   
  \noindent Using the fact that $u_m<0$, we have:

  \begin{equation*}
  \frac{1}{\tan\theta_1+\tan\theta_2}(-\frac{u_m}{g}\frac{1}{(u_{m}^{2}+\frac{2\sigma}{g}(1-\cos\psi_1))^{\frac{3}{2}}})>0.
  \end{equation*}
  
  \noindent Therefore, to determine the sign of $\frac{dg_{1}(u_{m})}{du_{m}}$, it remains to determine the sign of the term $\tan\theta_1\sin\psi_1-(\cos\psi_1-\cos\psi_2)$ in \eqref{equ:2.2.49}, which only depends on the inclined angles and contact angles. When $\theta_{1}$ and $\theta_{2}$ are fixed, this term becomes a function of contact angle $\gamma$. We denote it by $k(\gamma)$.
  
  \textbf{Step 3}: In this step, we show that $k(\gamma)$ is positive. 
  
  Using the definition of $\psi_1$ and $\psi_2$ given in equation \eqref{equ:1.2.8} and equation \eqref{equ:1.2.9}, we compute

  \begin{align*}
      k(\gamma)&=\tan\theta_1\sin\psi_1-(\cos\psi_1-\cos\psi_2) \notag\\
       &=\tan\theta_1\sin(\frac{\pi}{2}-\theta_1+\gamma)-(\cos(\frac{\pi}{2}-\theta_1+\gamma)-\cos(\frac{\pi}{2}-\theta_2+\gamma)) \notag\\
        &=\tan\theta_1\cos(\theta_1-\gamma)-(\sin(\theta_1-\gamma)-\sin(\theta_2-\gamma))\notag \\
         &\geq \tan\theta_1\cos(\theta_1-\gamma)-(\sin(\theta_1-\gamma)+\sin(\gamma))=p(\gamma)
  \end{align*}
  
   Observe that when $\gamma=0$, we have

  \begin{equation*}
      p(0)=\sin\theta_1-\sin\theta_1+\sin0=0
  \end{equation*}

  \noindent Then we compute the derivative of function $p(\gamma)$ with respect to $\gamma$.

  \begin{align*}
       p'(\gamma)&=\tan\theta_1\sin(\theta_1-\gamma)+\cos(\theta_1-\gamma)-\cos\gamma \\
       &=\tan\theta_1(\sin\theta_1\cos\gamma-\cos\theta_1\sin\gamma)+\cos\theta_1\cos\gamma+\sin\theta_1\sin\gamma-\cos\gamma \\
        &=\frac{\sin^{2}\theta_1}{\cos\theta_1}\cos\gamma-\sin\theta_1\sin\gamma+\cos\theta_1\cos\gamma+\sin\theta_1\sin\gamma-\cos\gamma\\
         &=\frac{\sin^{2}\theta_1}{\cos\theta_1}\cos\gamma+\cos\theta_1\cos\gamma-\cos\gamma\\
         &=\frac{1}{\cos\theta_1}\cos\gamma-\cos\gamma\geq 0~for~\gamma\in (0,\frac{\pi}{2})
  \end{align*}

  \noindent The last inequality holds since $\gamma\in(0,\frac{\pi}{2})$ when $[\![\gamma]\!]<0$. Combining this with the fact that $p(0)=0$, we have

  \begin{equation*}
      k(\gamma)\geq p(\gamma)\geq p(0)=0,
  \end{equation*}

  \noindent which implies that \eqref{equ:2.2.41} is increasing monotonically.

  To verify relation \eqref{equ:2.2.42}, we observe that

  \begin{equation*}
      g_2(u_{m})=g_{1}(u_{m})+\frac{1}{2}\frac{1}{\tan\theta_1+\tan\theta_2}(u_1-u_2)
  \end{equation*}

  \noindent Since both $g_{1}(u_{m})$ and $u_1-u_2$ are increasing functions of $u_{m}$, it follows immediately that $g_2(u_m)$ is also a monotone increasing function of $u_m$.

  Finally, it remains to establish the positivity of $g_{1}$ and $g_{2}$. Using the following fact

  \begin{equation*}
  \lim_{u_m\rightarrow -\infty}r_1=\lim_{u_m\rightarrow -\infty}(u_1-u_2)=0,
  \end{equation*}
  we obtain
  \begin{equation}{\label{equ:2.2.200}}
      \lim_{u_m\rightarrow -\infty} g_1(u_m)=\lim_{u_m\rightarrow -\infty}g_2(u_m)=0.
  \end{equation}
  Then the positivity of $g_1$ and $g_2$ directly follow from \eqref{equ:2.2.200} and the monotonicity of $g_1$ and $g_2$. 

  \textbf{Remark}: The positivity can also be deduced from Theorem \ref{lem:pos}, which shows that $x_{m}>0$ when $[\![\gamma]\!]<0$. This fact implies that

  \begin{equation*}
      r_1\tan\theta_1-r_2\tan\theta_2-(u_1-u_2)
  \end{equation*}

  \noindent is positive. Consequently, $r_1\tan\theta_1-(u_1-u_2)$ should also be positive.
  
  \textbf{Case 2 $\theta_1\in (\frac{\pi}{2},\pi-\theta_2)$}:
  
  In this case, we observe that 
  \begin{equation*}
  \tan\theta_1 ~~\operatorname{and}~~ \tan\theta_1+\tan\theta_2
  \end{equation*}
  are both negative, under the assumption that $\frac{\pi}{2}<\theta_1<\pi-\theta_2$. Therefore, $\frac{\tan\theta_1}{\tan\theta_1+\tan\theta_2}$ is positive. However, the following term
  \begin{equation*}
  r_1\tan\theta_1-(u_1-u_2)
  \end{equation*}
  is negative since $u_1>u_2$, $r_1>0$ and $\tan\theta_1<0$. Hence, $g_1(u_m)$ should be positive and so is $g_{2}(u_m)$. 
  
   Using the Remark following Theorem 2.7 and Step 1 in Case 1, we know that both $r_1$ and $u_1-u_2$ are monotone increasing function of $u_m$. Consequently, 
   \begin{equation*}
   r_1\tan\theta_1-(u_1-u_2) ~\operatorname{and}~ r_1\tan\theta_1-\frac{1}{2}(u_1-u_2)
   \end{equation*}
   are monotonically decreasing. Since $\frac{1}{\tan\theta_1+\tan\theta_2}$ is negative, it follows that $g_1(u_m)$ and $g_{2}(u_m)$ are monotone increasing functions of $u_{m}$. This completes the proof of the Theorem.
  \end{proof}

  \subsubsection{ \textbf{The case $\theta_1=\frac{\pi}{2}$}}
  
  When $\theta_1=\frac{\pi}{2}$, we have $\mathscr{V}=\mathscr{V}_1+\mathscr{V}_2+\mathscr{V}_4+\mathscr{V}_5$ where $\mathscr{V}_{i}, ~i=1,2,4,5$ are defined as in Section 2.2.1. From Theorem \ref{thm:v_1}, we have the monotonicity of $\mathscr{V}_1$ and $\mathscr{V}_2$. Then it remains to show the monotonicity of $\mathscr{V}_{4}$ and $\mathscr{V}_{5}$

  \begin{theorem}
      When $\theta_1=\frac{\pi}{2}$ and $[\![\gamma]\!]<0$,  $\mathscr{V}_4$ and $\mathscr{V}_5$ are both monotone increasing functions of $u_m$.
  \end{theorem}

  \begin{proof}

      First we show the monotonicity for $\mathscr{V}_4$. By definition,

      \begin{equation*}
          \mathscr{V}_4=\frac{1}{2}(r_2+x_m)^{2}\tan\theta_2
      \end{equation*}

      \noindent  From Theorem \ref{lem:pos}, we know that $x_m=r_1$. Using this fact in the expression of $\mathscr{V}_{4}$ above, we have

      \begin{equation*}
          \mathscr{V}_4=\frac{1}{2}(r_2+r_1)^{2}\tan\theta_2.
      \end{equation*}

      \noindent Since $r_1$ and $r_2$ are both positive and monotonically increasing functions of $u_{m}$, it follows that $\mathscr{V}_{4}$ is monotonically increasing .

       For $\mathscr{V}_5$, we have

      \begin{equation*}
          \mathscr{V}_5=x_m(u_1-u_2)=r_1(u_1-u_2).
      \end{equation*}

      \noindent Since $r_1$ and $u_1-u_2$ are both positive and monotone increasing functions of $u_{m}$, $\mathscr{V}_{5}$ is monotonically increasing in $u_{m}$. This completes the proof of the theorem.
      
  \end{proof}

  \subsection{The main result of Section 2} In this subsection, we show the main results derived in Section 2.2.
  
  Combining the result of Theorem 2.11 and Theorem 2.13, we have the following theorem:

  \begin{theorem}
      Suppose $\theta_1,\theta_2\in(0,\pi)$, $\theta_1\in (0,\pi-\theta_2)$, and $[\![\gamma]\!]\leq 0$. We conclude that $\mathscr{V}(u_m)$ is a monotone increasing function of $u_{m}$. 
  \end{theorem}
  
  Finally, we have the following theorem:

 \begin{theorem}

    For any prescribed constant $V$ and fixed inclined angles $\theta_1$ and $\theta_2$ such that $\theta_{1}\in(0,\pi-\theta_{2})$, there exist a $u_m$ and at least one corresponding solution to the equation system \eqref{equ:1.2.600} such that 

    \begin{equation*}
        \mathscr{V}(u_{m})=V
    \end{equation*}
    
    Furthermore, $\mathscr{V}(u_{m})$ is a monotonically increasing function with respect to $u_{m}$ when $[\![\gamma]\!]<0$. In this case, the choice of $u_{m}$ and the corresponding solution function are unique.
 \end{theorem}

  When $\psi_1<0$ and $\psi_2>0$, the function curve reaches a minimum at a point $(x_m,u_m)$. Then we use a similar computation as in the previous case; we can also prove the existence. However, the uniqueness is still unknown in this case.
  
  Now we have a solution to equation \eqref{equ:1.2.600}, and the volume of the region enclosed by this curve and two boundary walls equals $V$. Finally, we aim to show that this curve can be expressed as a graph of a function in polar coordinates.

  \begin{theorem}
      Suppose that $\theta_1,\theta_2\in(0,\pi)$, and $\theta_1\in (0,\pi-\theta_2)$. For any prescribed positive number $V$.  The boundary curve derived in Theorem 2.15 can be be expressed as a graph of a function $\rho(\theta)$ in polar coordinates. Furthermore, this function $\rho(\theta)$ satisfies the following property:
      \begin{align}
          \mathcal{V}(\rho)=\mathscr{V}({u_{m}})=V
      \end{align}
  \end{theorem}

  \begin{proof}

  We prove this theorem by a contradiction argument. Assume that there is a $\theta_{0}\in (\theta_{2},\pi-\theta_{1})$ such that the line $y=(\tan\theta_{0}) x$ intersects the boundary curve at more than one point.  Then there exist $\psi_{A}$ and $\psi_{B}$ satisfying the following three properties
  \begin{itemize}
      \item $\psi_{2}<\psi_{B}<\psi_{A}<\psi_{1}$
      \item the boundary curve $(x(\psi),y(\psi))$ intersect $y=(\tan\theta_{0}) x$ when $\psi=\psi_{A}$ or $\psi=\psi_{B}$.
      \item the boundary curve $(x(\psi),y(\psi))$ does not intersect $y=(\tan\theta_{0}) x$ when $\psi\in (\psi_{A},\psi_{1})\cup (\psi_{B},\psi_{A})$. 
  \end{itemize} 
  The first and second properties above directly come from our assumption that there are more than two intersection points. We now assume that there exists a $\psi_{B}$ such that the third property holds and proceed to the main theorem under this assumption first.
  
  \textbf{The proof of the main theorem}

  Using the three properties above, it follows directly from geometric considerations that
  \begin{align}
      \psi_{A}\leq\theta_{0}~\operatorname{and}~\psi_{B}\leq\theta_{0}-\pi
  \end{align}
  \noindent Therefore, we have
  \begin{align}{\label{equ:left}}
      -\pi<\psi_{2}<\psi<\psi_{B}\leq\theta_{0}-\pi<0
  \end{align}
  \noindent for all $\psi\in (\psi_{2},\psi_{B})$. Then for any such $\psi$, we have
  \begin{align}{\label{equ:mix}}
      \frac{x(\psi)}{y(\psi)}=\frac{x(\psi_{B})-\int_{\psi}^{\psi_{B}}\cot\psi dy}{y(\psi_{B})-\int_{\psi}^{\psi_{B}}dy}
  \end{align}
  \noindent From equation \eqref{equ:left}, we have $\cot\psi>\cot(\theta_{0}-\pi)=\cot\theta_{0}$. Using this result together with $y(\psi_{B})-\int_{\psi}^{\psi_{B}}dy>0$, we deduce from equation \eqref{equ:mix} that
  \begin{align}{\label{equ:cot}}
      \frac{x(\psi)}{y(\psi)}=\frac{x(\psi_{B})-\int_{\psi}^{\psi_{B}}\cot\psi dy}{y(\psi_{B})-\int_{\psi}^{\psi_{B}}dy}<\frac{x(\psi_{B})-\cot\psi_{B}\int_{\psi}^{\psi_{B}} dy}{y(\psi_{B})-\int_{\psi}^{\psi_{B}}dy}=\cot\psi_{B}
  \end{align}
  \noindent for all $\psi\in(\psi_{2},\psi_{B}]$. 
  
  However, since $(x_{2},u_{2})=(x(\psi_{2}),y(\psi_{2}))$, we have
  \begin{align}
      \frac{x(\psi_{2})}{y(\psi_{2})}=\cot \theta_{2}>\cot \theta_{0}, 
  \end{align}
  \noindent which contradicts \eqref{equ:cot}. Therefore, the assumption of the existence of such an angle $\theta_{0}$ is false.

  \textbf{The proof of three properties} It remains to show that there exists a $\psi_{B}$ satisfying the three properties stated above. 

  We again proceed by contradiction. Suppose the third property does not hold. Then there exists a sequence $\psi_{i}$ such that $\psi_{i}<\psi_{i+1}<\psi_{A}$, $\psi_{n}\rightarrow \psi_{A}$ when $n\rightarrow +\infty$, and the straight line $y=(\tan\theta_{0})x$ intersects the boundary curve at $(x(\psi_{n}),y(\psi_{n}))$. 
  
  By the mean value theory, there exists a sequence of angles $\tilde{\psi}_{i}$ such that $\tilde{\psi}_{i}\in(\psi_{i-1},\psi_{i})$, and
  \begin{align}{\label{equ:seq_theta}}
  \tilde{\psi}_{n}=\theta_{0}+k\pi~~~~\operatorname{for~some~integer~k}
  \end{align}
  From the definition of $\psi$, we have
  \begin{align}
      \psi\in(\psi_{2},\psi_{1})\subseteq (-\pi,\pi)
  \end{align}
  \noindent Therefore, there can be at most two values of $\psi\in(\psi_{2},\psi_{1})$ satisfying \eqref{equ:seq_theta}, which contradicts the existence of such an infinite sequence $\{\psi_{i}\}$. Hence, a $\psi_{B}$ satisfying the three properties must exist.
  \end{proof}
  
    \end{subsection}
    \end{section}

    \begin{section}{The second case: $\psi_1\cdot \psi_2\geq 0$ }

    In this section, we discuss the case when $\psi_1$ and $\psi_2$ have the same sign. We begin by examining a special case before establishing the general theorem describing the shape of the surface curve. We then use arguments similar to those used in the previous case to derive the desired result.
    
    \begin{subsection}{A special case when $\theta_1=\frac{\pi}{2}$ and $\theta_2=0$}

    As in Section 2, we first identify several key geometric properties that the solution curve must satisfy. This is an \emph{a priori} discussion, in which we assume the existence of a $C^{2}$ curve satisfying the following four properties

    \begin{itemize}
    
    \item It intersects each of the inclined walls at exactly one point, with slope angles $\psi_{1}$ and $\psi_{2}$.
    
     \item It satisfies the equation \eqref{equ:1.2.7} locally. 
     
    \item It can be piece-wise parameterized in polar coordinates by $\theta$.

    \item This curve can not intersect itself.
    
    \end{itemize}
    Then we reparameterize the first equation of \eqref{equ:1.2.3} by $\psi$ as defined in \ref{def:slope}. We write it down as follows 
    \begin{equation}{\label{equ:bdd_n_1}}
        \begin{cases}
        \frac{dx}{d\psi}=\sigma\frac{\cos\psi}{gy}\\
    \frac{dy}{d\psi}=\sigma\frac{\sin\psi}{gy}\\
    \end{cases}
    \end{equation}
    \noindent We retain the assumptions stated above in this subsection.
    
    While the overall setting resembles that of Section 2, the analysis here is considerably more involved. We start by studying a special case, which is presented in the theorem below.

    \begin{theorem}
    
    Suppose that there exists a curve satisfying four properties above when $\theta_1=\frac{\pi}{2}$ and $\theta_2=0$. Suppose that  $[\![\gamma]\!]>0$, and the contact angle $\gamma\in(-\frac{\pi}{2},-\frac{\pi}{4})$. Furthermore, suppose curve can be expressed as a graph of function with respect to $x$ denoted by $v(x)$. Then the slope angle $\psi$ (defined by equation \eqref{def:slope}) is neither monotonically increasing nor monotonically decreasing of $x$ when the volume of the area enclosed by this curve and two inclined walls denoted as $V$ is sufficiently large.
     
    \end{theorem}

    \textbf{Remark} From the assumption made in this theorem, we have
    \begin{align}
        \psi_{1}=\gamma+\frac{\pi}{2}-\theta_{1}=\gamma\in (-\frac{\pi}{2},-\frac{\pi}{4}),
    \end{align}
    and
    \begin{align}
        \psi_{2}=-\gamma-\frac{\pi}{2}+\theta_{2}=-\gamma-\frac{\pi}{2}\in(-\frac{\pi}{4},0).
    \end{align}
    \noindent Hence, $\psi_{1}$ and $\psi_{2}$ have then same sign from the setting in this theorem.
    \begin{proof}
    
     From the assumption that the curve can be expressed as a graph of a function of $x$, we consider the problem in Cartesian coordinates, with the origin located at the intersection point of the two walls.  Reparameterizing \eqref{equ:bdd_n_1} with respect to $x$, we obtain the equation for the free surface:

     \begin{equation}{\label{equ:3.1.1}}
         gv(x)-\sigma \mathcal{H}=P_{0},
     \end{equation}

     \noindent where

     \begin{equation*}
         \mathcal{H}=\partial_{\theta}(\frac{\partial_{x}v(x)}{\sqrt{1+\vert \partial_{x}v\vert^{2}}})
     \end{equation*}

     \noindent denotes the mean curvature of the curve.
     
     Suppose, for the sake of contradiction, that $P_{0}\le 0$. Integrating both sides of \eqref{equ:3.1.1} over the interval $[0,x_2]$, we obtain

     \begin{equation*}
         gV+\sigma \sin(\psi_1)-\sigma \sin(\psi_2)=x_2P_{0}
     \end{equation*}

     \noindent Since $V$ is assumed to be sufficiently large, the left-hand side is strictly positive, whereas the right-hand side is non-positive under the assumption $P_0\le 0$. This contradiction implies that $P_{0}>0$.
     
     Since $v(x_2)=0$, we deduce from \eqref{equ:3.1.1} that

     \begin{equation*}
         -\sigma \partial_{x}(\sin\psi)|_{x=x_2}=P_{0}> 0,
     \end{equation*}

     \noindent which implies

     \begin{equation*}
         \partial_{x}(\sin\psi)|_{x=x_2}< 0.
     \end{equation*}

     \noindent Hence, if $\psi$ were a monotone function of $x$, it would have to be monotonically decreasing. However, this can not be true, since

     \begin{equation}{\label{equ:3.1.2}}
         \sin(\psi)|_{x=x_{1}}=\sin\psi_1=\sin(\gamma)<\sin(-\frac{\pi}{2}-\gamma)=\sin\psi_2=\sin(\gamma)|_{x=x_2},
     \end{equation}

    \noindent where we have used the definitions of $\psi_1$ and $\psi_2$ given in \eqref{equ:1.2.8} and \eqref{equ:1.2.9}, together with the assumption $\gamma<-\frac{\pi}{4}$, which implies
    \begin{align}
        -\sin(\gamma)<-\sin(\frac{\pi}{2}+\gamma).
    \end{align}
     From equation \eqref{equ:3.1.2}, we know that $\psi$ cannot be monotonically decreasing with respect to $x$. Therefore, when $V$ is large enough, $\psi$ is not monotonic in $x$, which implies that the boundary curve in this case can not be parameterized by $\psi$. 
    \end{proof}

    Theorem 3.1 implies that when $\psi_{1}$ and $\psi_{2}$ have the same sign, the boundary curve cannot, in general, be parameterized by $\psi$.  We now aim to establish a theorem for this special case, which may provide insight into the curve's shape in the general case.
    
    \begin{theorem}{The shape of the function }

Assume the same hypotheses as in Theorem 3.1. Then the following statements hold:
    
    1: If
    \begin{align}
        gV+\sigma \sin(\gamma)-\sigma \sin(-\frac{\pi}{2}-\gamma)\leq 0, 
    \end{align}
     then $\psi$ is a monotone increasing function of $x$.

    2:If

    \begin{align}
        gV+\sigma \sin(\gamma)-\sigma \sin(-\frac{\pi}{2}-\gamma)>0,
    \end{align}
      then $\psi$ attains its unique maximal value $\psi_m$ in the interior of the interval where $ \psi_m\in(\max(\gamma-\frac{\pi}{2},-\gamma),0)$. 
    \end{theorem}

    \begin{proof}

    \textbf{proof of} (1):  Integrating both sides of the equation \eqref{equ:3.1.1} from 0 to $x_2$, we obtain

    \begin{equation}{\label{equ:special}}
        gV+\sigma \sin(\frac{\pi}{2}-\gamma)-\sigma\sin(\gamma)=x_2P_{0}
    \end{equation}

    \noindent Using the condition $gV+\sigma \sin(\phi)-\sigma \sin(\frac{\pi}{2}-\phi)\leq 0$, we derive $P_{0}\leq 0$ from equation \eqref{equ:special}. Substituting this into \eqref{equ:3.1.1} gives
    
    \begin{equation*}
        \sigma \partial_{x} \sin\psi=gv(x)-P_{0}>0~\operatorname{for~any}~x.
    \end{equation*}

    \noindent Hence, $\psi$ is an increasing function of $x$, and therefore the curve can be parameterized by $\psi$.

    \begin{equation*}
        ~
    \end{equation*}
    
    \textbf{proof of} (2): From Theorem 3.1, $P_{0}> 0$ when  $gV+\sigma \sin(\phi)-\sigma \sin(\frac{\pi}{2}-\phi)>0$. In this case $\psi$ is neither monotonically decreasing nor increasing. We divide the main proof into several steps.

    \textbf{Step 1} In this step, we prove that the boundary curve in this case must attain its maximum $\psi_{m}$ at some point. Moreover, $\psi_m\leq0$. 
    
    To prove that the boundary curve attains its maximum at some point. It suffices to show the following property

    \begin{equation}{\label{key_assume}}
        \sigma \frac{d}{dx}\sin(\psi)|_{x=0}=gv(x_1)-P_{0}>0.
    \end{equation}

    \noindent We argue by contradiction and assume instead that

    \begin{equation}{\label{equ:3.1.3}}
        \sigma \frac{d}{dx}\sin(\psi)|_{x=0}=gv(x_1)-P_{0}\leq 0
    \end{equation}

     Differentiating both sides of equation \eqref{equ:3.1.1} with respect to $x$, we obtain

    \begin{equation}{\label{equ:3.1.4}}
        \sigma \frac{d^{2}}{dx^{2}}\sin(\psi)|_{x=0}=g\partial_{x}v|_{x=0}=g \tan\psi|_{x=0}=g \tan (\gamma)< 0.
    \end{equation}

    \noindent Consequently,

    \begin{equation*}
        \sigma \frac{d^{2}}{d^{2}x}\sin(\psi)<0,
    \end{equation*}

    \noindent in a neighborhood of $x=0$, which implies

    \begin{equation*}
        \sin\psi(x)<0.
    \end{equation*}
    
    \noindent for $x$ close enough to $0$. We now define

    \begin{equation*}
        x_{0}:=\min\{x:\frac{d^{2}}{dx^{2}}\sin\psi=0\}.
    \end{equation*}

    \noindent This point is well-defined since, by assumption, $\psi$ is not a monotone decreasing function of $x$. From the computation in equation \eqref{equ:3.1.4}, we have $\tan\psi=0$ when $x=x_0$. Hence, $\sin\psi(x_0)=0>\sin\psi_1$. Combining this result with the assumption \eqref{equ:3.1.3}, we deduce that there exits a point $x_{1}\in(x,x_{0})$ such that $\frac{d}{dx}(\sin\psi)|_{x=x_{1}}=0$. However, by the definition of $x_{0}$, we know that 
    \begin{align}
        \frac{d^{2}}{dx^{2}}(\sin\psi)<0
    \end{align}
    for all $0<x<x_0$. Therefore, $\frac{d}{dx}(\sin\psi)<0$ for all $x\in(0,x_0)$. This  contradicts the existence of such points $x_0$ and $x_{1}$. Therefore, assumption \eqref{equ:3.1.3} can not hold.
    
    We now show that the maximum angle $\psi_{m}\leq0$. From \eqref{key_assume}, we obtain the following property

    \begin{equation}{\label{equ:3.1.5}}
        \sigma \frac{d}{d\psi}\sin\psi|_{x=x_{1}}=gv(x_{1})-P_{0}>0
    \end{equation}

    \noindent Suppose, for the sake of contradiction, that the maximum value $\psi_m>0$. Then we have:

    \begin{equation*}
        \sigma \frac{d}{d\psi}\sin\psi|_{x=x_m}=gv(x_m)-P_{0}=0
    \end{equation*}

    \noindent From equation \eqref{equ:3.1.4}, we have

    \begin{equation}{\label{equ:3.1.6}}
        \sigma \frac{d^{2}}{d\psi^{2}}\sin(\psi)|_{x=x_m}=g\tan\psi_m>0
    \end{equation}

    \noindent However, since $\psi_m$ is assumed to be a maximum, we must have
    
    \begin{equation*}
        \sigma \frac{d^{2}}{d\psi^{2}}\sin(\psi)|_{\psi=\psi_m}=g\tan\psi_m\leq 0,
    \end{equation*}

    \noindent which contradicts \eqref{equ:3.1.6}. Therefore, the maximum for the inclined angle $\psi_{m}$ must satisfy $\psi_m\leq 0$.
    
     \textbf{Step 2}: In this step, we prove that $\psi_m\neq 0$. 
     
     We proceed by a contradiction argument again. Suppose that $\psi_m=0$. From the first and second equation of the system \eqref{equ:1.2.600}, we have

    \begin{equation}{\label{equ:3.1.7}}
        \frac{d\psi}{dx}=\frac{gv-P_{0}}{\sigma \cos\psi}
    \end{equation}

    \begin{equation}{\label{equ:3.1.8}}
        \frac{dy}{d\psi}=\frac{\sigma \sin\psi}{gv-P_{0}}
    \end{equation}

    \noindent Integrating both sides of the equation \eqref{equ:3.1.8} from $\psi$ to $\psi_m$, we obtain

    \begin{equation}{\label{equ:3.1.9}}
        y(\psi)=\frac{P_{0}\pm \sqrt{P_{0}^{2}-2g\sigma(\cos\psi-C)}}{g},
    \end{equation}

    \noindent where $C$ is a constant determined by condition $v(\psi_m)=\frac{P_{0}}{g}$.

    Plugging the equation \eqref{equ:3.1.9} into the ODE \eqref{equ:3.1.7}, we obtain

    \begin{equation}{\label{equ:3.1.110}}
         F(\psi):=\frac{d\psi}{dx}=\frac{\sqrt{P_{0}^{2}-2g\sigma(\cos\psi-C)}}{\sigma \cos\psi}.
    \end{equation}
    
    \noindent Differentiating function $F$ with respect to $\psi$ near $\psi_m$,  we obtain

    \begin{equation}{\label{equ:3.1.10}}
        \frac{dF(\psi)}{d\psi}=-\frac{\sin\psi\sqrt{P_{0}^{2}-2g\sigma(\cos \psi-C)}}{\sigma \cos^{2}\psi}+\frac{d\sqrt{P_{0}^{2}-2g\sigma(\cos \psi-C)}}{d\psi}\frac{1}{\sigma \cos \psi}.
    \end{equation}

    \noindent  The first term of equation \eqref{equ:3.1.10} is a smooth function that tends to zero as $\psi\rightarrow 0$, and is therefore bounded in a neighborhood of $\psi_m=0$. For the second term, we have
    
    \begin{equation}{\label{equ:3.1.11}}
        \frac{d\sqrt{P_{0}^{2}-2g\sigma(cos \psi-C)}}{d\psi}=-\frac{1}{2}\frac{2g\sigma \sin\psi}{\sqrt{P_{0}^{2}-2g\sigma(\cos \psi-C)}}.
    \end{equation}

     \noindent Performing a Taylor expansion of the right-hand side of equation \eqref{equ:3.1.11} near $\psi=0$, we obtain

    \begin{align}
        \frac{d\sqrt{P_{0}^{2}-2g\sigma(\cos \psi-C)}}{d\psi}&=-\frac{1}{2}\frac{2g\sigma \sin\psi}{\sqrt{P_{0}^{2}-2g\sigma(\cos \psi-C)}} \notag\\
        &=-g\sigma \frac{\psi+O(\psi^{2})}{\sqrt{2g\sigma}\sqrt{\frac{P_{0}^{2}}{2g\sigma}+C-\cos\psi}} \label{equ:3.1.12}.
    \end{align}

     From the definition of constant $C$, we have

    \begin{equation*}
        \frac{P_{0}^{2}}{2g\sigma}+C-\cos\psi_m=0
    \end{equation*}

    \noindent Hence, the denominator in \eqref{equ:3.1.12} can be expanded as

    \begin{equation*}
        \sqrt{\frac{P_{0}^{2}}{2g\sigma}+C-\cos\psi}=\sqrt{1-\cos\psi}=\frac{\sqrt{2}}{2}\psi+O(\psi^{2})
    \end{equation*}
    
    \noindent  Substituting it into the derivative term \eqref{equ:3.1.12}, we obtain

    \begin{align}
         \eqref{equ:3.1.12}&=-g\sigma \frac{\psi+O(\psi^{2})}{\sqrt{2g\psi}\sqrt{\frac{P_{0}^{2}+C}{2g\psi}-\cos\psi}} \notag\\
         &= -g\sigma \frac{\psi+O(\psi^{2})}{\sqrt{2g\sigma}\sqrt{1-\cos\psi}} \notag\\
         &=-g\sigma \frac{\psi+O(\psi^{2})}{\sqrt{g\sigma}\psi+O(\psi^{2})} \notag\\
          &=-\sqrt{g\sigma} (1+O(\psi)) \label{equ:3.1.13}
    \end{align}
    
      Thus, for the ODE \eqref{equ:3.1.110} written as follows

    \begin{equation*}
        \frac{d\psi}{dy}=\frac{\sqrt{P_{0}^{2}-2g\sigma(\cos\psi-C)}}{\sigma \cos\psi}=F(\psi),
    \end{equation*}

    \noindent we infer from \eqref{equ:3.1.10} and \eqref{equ:3.1.13} that $F(\psi)$ has a bounded first derivative and therefore Lipschitz continuous in a neighborhood of $\psi=\psi_m=0$. By the uniqueness theorem for ODEs,  $\psi=0$ then must be the unique solution of the ODE, which contradicts our assumption that $\psi_{m}$ is a nontrivial maximum.

    \textbf{Step 3} We prove the uniqueness of the maximum in this step. 
    
    From Step 1 and Step 2, $\psi$ reaches to a maximum $\psi_{m}<0$. Since
    \begin{align}
        \frac{d}{dx}\sin\psi=gv(x)-P_{0},
    \end{align}
    we have
    \begin{align}{\label{equ:negative}}
        \frac{d}{dx}\sin\psi|_{x=x_{m}}=gv(x_{m})-P_{0}=0.
    \end{align}
    \noindent Since $\psi_{m}<0$, the function $v(x)$ is monotonically decreasing in a neighborhood of $x_{3}$. Therefore, we have
    \begin{align}
        \frac{d}{dx}\sin\psi=gv(x)-P_{0}<0
    \end{align}
    for all $x\in (x_{m},x_{m}+\delta)$ for some small $\delta$. We can extend this argument inductively to conclude that the relation in \eqref{equ:negative} is true for all $x\in (x_{m},x_{4})$ where $x_{4}$ is a point such that $x_{4}>x_{m}$ and $\psi(x_{4})=\psi_{2}$. 
    
    If the curve $(x,v(x))$ does not touch the right wall at $x= x_{4}$, then there exists a point $x_{5}>x_{4}$ such that $\psi(x_{5})=-\frac{\pi}{2}$ which contradicts the assumption that the curve can be expressed as a function of $x$. Therefore, the point $(x_{4},y(x_{4}))=(x_{2},u_{2})$ is the contact point with the right wall. 

    From the discussion above, $\psi(x)$ is monotonically increasing when $x\in (x_{1},x_{m})$, and then monotonically decreasing when $x\in (x_{m},x_{2})$. This implies that the maximum $\psi_{m}=\psi(x_{m})$ is unique.
    
    \end{proof}
 
    \end{subsection}

    \begin{subsection}{The shape of the curve when $\theta_1\leq \frac{\pi}{2}$}

     After discussing the special case, we now focus on the general case. This subsection provides an a priori estimate.  We assume that $0\leq \theta_2\leq \theta_1\leq \frac{\pi}{2}$ and that there exists a $C^{2}$ curve satisfying four properties assumed in Section 3.1. From the discussion in the previous subsection, we may infer that $\psi$ can serve as a valid global parameter only when $V$ is small. This makes the analysis more straightforward in such cases, as we can apply a similar approach to the first case. Therefore, our primary focus should be on the case where $V$ is large, which presents additional analytical challenges.
     
    We begin by deriving a theorem concerning the geometric structure of the curve. Assume that $V$ is large enough such that the boundary curve can not be parameterized by $\psi$. Under this assumption,  we establish the following lemma and theorems

    \begin{lemma}

    Suppose that $0\leq \theta_2\leq \theta_1\leq \frac{\pi}{2}$ and that $\psi_1$ and $\psi_2$ have the same sign. Then $\psi_1<0$ and $[\![\gamma]\!]>0$.  
    \end{lemma}

    \textbf{Remark} We will retain these background assumptions on
$\theta_1,\theta_{2},\psi_{1}$ and $\psi_{2}$ throughout the remainder of this section.
    
    \begin{proof}

        First, we prove that $[\![\gamma]\!]>0$.  By definition, we have:

        \begin{equation*}
            \psi_1=\frac{\pi}{2}-\theta_1+\gamma,
        \end{equation*}

        \noindent and
        
        \begin{equation*}
            \psi_2=-\gamma-\frac{\pi}{2}+\theta_2.
        \end{equation*}

        \noindent Suppose, for the sake of contradiction, that $[\![\gamma]\!]\leq 0$. Then, by equation \eqref{equ:1.2.110}, we have $\gamma\in[0,\frac{\pi}{2})$. Since $\theta_1\in(0,\frac{\pi}{2})$, it follows that 
        \begin{align}
            \psi_1=\frac{\pi}{2}-\theta_1+\gamma\geq 0~~\operatorname{and}~~\psi_2=-\gamma-\frac{\pi}{2}+\theta_2<0.
        \end{align}
        
        \noindent Thus, $\psi_1$ and $\psi_2$ have opposite signs, contradicting our assumption. Therefore, $[\![\gamma]\!]>0$
        
        Next, we show that $\psi_{1}<0$. Suppose, to the contrary, that  $\psi_1=\frac{\pi}{2}-\theta_1+\gamma>0$, which implies $-\gamma<\frac{\pi}{2}-\theta_1$. Since $\theta_2\leq \theta_1\leq \frac{\pi}{2}$, we have 
        \begin{align}
            -\gamma<\frac{\pi}{2}-\theta_1\leq \frac{\pi}{2}-\theta_2,
        \end{align}
        which implies that  $\psi_2=-\gamma-\frac{\pi}{2}+\theta_2<0$. So $\psi_1$ and $\psi_2$ again have opposite signs, contradicting our assumption. Therefore, both $\psi_{1}$ and $\psi_{2}$ must be negative.
    \end{proof}

    \textbf{Remark}: We use equations \eqref{equ:1.2.8} and \eqref{equ:1.2.9} in Lemma 3.3 to show that $\psi_1=\frac{\pi}{2}-\theta_1+\gamma\in(-\frac{\pi}{2},0)$ and $\psi_2=-\gamma-\frac{\pi}{2}+\theta_2\in(-\frac{\pi}{2},0)$
    
   From Lemma 3.3, we now establish the following theorem describing the shape of the curve.

    \begin{theorem}

          Suppose that there exists a $C^{2}$ curve satisfying four properties assumed in this Section, and that the boundary curve can not be parameterized by $\psi$. Then $\psi(x)$  increases from $\psi_1$ to its unique negative maximum $\psi_m$. And then decreases from the maximal point to $\psi_{2}=\psi$. Moreover, the curve is a graph of $x$.
        
    \end{theorem}

    \begin{proof}

        Since $\psi_1\in(-\frac{\pi}{2},0)$ by Lemma 3.3, the curve can be locally expressed as the graph of a function $v(x)$ near contact point $(x_{1},y_{1})$. In this neighborhood, the equation \eqref{equ:1.2.7} can be rewritten as

        \begin{equation}{\label{equ:3.3.16}}
            \partial_{x}(\sin\psi)=gv-P_{0}
        \end{equation}

        \noindent Consider the contact point $(x_{1},y_{1})$. Suppose that $gy_1-P_{0}\leq 0$. Then $\sin\psi$ is a decreasing function in a small neighborhood of $(x_1,y_1)$. Since $\psi_1<0$, we can easily show $\psi<\psi_1<0$ when $x\in (x_1,x_1+\delta)$ for some small $\delta$.  Repeating the same discussion at the point $(x_1+\delta,v(x_1+\delta))$, we can show that $\psi(x)<\psi_1$ for any $x\in (x,x+2\delta)$. By induction, we obtain $\psi(x)<\psi_1$ for all $x\in (x_1,x_3)$, where $x_{3}$ is a point such that $\psi(x_3)=-\frac{\pi}{2}$.  This portion of the curve cannot intersect the right wall because of the assumption that the curve can not be parameterized by $\psi$.

         As in the proof of Theorem \ref{Thm:g_1}, we can extend the solution function to $\psi\in (-\frac{\pi}{2},-\frac{3}{2}\pi)$ and further to $\psi\in (-\frac{3}{2}\pi,-2\pi-\psi_1)$. Then using a similar discussion as in Theorem \ref{Thm:g_1}, one finds that the curve would either self-intersect or pass through the left wall due to symmetry, which leads to a contradiction.

        Hence, we must have $gy_1-P_{0}>0$, implying that $\psi$ increases to the maximal value $\psi_{m}$. Since $\psi_{m}$ is the maximum value, we obtain $gv(x_{m})=P_{0}$ from equation \eqref{equ:3.3.16}. Moreover, by Theorem 3.2, it holds that $\psi_{m}<0$ (Although Theorem 3.2 was proven in a special case, the proof also works for the general case). Examining equation \eqref{equ:3.3.16},  $\psi$ decreases when $gv(x)<P_{0}$, implying that $\psi$ decreases monotonically to $\psi_2$ as a function of $x$. The curve must intersect the right wall at this point; otherwise, an argument identical to that in the proof of Theorem \ref{Thm:g_1} leads to a contradiction.  This completes the proof. Moreover, using the Remark following Lemma 3.3,  we conclude that $\psi(x)\in(\min(\psi_1,\psi_2),0)\subseteq(-\frac{\pi}{2},0)$ along this curve, confirming that the curve is globally representable as the graph of a function of $x$.
        
    \end{proof}

    \begin{figure}
        \centering
        \includegraphics[width=0.92\linewidth=4]{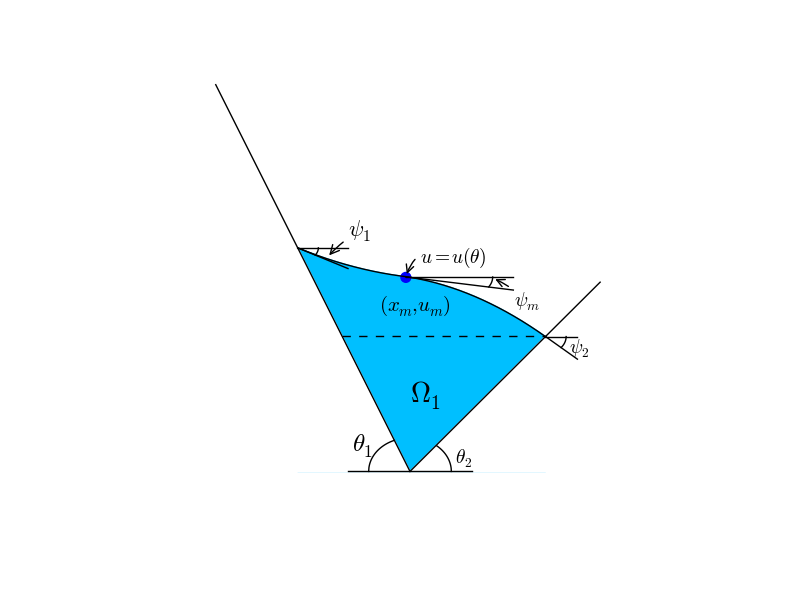}
        \caption{\label{figure7}}
     \end{figure}

     \subsection{The existence of the solution function}

     $~$
     
     \textbf{The idea}:
    
    From Section 3.2, we know that $\psi$ might not be a valid parameter for the boundary curve and that it attains its maximum value $\psi_{m}$ at a point $(x_m,v_m)=(x(\theta_m),v(\theta_{m}))$. After applying the shift $(x,y)\rightarrow (x,y-\frac{P_{0}}{g})$, we obtain $u_{m}=v_m-\frac{P_{0}}{g}=0$, implying that the coordinates of the special point is $(x_m,0)$ with the maximum angle $\psi_m$. 
    
   Since the curve can be described as a graph of a function $u(x)$, we construct the boundary curve via solving the first equation of \eqref{equ:1.3.3} with initial data $(x,y,\psi)=(x_{m},0,\psi_{m})$.  We then integrate both sides of the first equation in the system \eqref{equ:1.3.3} from $x_1$ to $x_2$ to show the following expression for $P_{0}$

    \begin{equation}{\label{equ:3.3.14}}
        P_{0}=\frac{g{\mathscr{V}+\frac{1}{2}gx_1^{2}\tan\theta_1+\frac{1}{2}gx_2^{2}\tan\theta_2-\sigma \sin(\psi_2)+\sigma \sin(\psi_1)}}{x_2-x_1}.
    \end{equation}

    \noindent where $\mathscr{V}$ denotes the total volume enclosed by the boundary curve and two walls. It remains to determine $P_{0}$ so that $\eqref{equ:3.3.14}$ shows the relation between $\mathscr{V}$ and $\psi_{m}$.

     By Theorem 3.4, we can construct the surface curve as two separated segments. Each segments are solutions to equation \eqref{equ:1.2.7} subject to initial condition $(x,y,\frac{dy}{dx}|_{x=x_{m}})=(x_{m},0,\psi_{m})$. The first extending from $x_1$ to $x_m$ and the second from $x_m$ to $x_2$. See the figure \ref{figure7}. The system \eqref{equ:1.3.3} provides a global description of the curve. By making the separation, we establish that each segment of the curve can be parameterized by $\psi$, which is defined by \eqref{def:slope}, allowing us to use equation \eqref{equ:1.2.7} for each part. The remainder of this chapter focuses on using $\psi_{m}$ to represent the unknown parameters $P_{0}$, $x_1$, and $x_2$. We first prove the following theorem in terms of the coordinates of the two contact points:

        \begin{theorem}

             Suppose the the curve constructed above intersects two inclined walls at $(x_{1},u_{1})$ and $(x_{2},u_{2})$ with corresponding slope angles $\psi_{1}$ and $\psi_{2}$. Then the following formulas for the coordinates of two contact points hold

            \begin{align}
                u_1&=\sqrt{\frac{2\sigma}{g}(\cos\psi_m-\cos\psi_1)} \label{equ:3.2.2}\\
                u_2&=-\sqrt{\frac{2\sigma}{g}(\cos\psi_m-\cos\psi_2)} \label{equ:3.2.3}\\
                x_1&=x_m-\int_{\psi_1}^{\psi_m}\frac{\sigma\cos\gamma}{g\sqrt{\frac{2\sigma}{g}(\cos\psi_m-\cos\gamma)}}d\gamma \label{equ:3.2.4}\\
                x_2&=x_m-\int_{\psi_m}^{\psi_2}\frac{\sigma\cos\gamma}{g\sqrt{\frac{2\sigma}{g}(\cos\psi_m-\cos\gamma)}}d\gamma \label{3.2.5}\\
                r(\psi)&=\vert x(\psi)-x_m\vert=\vert\int_{\psi_m}^{\psi}\frac{\sigma\cos\gamma}{g\sqrt{\frac{2\sigma}{g}(\cos\psi_m-\cos\gamma)}}d\gamma \label{3.2.100} \vert
            \end{align}

        \end{theorem}
        
        \begin{proof}
           Since each segment of the curve can be parameterized by $\psi$, the computation proceeds identically to that in Theorem \ref{thm:bdd_p}. We simply use the equation \eqref{equ:1.2.7} with the corresponding boundary condition to derive those relations above.
        \end{proof}

        Unlike the previous section, the curve is now separated into two segments, each potentially exhibiting distinct properties. We need to prove that the second segment does not intersect the left wall, which consequently implies that $x_2>0$. 

        \begin{lemma}

        For any $x\in(x_m,x_2)$, the curve $u(x)$ does not intersect the left wall.

        \end{lemma}

        \begin{proof}

             By construction, we know that the maximal point $(x_m,u_m)$ lies to the right of the left wall. Therefore, to prove the theorem, it suffices to prove that $\psi(x)\geq -\theta_1$ for all $x\in (x_m,x_2)$.

            From Theorem 3.4, we know that $\psi$ is a monotonically decreasing function of $x$. Therefore, it is enough to verify that $\psi_2\geq -\theta_1$. From the definition of $\psi_2$ given in \eqref{equ:1.2.9}, we  have
            \begin{equation}
                \psi_2=\gamma-\frac{\pi}{2}+\theta_2,~~\operatorname{and}~~\psi_1=\frac{\pi}{2}-\theta_1-\gamma.
            \end{equation}
             By Lemma 3.3, we know that both $\psi_1<0$ and $\psi_2<0$. This implies

            \begin{equation*}
                \frac{\pi}{2}-\theta_1<\gamma<\frac{\pi}{2}-\theta_2.
            \end{equation*}

            \noindent Consequently,

            \begin{equation*}
               \psi_{2}=\gamma-\frac{\pi}{2}+\theta_2> \frac{\pi}{2}-\theta_1-\frac{\pi}{2}+\theta_2=\theta_2-\theta_1\geq -\theta_1 
            \end{equation*}

            \noindent Thus, $\psi_{2}\geq -\theta_{1}$ and the theorem is proved.
        \end{proof}

        Now we have two parameters $x_{m}$ and $\psi_{m}$. Our next goal is to establish a relation between these two parameters. Following the same approach as in Section 2, we compute $x_m$ and obtain:

        \begin{equation}{\label{equ:3.2.6}}
            x_m=\frac{r_1\tan\theta_1-r_2\tan\theta_2-(u_1-u_{2})}{\tan\theta_1+\tan\theta_2}~\operatorname{when}~\theta_1\neq \frac{\pi}{2}
        \end{equation}

        \begin{equation}{\label{equ:3.2.7}}
            x_m=r_1~\operatorname{when}~\theta_1=\frac{\pi}{2}
        \end{equation}

        \noindent Here $r_1=x_m-x_1$ and $r_2=x_2-x_m$. For the special case $\theta_1=\frac{\pi}{2}$, the computation is straightforward and follows a similar approach  to Subsection 2.2. Hence, in what follows, we assume $\theta_1\neq \frac{\pi}{2}$. 
        
         We next compute $P_0$ by determining the $y$ component of the contact point $O$, which is the intersection point of two walls. 
        After performing the vertical shift $(x,y)\rightarrow (x,y-\frac{P_{0}}{g})$, the $y$ coordinate of the point $O$ becomes $y_{0}=-\frac{P_{0}}{g}$.
       On the other hand, using the same computation as in Theorem \ref{thm:relation}, we obtain 
       \begin{align}
       y_{0}=-((r_2+x_m)\tan\theta_2-u_2).
       \end{align}
        Combining these two expressions for $y_{0}$, we obtain 
        \begin{align}
        P_{0}=g((r_2+x_m)\tan\theta_2-u_2).
        \end{align}
        Substituting this relation for $P_{0}$ into equation \eqref{equ:3.3.14}, we obtain

        \begin{equation}{\label{equ:3.2.8}}
           g{\mathscr{V}}(\psi_m)=(x_2-x_1)g((r_2+x_m)\tan\theta_2-u_2)-\frac{1}{2}gx_1^{2}\tan\theta_1-\frac{1}{2}gx_2^{2}\tan\theta_2+\sigma \sin\psi_2-\sigma \sin\psi_1.
        \end{equation}

        Having established an explicit expression for the functional $\mathscr{V}$, we now aim to simplify its representation \eqref{equ:3.2.8}. To achieve this goal, we plug equation \eqref{equ:3.2.6} into \eqref{equ:3.2.8}, which leads to the following desired simplified form

        \begin{align}
            g{\mathscr{V}}(\psi_m)=&(x_2-x_1)g(\frac{r_2\tan\theta_1+r_1\tan\theta_1}{\tan\theta_1+\tan\theta_2}\tan\theta_2-(u_1-u_2)\frac{\tan\theta_2}{\tan\theta_1+\tan\theta_2}-u_2)\notag \\
            &-\frac{1}{2}gx_1^{2}\tan\theta_1-\frac{1}{2}gx_2^{2}\tan\theta_2+\sigma \sin\psi_2-\sigma \sin\psi_1 \notag\\
             =&g((r_2+r_1)\frac{\tan\theta_2\tan\theta_1}{\tan\theta_1+\tan\theta_2}-u_1\frac{\tan\theta_2}{\tan\theta_1+\tan\theta_2}-u_2\frac{\tan\theta_1}{\tan\theta_1+\tan\theta_2})(r_1+r_2) \notag\\
             &-\frac{1}{2}g(r_1-x_m)^{2}\tan\theta_1-\frac{1}{2}g(r_2+x_m)^{2}\tan\theta_2+\sigma \sin\psi_2-\sigma \sin\psi_1 \notag \\
             =&g((r_2+r_1)\frac{\tan\theta_2\tan\theta_1}{\tan\theta_1+\tan\theta_2}-u_1\frac{\tan\theta_2}{\tan\theta_1+\tan\theta_2}-u_2\frac{\tan\theta_1}{\tan\theta_1+\tan\theta_2})(r_1+r_2) \notag \\
             &-\frac{1}{2}g\frac{\tan\theta_1\tan\theta_2}{\tan\theta_1+\tan\theta_2}(r_1^2+r_2^2)-\frac{1}{2}g\frac{1}{(\tan\theta_1+\tan\theta_2)}(u_1-u_2)^{2}+\sigma \sin\psi_2-\sigma \sin\psi_1 \notag\\
              =&g(\frac{1}{2}(r_2+r_1)\frac{\tan\theta_2\tan\theta_1}{\tan\theta_1+\tan\theta_2}-u_1\frac{\tan\theta_2}{\tan\theta_1+\tan\theta_2}-u_2\frac{\tan\theta_1}{\tan\theta_1+\tan\theta_2})(r_1+r_2)\notag \\
               &+\frac{1}{2}g(r_1+r_2)^{2}\frac{\tan\theta_1\tan\theta_2}{\tan\theta_1+\tan\theta_2}-\frac{1}{2}g\frac{\tan\theta_1\tan\theta_2}{\tan\theta_1+\tan\theta_2}(r_1^2+r_2^2)\notag \\
                &-\frac{1}{2}g\frac{1}{(\tan\theta_1+\tan\theta_2)}(u_1-u_2)^{2}+\sigma \sin\psi_2-\sigma \sin\psi_1 \notag\\
                 =&g(\frac{1}{2}(r_2+r_1)\frac{\tan\theta_2\tan\theta_1}{\tan\theta_1+\tan\theta_2}-u_1\frac{\tan\theta_2}{\tan\theta_1+\tan\theta_2}-u_2\frac{\tan\theta_1}{\tan\theta_1+\tan\theta_2})(r_1+r_2) \notag \\
                 &+gr_1r_2\frac{\tan\theta_1\tan\theta_2}{\tan\theta_1+\tan\theta_2}-\frac{1}{2}g\frac{1}{(\tan\theta_1+\tan\theta_2)}(u_1-u_2)^{2}+\sigma \sin\psi_2-\sigma \sin\psi_1 \label{equ:3.2.200}.
        \end{align}

        \noindent We use $f(\psi)$ to denote right-hand side of \eqref{equ:3.2.200}. Then we establish the following theorem to discuss the asymptotic behavior of $f(\psi)$.

        \begin{theorem}

        If we regard $r_1$ and $r_2$ as functions of $\psi_m$, then the function $f$ defined below diverges to $+\infty$ as $\psi_m\rightarrow 0$.

        \begin{align}
             f(\psi_{m})=&g(\frac{1}{2}(r_2+r_1)\frac{\tan\theta_2\tan\theta_1}{\tan\theta_1+\tan\theta_2}-u_1\frac{\tan\theta_2}{\tan\theta_1+\tan\theta_2}-u_2\frac{\tan\theta_1}{\tan\theta_1+\tan\theta_2})(r_1+r_2) \notag \\
              &+gr_1r_2\frac{\tan\theta_1\tan\theta_2}{\tan\theta_1+\tan\theta_2}-\frac{1}{2}g\frac{1}{(\tan\theta_1+\tan\theta_2)}(u_1-u_2)^{2}+\sigma \sin\psi_2-\sigma \sin\psi_1 \label{equ:3.2.9}
        \end{align}
        \end{theorem}

        \begin{proof}

            \textbf{Step 1} In this step, we show that $r_1$ and $r_2$ diverge  to $+\infty$ as $\psi_m\rightarrow 0$.
            
           From equation \eqref{equ:3.2.4}, we have

            \begin{equation*}
                r_1=\int_{\psi_1}^{\psi_m}\frac{\sigma\cos\gamma}{g\sqrt{\frac{2\sigma}{g}(\cos\psi_m-\cos\gamma)}}d\gamma
            \end{equation*}

            \noindent When  $\psi_m=0$, applying Taylor expansion of $\cos\gamma$ near $0$, we have

            \begin{equation*}
                \cos\gamma=1-\frac{1}{2}\gamma^{2}+O((\gamma)^{3}).
            \end{equation*}

           \noindent Substituting this expansion into the integral above yields

            \begin{equation}{\label{equ:3.2.10}}
                r_1(0)\geq \int_{\psi_1}^{0} \frac{\sqrt{\sigma}}{\sqrt{2\sigma g}}\frac{\cos\gamma}{-\gamma}d\gamma=+\infty.
            \end{equation}

            \noindent By a simple observation, $r_1$ is an increasing function of $\psi_m$ for $\psi_m\in(\psi_1,0)$. Therefore, by Fatou's lemma, we obtain

            \begin{equation*}
                \infty=\int_{\psi_1}^{0}\frac{\cos\gamma}{g\sqrt{\frac{2\sigma}{g}(1-\cos\gamma)}}d\gamma\leq \lim_{\psi_{m}\rightarrow 0} \int_{\psi_1}^{\psi_m}\frac{\cos\gamma}{g\sqrt{\frac{2\sigma}{g}(\cos\psi_m-\cos\gamma)}}d\gamma
            \end{equation*}

           \noindent which implies

            \begin{equation*}
            \lim_{\psi_m\rightarrow 0} r_1=+\infty.
            \end{equation*}

            Using the same argument, we conclude that $r_2\rightarrow +\infty$ as $\psi_m\rightarrow 0$, which completes the proof of step 1

            \textbf{Step 2}  In this step, we examine the behavior of $u_1$ and $u_2$ as functions of $\psi_m$. 
            
            Using \eqref{equ:3.2.2} and \eqref{equ:3.2.3}, we obtain
            \begin{align}
                u_1\rightarrow \sqrt{\frac{2\sigma}{g}(1-\cos\psi_1)}~~\operatorname{and}~~u_2\rightarrow \sqrt{\frac{2\sigma}{g}(1-\cos\psi_2)}
            \end{align}
            as $\psi_m\rightarrow 0$. Both $u_1$ and $u_2$ remain uniform bounded for all values of $\psi_m$.

            \textbf{Step 3} In this step, we examine the asymptotic behavior of function $f$ as $\theta_2\neq 0$. 
            
            We begin by analyzing equation \eqref{equ:3.2.9}. Consider first the contribution from the first row of \eqref{equ:3.2.9}. We have

            \begin{equation*}
               g(\frac{1}{2}(r_2+r_1)\frac{\tan\theta_2\tan\theta_1}{\tan\theta_1+\tan\theta_2}-u_1\frac{\tan\theta_2}{\tan\theta_1+\tan\theta_2}-u_2\frac{\tan\theta_1}{\tan\theta_1+\tan\theta_2})(r_1+r_2)\rightarrow +\infty
            \end{equation*}

            \noindent as $\psi_m\rightarrow 0$. Indeed, by Steps 1 and 2,

            \begin{equation*}
                \frac{1}{2}(r_2+r_1)\frac{\tan\theta_2\tan\theta_1}{\tan\theta_1+\tan\theta_2}-u_1\frac{\tan\theta_2}{\tan\theta_1+\tan\theta_2}-u_2\frac{\tan\theta_1}{\tan\theta_1+\tan\theta_2}\rightarrow +\infty,
            \end{equation*}
            
            \noindent and moreover

            \begin{equation*}
            r_2+r_1\rightarrow +\infty
            \end{equation*}
            by step 1. 

             Next, we consider the second row  of \eqref{equ:3.2.9}. We obtain

            \begin{equation}{\label{equ:3.2.12}}
                 gr_1r_2\frac{\tan\theta_1\tan\theta_2}{\tan\theta_1+\tan\theta_2}-\frac{1}{2}g\frac{1}{(\tan\theta_1+\tan\theta_2)}(u_1-u_2)^{2}+\sigma \sin\psi_2-\sigma \sin\psi_1\rightarrow +\infty.
            \end{equation}

            \noindent This follows since the first term in equation \eqref{equ:3.2.12} diverges to $\infty$ as $\psi_m\rightarrow 0$, while the remaining terms are uniformly bounded for all $\psi_m\in (\psi_1,0)$.
            
            In conclusion, we deduce that $f(\psi_m)\rightarrow +\infty$ as $\psi_m\rightarrow 0$

            \textbf{Step 4} In this step, we discuss the asymptotic behavior of $f$ when $\theta_2=0$.

            When $\theta_{2}=0$, equation \eqref{equ:3.2.9} reduces to 

            \begin{equation}{\label{equ:3.2.1000}}
                f(\psi_m)=-u_2g(r_1+r_2)-\frac{1}{2}g\frac{1}{\tan\theta_1}(u_1-u_2)^{2}+\sigma \sin\psi_2-\sigma \sin\psi_1
            \end{equation}

            \noindent From Theorem 3.4, we know that $u_2<0$. Therefore, the first term on the right hand side of \eqref{equ:3.2.1000} is positive and dominated others as $\psi_{m}\rightarrow 0$. Therefore, $f(\psi_m)\rightarrow +\infty$ as $\psi_{m}\rightarrow 0$ by results in Step 1 and Step 2.
            
        \end{proof}

         From Theorem 3.7, we know that as $\psi_m\rightarrow 0$, the total volume $\mathscr{V}(\psi_m)$ enclosed by the boundary curve and two walls diverges to $+\infty$. We now turn to the lower bound of the total volume and establish the following theorem.

        \begin{figure}
        \centering
        \includegraphics[width=0.92\linewidth=2]{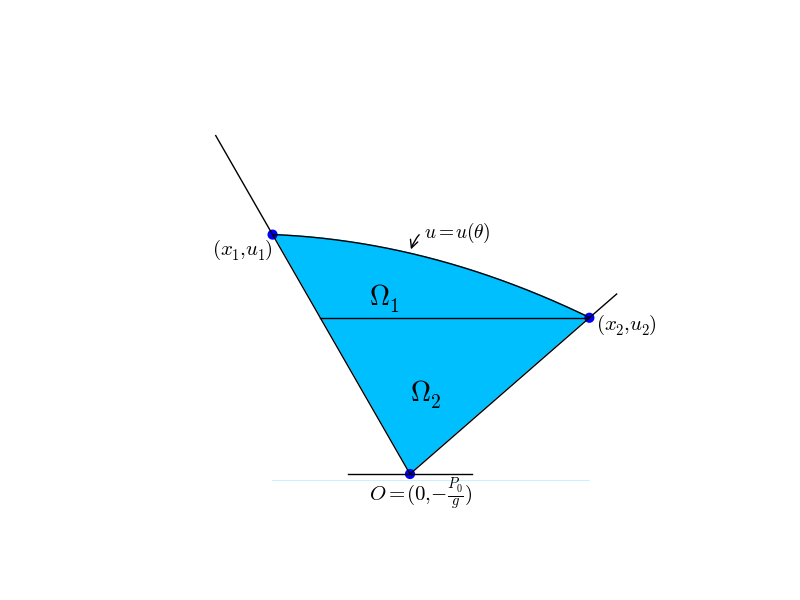}
        \caption{\label{figure8}}
        \end{figure}
     
        \begin{theorem}

        The total volume ${\mathscr{V}}(\psi_m)$, given by equation \eqref{equ:3.2.200}, admits a lower bound $V_m$. More precisely, there exists a minimum volume $V_{m}$ given by

        \begin{equation*}
             V_{m}=\min_{\psi_m \in(\max(\psi_1,\psi_2),0)}\mathscr{V}(\psi_m)
        \end{equation*}

        \noindent The quantity $V_m$ depends on $\psi_1$, $\psi_2$, $\theta_1$, and $\theta_2$. Moreover, $V_m>0$ whenever $\psi_1\neq \psi_2$.
        \end{theorem}

        \begin{proof}
        
        Just as what we did in Section 2, we separate the area into several parts. Let $\Omega_{1}$ denote the area enclosed by $y=u_2$ and two walls, and let the corresponding volume function be ${\mathscr{V}}_1(\psi_m)$. Let $\Omega_{2}$ denote the area $\Omega/\Omega_{1}$, and let the corresponding volume function be ${\mathscr{V}}_2(\psi_m)$.  Then the total volume satisfies ${\mathscr{V}}(\psi_m)\geq {\mathscr{V}}_2(\psi_m)$. See the figure \ref{figure7}.

        We now proceed to compute $\mathscr{V}_{2}$. Suppose that $\psi_{1}\neq \psi_{2}$; without loss of generality, we assume that $\psi_{1}>\psi_{2}$. Hence, by a direct geometric observation, we have

        \begin{equation}{\label{equ:V_2}}
        \begin{aligned}
            \mathscr{V}_{2}(\psi)\geq& \int_{\psi_{2}}^{\psi_{m}}(\int_{\psi}^{\psi_{m}}\frac{\sigma\cos\gamma}{g\sqrt{\frac{2\sigma}{g}(\cos\psi_m-\cos\gamma)}}d\gamma -\frac{1}{\tan\theta_{1}}(\sqrt{\frac{2\sigma}{g}(\cos\psi_{m}-\cos\psi)}))\frac{\sigma}{g}\frac{-\sin\psi}{\sqrt{\frac{2\sigma}{g}(\cos\psi_{m}-\cos\psi)}}d\psi\\
            \geq& \int_{\psi_{2}}^{\psi_{1}}(\int_{\psi}^{\psi_{1}}\frac{\sigma\cos\gamma}{g\sqrt{\frac{2\sigma}{g}(\cos\psi_m-\cos\gamma)}}d\gamma-\frac{1}{\tan\theta_{1}}(\sqrt{\frac{2\sigma}{g}(\cos\psi_{m}-\cos\psi)}))\frac{\sigma}{g}\frac{-\sin\psi}{\sqrt{\frac{2\sigma}{g}(1-\cos\psi)}}d\psi
        \end{aligned}
        \end{equation}

        \noindent Note that $\frac{\sigma}{g}\frac{-\sin\psi}{\sqrt{\frac{2\sigma}{g}(1-\cos\psi)}}>0$ for all $\psi\in(\psi_{2},\psi_{m})$ and is independent of $\psi_{m}$. Therefore, it suffices to analyze the following integrant
        \begin{align}
            (\int_{\psi}^{\psi_{1}}\frac{\sigma\cos\gamma}{g\sqrt{\frac{2\sigma}{g}(\cos\psi_m-\cos\gamma)}}d\gamma -\frac{1}{\tan\theta_{1}}(\sqrt{\frac{2\sigma}{g}(\cos\psi_{m}-\cos\psi)}))
        \end{align}
        \noindent We use following computation to handle it
        \begin{align}{\label{equ:key_ineq_1}}
            &(\int_{\psi}^{\psi_{1}}\frac{\sigma\cos\gamma}{g\sqrt{\frac{2\sigma}{g}(\cos\psi_m-\cos\gamma)}}d\gamma -\frac{1}{\tan\theta_{1}}(\sqrt{\frac{2\sigma}{g}(\cos\psi_{m}-\cos\psi)}))\notag\\
            &=(\int_{\psi}^{\psi_{1}}\frac{\sigma\cos\gamma}{g\sqrt{\frac{2\sigma}{g}(\cos\psi_m-\cos\gamma)}}d\gamma +\frac{1}{\tan\theta_{1}}\int_{\psi}^{\psi_{1}}\frac{\sigma \sin \gamma}{g\sqrt{\frac{2\sigma}{g}(\cos\psi_{m}-\cos\gamma)}} d\gamma)\notag\\
            &\geq (\int_{\psi}^{\psi_{1}}\frac{\sigma\cos\gamma}{g\sqrt{\frac{2\sigma}{g}(\cos\psi_m-\cos\gamma)}}d\gamma +\int_{\psi}^{\psi_{1}}\frac{\sigma \frac{\tan\gamma}{\tan\theta_{1}} \cos \gamma}{g\sqrt{\frac{2\sigma}{g}(\cos\psi_{m}-\cos\gamma)}} d\gamma)
        \end{align}
        \noindent Since,
        \begin{align}
            \frac{\tan\gamma}{\tan\theta_{1}}\geq \frac{\tan\psi_{1}}{\tan\theta_{1}}>p>-1
        \end{align}
        \noindent for some constant $p>-1$ uniformly for all $\psi_{m}\in(\psi_{1},0)$. Therefore, using this result in equation \eqref{equ:key_ineq_1}, we have
        \begin{align}
            &(\int_{\psi}^{\psi_{1}}\frac{\sigma\cos\gamma}{g\sqrt{\frac{2\sigma}{g}(\cos\psi_m-\cos\gamma)}}d\gamma -\frac{1}{\tan\theta_{1}}(\sqrt{\frac{2\sigma}{g}(\cos\psi_{m}-\cos\psi)}))\notag\\
            &\geq ((1+p)\int_{\psi}^{\psi_{1}}\frac{\sigma\cos\gamma}{g\sqrt{\frac{2\sigma}{g}(1-\cos\gamma)}}d\gamma ).
        \end{align}
        \noindent Using this estimate in equation \eqref{equ:V_2}, we obtain
        \begin{align}
            \mathscr{V}_{2}\geq (1+p)\int_{\psi_{2}}^{\psi_{1}}(\int_{\psi}^{\psi_{1}}\frac{\sigma\cos\gamma}{g\sqrt{\frac{2\sigma}{g}(1-\cos\gamma)}}d\gamma))\frac{\sigma}{g}\frac{-\sin\psi}{\sqrt{\frac{2\sigma}{g}(1-\cos\psi)}}d\psi.
        \end{align}
        The right hand side of inequality is uniformly bounded from below for any $\psi_{m}\in (\psi_{1},0)$. (It is independent of $\psi_{m}$). Hence, the result for then case when $\psi_{1}\neq \psi_{2}$ is proved.

        When $\psi_{1}=\psi_{2}$, from equation \eqref{equ:3.2.4} and \eqref{3.2.5}, we have $x_{1}\rightarrow x_{m}$ and $x_{2}\rightarrow x_{m}$ as $\psi_{m}\rightarrow \psi_{1}$. It implies that $x_{1}\rightarrow 0$ as $\psi_{m}\rightarrow \psi_{1}$, and $x_{2}\rightarrow 0$ as $\psi_{m}\rightarrow \psi_{1}$. Combining these two limits, we conclude that $\mathscr{V}(\psi_{m})\rightarrow 0$ as $\psi_{m}\rightarrow \psi_{1}$. This completes the proof.
        \end{proof}

       It therefore suffices to consider the case in which the surface curve can be parameterized by $\psi$. In the preceding analysis, we assumed that the maximal angle $\psi_m$ is attained in the interior of the domain. It remains to address the complementary situation in which $\psi(x)$ attains its maximum at the boundary.

       The subsequent analysis is analogous to that in Section~2. We construct a solution to equation \eqref{equ:1.3.4} with initial condition $(x_1,u_1)$, corresponding to the left contact point. Based on this construction, we then derive an explicit expression for the location of the other contact point.

        \begin{theorem}

         Let $(x_1,u_1)$ and $(x_2,u_2)$ be coordinates of two contact points. We have the following two results
        
        1: If $\psi_1\geq \psi_2$, $\psi(x)$ is a monotonically decreasing function of $x$. The corresponding expressions for $(x_{2},u_{2})$ are given by

        \begin{align}
             x_2&=x_1+\int_{\psi_2}^{\psi_1}\frac{\sigma\cos\gamma d\gamma}{g\sqrt{u_1^{2}+\frac{2\sigma}{g}(\cos\psi_1-\cos\psi_2)}}\label{equ:3.2.15}\\
             u_2&=\sqrt{u_1^{2}+\frac{2\sigma}{g}(\cos\psi_1-\cos\psi_2)}\label{equ:3.2.16}
        \end{align}

        2: If $\psi_2\geq \psi_1$, $\psi(x)$ is a monotonically increasing function of $x$. The corresponding expressions for $(x_{2},u_{2})$ are given by

        \begin{align}
            x_2&=x_1+\int_{\psi_1}^{\psi_2}\frac{\sigma\cos\gamma d\gamma}{g\sqrt{u_1^{2}+\frac{2\sigma}{g}(\cos\psi_1-\cos\psi_2)}} \label{equ:3.2.17}\\
            u_2&=\sqrt{u_1^{2}+\frac{2\sigma}{g}(\cos\psi_1-\cos\psi_2)}\label{equ:3.2.18}
        \end{align}

        \end{theorem}

        \begin{proof}

            The proof follows in the same manner as the proofs of Theorem 3.5 and Theorem \ref{thm:bdd_p} 
        \end{proof}

       For the remainder of the discussion, we assume that $\psi_2 \le \psi_1 < 0$.  
The case $\psi_1 < \psi_2$ can be treated analogously and is therefore omitted.  
See Figure~\ref{figure8} for an illustration of the geometry.

As in the derivation of \eqref{equ:3.2.6}, we obtain
\begin{equation}\label{equ:3.1.70}
    \tan\theta_1 (-x_1)=x_2\tan\theta_2+(u_1-u_2).
\end{equation}
Substituting $x_2=x_1+r_2$ into \eqref{equ:3.1.70} yields
\begin{equation}\label{equ:3.1.34}
    x_1=\frac{-r_2\tan\theta_2+(u_2-u_1)}{\tan\theta_1+\tan\theta_2}.
\end{equation}

Next, we express the total volume enclosed by the previously constructed curve in terms of $x_1$ and $u_1$.  
We decompose the total volume into two regions. The first region, denoted by $\Omega_1$, is the area enclosed by the boundary curve, the left wall, and the horizontal line $y=u_2$. The second region, denoted by $\Omega_2$, is the area bounded by the two walls and the line $y=u_2$; see Figure~\ref{figure8}.  

We denote the corresponding volumes of these regions by $\mathscr{V}_1$ and $\mathscr{V}_2$, respectively.

         Using the same computation as in Section 2, we obtain the following results

         \begin{align}
             {\mathscr{V}}_1(u_1)&=\mathscr{V}_1(x_1(u_1),u_1)\notag \\
             &=\int_{\psi_1}^{\psi_2}\int_{\psi}^{\psi_1}\frac{\sigma\cos\gamma}{g\sqrt{u_1^{2}+\frac{2\sigma}{g}(1-\cos\gamma)}}d\gamma \frac{\sigma\sin\psi}{g\sqrt{u_1^{2}+\frac{2\sigma}{g}(1-\cos\psi)}}d\psi-\frac{1}{2\tan\theta_1}(u_2-u_1)^{2} \label{equ:3.1.35}\\
             {\mathscr{V}}_2(u_1)&=\mathscr{V}_2(x_1(u_1),u_1)=\frac{1}{2}(\frac{1}{\tan\theta_1}+\frac{1}{\tan\theta_2})(u_2-(u_1+x_1\tan\theta_1)) \label{equ:3.1.36}
         \end{align}

        \noindent  Substituting equation \eqref{equ:3.1.34} into equation \eqref{equ:3.1.36}, we obtain

        \begin{equation*}
            {\mathscr{V}}_2=\frac{1}{2}\frac{1}{(\tan\theta_1+\tan\theta_2)}(\frac{1}{\tan\theta_1}+\frac{1}{\tan\theta_2})(r_2\tan\theta_2+(u_2-u_1)\tan\theta_2)
        \end{equation*}

         Based on the computation above,  the total volume ${\mathscr{V}}={\mathscr{V}}_1+{\mathscr{V}}_2$ can now be expressed in terms of $u_1$. Combining the ODE \eqref{equ:1.3.4}:

        \begin{equation*}
            \partial_{x}(\sin\psi)=gu
        \end{equation*}

        \noindent with the result in Theorem 3.9 which states $\psi(x)$ is monotonically decreasing, we obtain:

        \begin{equation*}
            gu_1=\partial_{x}(\sin\psi)(x_1)\leq 0,
        \end{equation*}

        \noindent which implies that

        \begin{equation*}
            u_1\leq 0.
        \end{equation*}

        \noindent Thus, we determine the admissible range of values for $u_1$
       and proceed to establish the following theorem.
        
        \begin{theorem}

            The total volume ${\mathscr{V}}(u_1)$ is a bounded function of $u_1$ when $u_1\leq 0$. Moreover, its upper bound $V_{1}$ satisfies $V_1\geq V_{m}$ if $\psi_{1}\neq \psi_{2}$.
        \end{theorem}
        \begin{proof}
        
            We first prove that $\mathscr{V}$ is a bounded function of $u_1$. We note that it is a continuous function of $u_1$. Furthermore, as $u_1\rightarrow -\infty$, equations \eqref{equ:3.2.15} and \eqref{equ:3.2.16} yield

            \begin{equation}{\label{equ:3.1.37}}
                r_2=x_2-x_1=\int_{\psi_2}^{\psi_1}\frac{\sigma\cos\gamma d\gamma}{g\sqrt{u_1^{2}+\frac{2\sigma}{g}(\cos\psi_1-\cos\psi_2)}}\rightarrow 0~\operatorname{when}~u_1\rightarrow -\infty,
            \end{equation}
            and
            \begin{equation}{\label{equ:3.1.38}}
                u_2-u_1\rightarrow 0~\operatorname{when}~u_1\rightarrow -\infty.
            \end{equation}

            \noindent Substituting \eqref{equ:3.1.37} and \eqref{equ:3.1.38} into \eqref{equ:3.1.34}, we obtain

            \begin{equation}{\label{equ:3.1.39}}
                x_1=\frac{-r_2\tan\theta_2+(u_2-u_1)}{\tan\theta_1+\tan\theta_2}\rightarrow 0~\operatorname{when}~u_1\rightarrow -\infty.
            \end{equation}

             Using the explicit expressions for $\mathscr{V}_1$ and $\mathscr{V}_2$ given in \eqref{equ:3.1.35} and \eqref{equ:3.1.36}, together with \eqref{equ:3.1.37},\eqref{equ:3.1.38} and \eqref{equ:3.1.39}, we obtain the following results
            
            \begin{equation*}
                \mathscr{V}_1\rightarrow 0~\operatorname{when}~u_1\rightarrow -\infty,
            \end{equation*}

            \noindent and

            \begin{equation*}
                 \mathscr{V}_2\rightarrow 0~\operatorname{when}~u_1\rightarrow -\infty.
            \end{equation*}

            \noindent Consequently,
\[
\mathscr{V}(u_1)=\mathscr{V}_1+\mathscr{V}_2 \;\longrightarrow\; 0
\qquad \text{as } u_1\to -\infty .
\]

It remains to examine the behavior at $u_1=0$. From the preceding discussion, the integral appearing in the representation \eqref{equ:3.1.35} diverges only when $\psi_1=0$. Since $\psi_1<0$ in the present setting, the volume functional $\mathscr{V}$ remains bounded at $u_1=0$. This proves that $\mathscr{V}$ is bounded on $(-\infty,0]$.

             Finally, we show that the upper bound $V_1$ satisfies $V_{1}\geq V_m$. When $u_1=0$, the corresponding curve coincides with that constructed in Theorem~3.8 with $\psi_m=\psi_1$. Therefore, by definition,

            \begin{equation*}
                V_1\geq \lim_{u_1\rightarrow 0} {\mathscr{V}}(u_1)={\mathscr{V}}(\psi_1)\geq V_m,
            \end{equation*}

            \noindent which finishes the proof.
        \end{proof}

        Combining Theorems~3.8 and~3.10, we obtain the following result.

        \begin{theorem}

            Assume that $\theta_1$ and $\theta_2$ are both acute angles. Then, for any prescribed total volume $V$, there exists at least one steady-state solution.

Moreover, there exist two positive constants $V_1$ and $V_m$, depending on $\theta_1$, $\theta_2$, $\sigma$, and $[\![\gamma]\!]$, such that the following statements hold:
\begin{itemize}
\item $V_{1}\geq V_{m}$
\item If $V \ge V_m$, there exists a value $\psi_m$ and a corresponding solution to system \eqref{equ:1.3.300} for which the slope angle $\psi$ attains its maximum in the interior interval $(\psi_2,\psi_1)$ (or $(\psi_1,\psi_2)$ if $\psi_1 \le \psi_2$).

\item If $V \le V_1$, there exists a value $u_1$ and a corresponding solution to system \eqref{equ:1.2.3} for which the slope angle $\psi$ attains its maximum on the boundary.

\item If $V_1 > V_m$, the steady-state solution is not unique.
            \end{itemize}

        \end{theorem}
        
        \end{subsection}

        \begin{subsection}{General case when $\theta_1>\frac{\pi}{2}$}

            In this case, the only difference from Subsection~3.2 is that $\psi_1$ and $\psi_2$ may have the same sign even when $[\![\gamma]\!]<0$. 

            By definition, $\psi_1=-\theta_1+\frac{\pi}{2}+\gamma$ and $\psi_2=\theta_2-\frac{\pi}{2}-\gamma$. It follows that $\psi_2<-\frac{\pi}{2}$ when $\gamma>\theta_2$, in which case the boundary curve may no longer be represented as the graph of a function with respect to $x$. Consequently, equation~\eqref{equ:3.3.14} can no longer be used to derive the equation for $P_0$. Instead, we must adopt the approach developed in Section~2, computing the total volume using the coordinates of a suitably chosen reference point. This constitutes the only essential difference between the present subsection and Subsection~3.2.
            
           In what follows, we only discuss the case $\gamma>\theta_2$. For $\gamma\leq \theta_2$, the analysis in Subsection 3.2 applies directly. We further assume that $\psi$ attains its maximum in the interior of the domain and that the curve cannot be parameterized by $\psi$. Specifically, we assume that $\psi(x)$ reaches its maximal value $\psi_m$ at the point $(x_m,0)$. We then construct the solution curve by separately constructing two curve segments, each satisfying equation~\eqref{equ:1.2.7} with initial data $(u_m, y_m=0, \psi_m)$.

             Suppose that the resulting curve intersects the two walls at contact points $(x_1,u_1)$ and $(x_2,u_2)$; see Figure~\ref{figure9}. The subsequent computations for these coordinates proceed in the same manner as those in Section~2, and are summarized as follows.

            \begin{figure}
            \centering
            \includegraphics[width=0.9\linewidth=2]{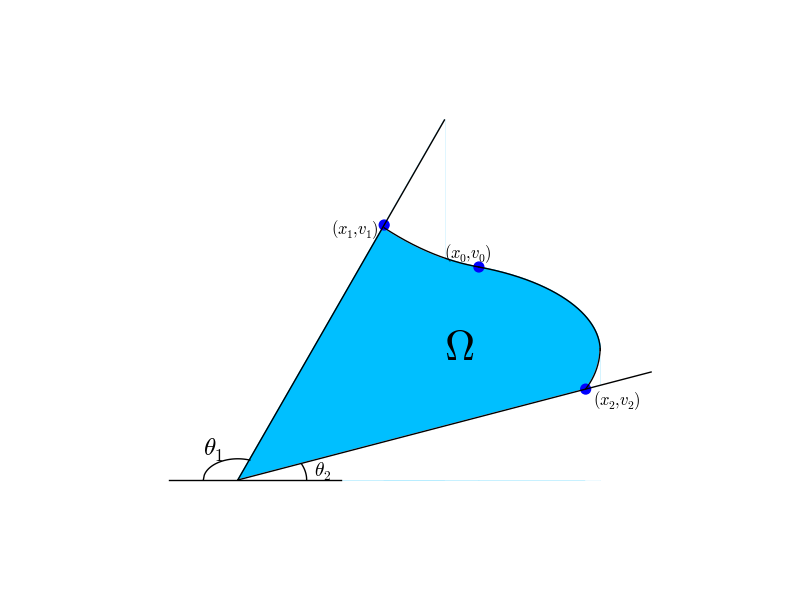}
           \caption{\label{figure9}}
           \end{figure}

           \begin{align}
               x_1&=x_m-\int_{\psi_1}^{\psi_m}\frac{\sigma\cos\gamma}{g\sqrt{\frac{2\sigma }{g}(\cos\psi_m-\cos\gamma)}}d\gamma   \label{equ:3.1.40}\\
               u_1&=\sqrt{\frac{2\sigma}{g}(\cos\psi_m-\cos\psi_1)} \label{equ:3.1.41}
               \end{align}

            \begin{align}
               x_2&=x_m+\int_{\psi_2}^{\psi_m}\frac{\sigma\cos\gamma}{g\sqrt{\frac{2\sigma }{g}(\cos\psi_m-\cos\gamma)}}d\gamma \label{equ:3.1.42}\\
               u_2&=-\sqrt{\frac{2\sigma}{g}(\cos\psi_m-\cos\psi_2)}\label{equ:3.1.43}
           \end{align}

            We employ the same approach as in the previous subsection to derive the relation between $x_m$ and $\psi_m$. In particular, we have

            \begin{equation*}
                -x_1\tan\theta_1-x_2\tan\theta_2=u_1-u_2,
            \end{equation*}

            \noindent which implies that
            
            \begin{equation}{\label{equ:3.1.44}}
                x_m=\frac{r_1\tan\theta_1-r_2\tan\theta_2-(u_1-u_2)}{\tan\theta_1+\tan\theta_2}.
            \end{equation}

            \noindent We next use $x_m$ and $\psi_m$ to express the total volume. A straightforward geometric computation yields

            \begin{align}
                 \mathscr{V}(\psi_{m})=&(r_1u_1-\int_{\psi_1}^{\psi_m}\int_{\psi}^{\psi_m}\frac{\sigma\cos\gamma}{g\sqrt{\frac{2\sigma}{g}(\cos\psi_m-\cos\gamma)}}d\gamma \frac{-\sigma\sin\psi}{g\sqrt{\frac{2\sigma}{g}(\cos\psi_m-\cos\psi)}}d\psi) \notag\\
                 &+\int_{\psi_2}^{\psi_m}\int_{\psi}^{\psi_m}\frac{\sigma\cos\gamma}{g\sqrt{\frac{2\sigma}{g}(\cos\psi_m-\cos\gamma)}}d\gamma \frac{-\sigma\sin\psi}{g\sqrt{\frac{2\sigma}{g}(\cos\psi_m-\cos\psi)}}d\psi \notag\\
                 &+\frac{1}{2}(u_1+u_2)^{2}\frac{1}{-\tan\theta_1}+(-u_2)r_1+\frac{1}{2}(u_2-u_1-x_1\tan\theta_1)^{2}(\frac{1}{\tan\theta_1}+\frac{1}{\tan\theta_2}) \label{equ:3.1.45}
            \end{align}

            We are now ready to establish the following theorem concerning the total volume.

            \begin{theorem}

                It holds that ${\mathscr{V}}(\psi_m)\rightarrow +\infty$ as $\psi_m\rightarrow 0$.
            \end{theorem}

            \begin{proof}

                 Using the computation in Theorem 3.7, we have

                \begin{equation*}
                    \int_{\psi}^{\psi_m} \frac{\cos\gamma}{g\sqrt{\frac{2\sigma}{g}(\cos\psi_m-\cos\gamma)}}d\gamma \rightarrow +\infty
                \end{equation*}

                \noindent as $\psi_m\rightarrow 0$ for any $\psi\in (\psi_2,\psi_m)$. Therefore, we have:

                \begin{equation*}
                \mathscr{V}(\psi_m)\geq \int_{\psi_2}^{\psi_m}\int_{\psi}^{\psi_m}\frac{\sigma\cos\gamma}{g\sqrt{\frac{2\sigma}{g}(\cos\psi_m-\cos\gamma)}}d\gamma \frac{-\sigma\sin\psi}{g\sqrt{\frac{2\sigma}{g}(\cos\psi_m-\cos\psi)}}d\psi\rightarrow +\infty
                \end{equation*}

                \noindent as $\psi_m\rightarrow 0$. Hence,

                \begin{equation*}
                \mathscr{V}(\psi_m)\rightarrow +\infty
                \end{equation*}

                \noindent as $\psi_m\rightarrow 0$, which finishes the proof.
            \end{proof}

             Just as what we did in the previous subsection, we want to show that ${\mathscr{V}}(\psi_m)$ has a lower bound $V_m$.

            \begin{theorem}

               The energy functional $\mathscr{V}(\psi_m)$ has a positive lower bound $V_m$ if $\psi_{1}\neq \psi_{2}$.
            \end{theorem}

            \begin{proof}

             Using equation \eqref{equ:3.1.45}, we have

            \begin{equation}{\label{equ:3.1.53}}
                 \mathscr{V}(\psi_m)\geq \int_{\psi_2}^{\psi_m}\int_{\psi}^{\psi_m}\frac{\sigma\cos\gamma}{g\sqrt{\frac{2\sigma}{g}(\cos\psi_m-\cos\gamma)}}d\gamma \frac{-\sigma\sin\psi}{g\sqrt{\frac{2\sigma}{g}(\cos\psi_m-\cos\psi)}}d\psi
            \end{equation}

            Using the definition of $\psi_1$ and $\psi_2$ given in \eqref{equ:1.2.8} and \eqref{equ:1.2.9}, we have $\psi_1=-\theta_1+\frac{\pi}{2}+\gamma$ and $\psi_2=\theta_2-\frac{\pi}{2}-\gamma$.  We first observe from these results that $\psi_1\in(-\frac{\pi}{2},0)$. Indeed,

            \begin{equation}{\label{equ:3.1.50}}
                0>\psi_1=-\theta_1+\frac{\pi}{2}+\gamma\geq -\theta_1+\frac{\pi}{2}+\theta_2>-\pi+\frac{\pi}{2}=-\frac{\pi}{2}
            \end{equation}
            
            \noindent Moreover, we have

            \begin{equation}{\label{equ:3.1.51}}
            \psi_m\geq \psi_1=-\theta_1+\frac{\pi}{2}-\gamma=\theta_2-\theta_1-\psi_2>-\pi-\psi_2
            \end{equation}
            
            From equation \eqref{equ:3.1.50} and equation \eqref{equ:3.1.51}, it follows that $\cos\psi_{m}>\cos\psi_2$. Consequently,

            \begin{equation}{\label{equ:3.1.52}}
                \int_{\psi}^{\psi_m}\frac{\cos\gamma}{g\sqrt{\frac{2\sigma}{g}(\cos\psi_m-\cos\gamma)}}d\gamma\geq \int_{\psi}^{\psi_m}\frac{\cos\gamma}{g\sqrt{\frac{2\sigma}{g}(\cos\psi_m+1)}}d\gamma =\frac{\sin\psi_m-\sin\psi}{g\sqrt{\frac{2\sigma}{g}(\cos\psi_m+1)}}
            \end{equation}

            \noindent The derivation of \eqref{equ:3.1.52} follows exactly the same argument as in Lemma~2.9 and is therefore omitted.
            
            Since $\psi\in (\psi_2,\psi_m)$, we have $\sin\psi_m>\sin\psi$.  Combining this result with equation \eqref{equ:3.1.52} and substituting them into equation \eqref{equ:3.1.53}, we obtain

            \begin{align*}
                \mathscr{V}(\psi_m)&\geq \sigma^{2}\int_{\psi_2}^{\psi_m}\frac{\sin\psi_{m}-\sin\psi}{g\sqrt{\frac{2\sigma}{g}(\cos\psi_m+1)}}\frac{-\sin\psi}{g\sqrt{\frac{2\sigma}{g}(\cos\psi_m-\cos\psi)}}d\psi\\
                 &>\frac{\sigma^{2}}{2\sigma g(\cos\psi_{m}+1)}\int_{\psi_2}^{\psi_{m}} -\sin\psi \sin\psi_{m}+\sin^{2}\psi d\psi\\
                 &>\frac{\sigma}{4 g}(-\sin\psi_{m}(\cos\psi_{m}-\sin\psi_2)+\int_{\psi_2}^{\psi_{m}}\sin^{2}\psi d\psi)\\
                 &>\frac{\sigma}{4 g}\int_{\psi_2}^{\psi_1}\sin^{2}\psi d\psi>C(\psi_1,\psi_{2})>0
            \end{align*}
            
            \noindent This estimate, together with Theorem 3.12, implies that ${\mathscr{V}}(\psi_m)$ is a positive continuous function for $\psi_m\in[\psi_1,0)$. Moreover, since ${\mathscr{V}}(\psi_{m})\rightarrow +\infty$ as $\psi_m\rightarrow 0$, $\mathscr{V}(\psi_m)$ is bounded from below by a positive constant $V_{m}$.
            \end{proof}

            It remains to consider the case in which the curve can be parameterized by $\psi$. We assume that $\psi_1>\psi_2$ and that the curve intersects the left wall at the contact point $(x_1,u_1)$. In this setting, the expressions for $x_2$ and $u_2$ in terms of $x_1$ and $u_1$ are given by \eqref{equ:3.2.15} and \eqref{equ:3.2.16}, respectively. Moreover, the representations of $x_1$ and of the total volume $\mathscr{V}(u_1)$ coincide with those in \eqref{equ:3.1.34}, \eqref{equ:3.1.35}, and \eqref{equ:3.1.36}. These observations lead to the following theorem.

            \begin{theorem}
            The total volume enclosed by the constructed curve ,denoted as $\mathscr{V}$, is a function of $u_1$, and is bounded from above when $u_1<0$, with a positive upper bound $V_a$. Furthermore, $V_a\geq V_m$, where $V_{m}$ is the constant derived in Theorem 3.13
            \end{theorem}

            \begin{proof}

                The proof follows directly from the argument used in Theorem 3.10. 
            \end{proof}

            Now we have constructed a boundary curve corresponding to a prescribed volume $V$, which represents a steady-state configuration. It remains to show that this curve can be represented as the graph of a function in polar coordinates.

            \begin{theorem}
                The surface function constructed in this section can be expressed as a graph of function $\rho(\theta)$ in polar coordinates.
            \end{theorem}

            \begin{proof}

                The main argument follows directly from that of Theorem~2.16. For brevity, the details are omitted.
            \end{proof}
            
            Finally, combining Theorem 3.11, Theorem 3.13, Theorem 3.14 and Theorem 3.15 together, we obtain the following existence result.

            \begin{theorem}
            
            When $\psi_1$ and $\psi_2$ have the same sign, there exist a solution function to the equation system \eqref{equ:1.2.3} for any prescribed volume $V$.
            \end{theorem}

        \end{subsection}

    \end{section}

    \begin{center}
        {\large A}CKNOWLEDGEMENTS
    \end{center}

    The author thanks his advisor Yan Guo for numerous comments. His mentorship and constructive feedback contribute significantly to the development of this work.

    This work is supported in part by NSF Grant DMS-2405051.

\bibliographystyle{plain}
\bibliography{sample}

\end{document}